\documentclass[12pt,leqno]{amsart}
\usepackage[dvips]{graphics}
\usepackage{amssymb}
\usepackage{verbatim}
\usepackage{hyperref}
\usepackage{color}

\numberwithin{equation}{section}
\newdimen\vintkern\vintkern12pt
\def\vint{-\kern-\vintkern\int}

\newtheorem{thm}{Theorem}[section]

\newtheorem{lem}[thm]{Lemma}
\newtheorem{cor}[thm]{Corollary}
\newtheorem{prop}[thm]{Proposition}

\newtheorem{quest}[thm]{Question}
\newtheorem{defn}[thm]{Definition}

\newtheorem{rem}[thm]{Remark}
\newtheorem{ex}[thm]{Example}

\newcommand{\tref}[1]{Theorem~\ref{#1}}
\newcommand{\cref}[1]{Corollary~\ref{#1}}

\newcommand{\R}{\mathbb{R}}

\newcommand{\Ss}{\mathbb{S}}

\newcommand{\diam}{\operatorname{diam}}

\newcommand{\blue}{\color{blue}}

\def\co{\colon\thinspace}

\begin{document}
	%\tableofcontents
	\pagebreak
	%\bibliographystyle{alpha}
	
	%\pagenumbering{roman}
	
	\title{ Structure of Submetries}
	
\thanks{V.K.  is partially supported by a Discovery grant from NSERC;
	A. L. was partially supported by the DFG grants   SFB TRR 191 and SPP 2026.}
	\author{Vitali Kapovitch}
	\address{University of Toronto}
\email{vkt@math.toronto.edu}
	
	\author{Alexander Lytchak}
	
	\address
	{Mathematisches Institut\\ Universit\"at K\"oln\\ Weyertal 86 -- 90\\ 50931 K\"oln, Germany}
	\email{alytchak@math.uni-koeln.de}


	\date{\today}
	
	%\thanks{The second author was partially supported by Swiss National Science Foundation Grant 153599}

	\begin{abstract}
	We  investigate the geometric and topological structure of equidistant decompositions of Riemannian manifolds.
	\end{abstract}
	
	\maketitle

	\renewcommand{\theequation}{\arabic{section}.\arabic{equation}}
	\pagenumbering{arabic}
	
	\section{Introduction}
\subsection{Subject of investigations}
An \emph{equidistant decomposition} $\mathcal F$ of a metric space $X$ is a decomposition of $X$ into a collection of  pairwise equidistant closed subsets $L_i, i\in I$, called the \emph{leaves} of the decomposition.
The \emph{space of leaves} $I$ of $\mathcal F$ can be equipped with a natural distance,
such that the canonical projection $P:X\to I$ is a \emph{submetry},
that is a map that sends metric balls in $X$ to metric balls in  $I$ of the same radius.
On the other hand, the fibers of any submetry $P:X\to Y$ provide an equidistant decomposition of the space $X$.  Basic examples of submetries are given by Riemannian submersions and quotient maps under proper isometric group actions.

Submetries were defined by V. Berestovskii  in \cite{Berest2}; in \cite{Berest} it was proved that a map $P:M\to N$ between  complete, smooth Riemannian manifolds is a submetry if and only if $P$ is a $\mathcal C^1$ Riemannian submersion.  Other large
  classical sources of equidistant decompositions  are provided by the decompositions into orbits of isometric group actions and \emph{singular Riemannian foliations} with closed leaves.  Singular Riemannian foliations, defined by P. Molino, \cite{Molino},
  include as subclasses many famous foliations in Riemannian geometry, like the isoparametric foliations, and   have been actively investigated recently from geometric, topological, analytic and algebraic points of view,  \cite{Thor}, \cite{LT},
	\cite{Rade},  \cite{Fer-Rad}, \cite{Al-RadI}, \cite{Mendes-Rad-slice},  \cite{Mendes-Rad}.
  %We  mention, that
	Submetries often appear in connections with rigidity phenomena, see  \cite{GG}, \cite{ Per-soul}, \cite{Ly}, \cite{Wilking}. { Conjecturally,
	collapsing of manifolds with lower curvature bounds  is modelled by submetries, see \cite{Yam}, \cite{CFG}, \cite{Kap-Guij} \cite{Kap-collapse}, \cite{Yam1}, \cite{Yam2}.}
  Recent appearance of submetries  in several completely unrelated settings, \cite{bipolar}, { \cite{GW}, \cite{Mondino},  \cite{Ber-finite},} \cite{Mendes-Rad},  { further motivates} a systematic study of the subject.

\subsection{Main results} The main objective of the present paper is the description of the structure of equidistant decompositions of  Riemannian manifolds. Equivalently, we describe the structure of
 possible spaces of leaves
$Y$ and of  submetries $P:M\to Y$ where $M$ is a   Riemannian manifold.
All Riemannian manifolds appearing in the paper are assumed to be sufficiently smooth,
in particular, they   have  \emph{local two sided curvature bounds}  in the sense of Alexandrov. A sufficient (and almost necessary) condition is that the Riemannian metric is  $\mathcal C^{1,1}$ in some  coordinates,  \cite{Ber-Nik}, \cite{KL}.
 Most results are local and  do not require  completeness of $M$. In fact they  are valid for \emph{local submetries}, see Subsection \ref{subsec: loc}.

  {The first theorem provides a characterization of possible leaves, see
  	Proposition \ref{lem: locsubm}  {and Remark \ref{rem: global}} for a local
  converse statement.}

\begin{thm} \label{thm: leaf}
Let $M$ be a  Riemannian manifold.
Any fiber $L$ of any  submetry $P:M\to Y$   is a set of positive reach in $M$.
%Either $L$ is nowhere dense in $M$ or $L$ is a connected component of $M$.
%
%On the other hand, if $L$ is a set positive reach, nowhere dense in $M$ then there exists a complete Riemannian metric on an open neighborhood $U$ of $L$ in $M$, and a submetry $P:U\to [0,\infty)$ with $L=P^{-1} (0)$.
\end{thm}

 Recall that a subset $L$ of
$M$ has \emph{positive reach}  if for some neighborhood $U$ of $L$ in $M$ and any $x\in U$ there exists a unique \emph {foot point} $\Pi ^L (x)\in L$ closest to $x$ in $L$, \cite{Federer}.
The structure of  sets of positive reach is well understood,
 \cite{Federer}, \cite{Kleinjohann}, \cite{Bangert}, \cite{Ly-reach}, \cite{Ly-conv},  \cite{Rataj}; some features are summarized in Section \ref{sec: posreach}.
 %In particular, we have:

In general, even for very nice manifolds $M$, some leaves of $P$  may be non-manifolds and have rather complicated local topological structure, see Section   \ref{sec: posreach} for examples.
However, most leaves are manifolds, and  any submetry is a Riemannian submersion on a large set:

\begin{thm} \label{thm: regular}
Let $M$ be a  Riemannian manifold  and let $P:M\to Y$ be a submetry. Then $Y$ has an open, convex,
 dense subset $Y_{reg}$, locally isometric
to a Riemannian manifold with a Lipschitz continuous Riemannian metric, and
$P:P^{-1} (Y_{reg}) \to Y_{reg} $ is a $\mathcal C^{1,1}$ Riemannian submersion.
\end{thm}

{ Here and below a   $\mathcal C^{1} $  map is called  $\mathcal C^{1,1} $ if its differentials  depend \emph{locally} Lipschitz continuously on the point.  Even for smooth  manifolds $M$ and $Y$,  $\mathcal C^{1,1}$ regularity  in Theorem \ref{thm: regular}  is optimal, see Example \ref{ex: twoside+}.}

  It was already  observed  in \cite{BGP},  that a lower bound on curvature in the sense of Alexandrov is preserved under submetries. Localizing the argument, we see that a space of leaves $Y$ as in Theorem  \ref{thm: leaf} is an { \emph{Alexandrov region},  
  	a length space in which  every point has a compact convex neighborhood isometric to an Alexandrov space, see Section \ref{sec: alex} below and \cite{Leb-Nep}.}
  % length space, which is	an \emph{Alexandrov region}, in the sense  that  it satisfies the definition of an Alexandrov space locally.   It is shown in  \cite{Leb-Nep} that  every point $y$ in an Alexandrov region has a compact convex neighborhood isometric to an Alexandrov space. Thus, this holds for the space of leaves $Y$ in Theorem  \ref{thm: leaf}.
	We can therefore use all the terminology established in Alexandrov spaces.  In particular, we
have
  spaces of directions $\Sigma _yY$ and tangent cones $T_yY$
at $y\in Y$, which are Alexandrov spaces of curvature $\geq 1$, respectively $\geq 0$.  Also the notions of \emph{boundary} and \emph{extremal subsets}, \cite{Petrunin-semi},
 on any Alexandrov region are well-defined.

 %{ Non-manifold fibers of $P$ are related the boundary $\partial Y$ of $Y$:}
 % is responsible for the non-manifold leaves of $P$:

%\begin{thm} \label{thm: boundary}
%Let $P:M\to Y$ be a submetry from a  Riemannian manifold $M$.
%For any point $y\in Y\setminus \partial Y$, the fiber $P^{-1} (y)$ is a $\mathcal C^{1,1}$-submanifold of $M$.
% \end{thm}

  The space of leaves   has a much more   special structure than a general Alexandrov space. 
	%{Theorem \ref{cor: quotient}.(2)  seems to be new even for orbit spaces  under isometric group actions. }

 %  {\blue In the next theorem and later,   \emph{geodesics} will always be globally globally  minimizing and  parametrized by arclength}.
	%From  Theorem \ref{thm: leaf} we deduce:

   % The following result is local in nature and the assumed uniform lower bound on the sectional curvature is irrelevant.
  % It has pointwise positive injectivity radius and  the distance functions to

%Thus if $P:X\to Y$ is a submetry and $X$ is a Riemannian manifold as above, then $Y$ is an Alexandrov space. However, it is an Alexandrov space of very special type as will be shown below.

 %In the next statement and the rest of the text, geodesics are always minimizing and parametrized by arclength.

\begin{thm} \label{cor: quotient}
 Let  $M$  be a  Riemannian manifold,
 let $P:M\to Y$ be a submetry and let $y\in Y$ be a point.
		%Then $Y$ is a  locally compact, locally geodesic metric space.
  %an Alexandrov space of curvature at least $\kappa$.
 Then   there exists some $r=r(y) >0$ such that the following holds true.
\begin{enumerate}
\item\label{inj-rad-base}   A geodesic of length $r$ starts in every direction $v\in \Sigma _y Y$.
\item\label{convex-balls-base}  For any $s\leq r$, the closed ball $\bar B_s (y)$ is strictly convex in $Y$.
\item\label{alexandrov-spheres-base} For any $s\leq r$, the boundary $\partial B_s (y)$ is an Alexandrov space.
\end{enumerate}	
\end{thm}

{ Theorem \ref{cor: quotient}.(2)  seems to be new even for orbit spaces  under isometric group actions, see, however,
\cite{Pet-Per}, \cite{Kap-reg}, \cite{Nepe}, for related weaker statements valid in general Alexandrov spaces.}

{ Note that in the formulation of Theorem \ref{cor: quotient}  and below, \emph{geodesic} will always be  globally length minimizing curve parametrized by arclength.}

%{\blue Theorem \ref{cor: quotient}.(2)  seems to be new even for orbit spaces  under isometric group actions. }
	% Theorem  \ref{cor: quotient}.(1) implies that} the exponential map $\exp_y$ defines a  homeomorphism between  the ball $B_r(0_y)$  in the tangent cone $T_yY$
%and the ball $B_r(y)$ around $y$ in $Y$.

%\emph{injectivity radius} of $Y$ at $y$. Thus, it is the largest number $r>0$, such that for every direction $w\in \Sigma _yY$
%there exists an (always minimizing, parametrized by arclength) geodesic $\gamma:[0,r]\to Y$ starting at $y$ in the direction of $w$.   In this case the naturally defined
 % exponential map $\exp_y$ is a homeomorphism between  a ball of radius $r$ in the tangent cone $T_yY$ around the origin $0_y$ and the ball $B_r(y)$ around $y$ in $Y$.
%The curvature bound $\kappa$ in (1)  can be chosen as the lower curvature bound of $M$ if it is finite.
%In the theorem   and in  the rest of the text, geodesics are always minimizing and parametrized by arclength.

%This result {\blue very much}  simplifies the complexity of the local structure of
%the Alexandrov  region $Y$.
{Also the structure of tangent cones  and  spaces of directions turns out to be  very restrictive, as  they are spaces of leaves of equidistant decompositions of  Euclidean spaces and spheres, respectively:}

\begin{thm} \label{thm: link}
	Let
	$P:M\to Y$ be a submetry, where $M$ is an  $n$-dimensional Riemannian manifold. Let $y\in Y$ be arbitrary. If $\Sigma _y Y\neq \emptyset$   then there  exists some $n> k \geq 0$
	and a submetry $\mathbb S ^k \to \Sigma _y Y$.
\end{thm}

		This result, \cite{Berest} and   Browder's theorem, \cite{Browder}, imply that the  Euclidean cone $Y=C( CaP^2)$  over the Cayley plane cannot be the base of a submetry $P:M\to Y$. This should be compared to the conjectured 
		impossibility to obtain $Y$ as a collapsed limit of Riemannian manifolds with a lower curvature bound, \cite{Kap-collapse}. 
%This implies that the space of directions $\Sigma_y Y$ and the tangent space $T_yY$  enjoy the same geometric properties which we derive for
%$Y$ in this paper.

 {From Theorem \ref{thm: link} we deduce:}

\begin{cor} \label{cor: directions}
	Let $M$ be a Riemannian manifold and  $P:M\to Y$ be a submetry.
	 Let $y\in Y$ be arbitrary. {Then, for some $l\geq 0$, the
	tangent space $T_yY$  has a canonical decomposition
 $$T _yY =\R ^l \times  T_y ^0  \;,$$
 where $T_y^0 Y$ is the Euclidean cone over an Alexandrov space
 $\Sigma _y ^0 Y $ of diameter at most $\frac \pi 2$.}
\end{cor}

%See Section \ref{} for further consequences of Theorem \ref{thm: link}.

{By the \emph{$l$-dimensional stratum $Y^l$} of the space of leaves $Y$  we denote   the set of points $y\in Y$  whose tangent space
	$T_yY$ splits off as a direct factor $\R^l$ but not $\R^{l+1}$,  as in Corollary \ref{cor: directions}.  The strata $Y^l$ define
a topologically and geometrically well behaved stratification of  $Y$:}
% but not $\R^{l+1}$ as a direct factor.

%The following result shows that also the structure of \emph{extremal subsets}
%(see  \cite{Pet-Per}, \cite{AKP}, \cite{Petrunin-semi})
%of $Y$ is rather simple, in comparison to a general Alexandrov space.
%In fact,
%the subsets $\bar E ^y$ appearing in the theorem are exactly the \emph{primitive}
%extremal subsets of $Y$.

\begin{thm}  \label{thm: extremal}
	Let $P:M\to Y$ be a submetry from a  Riemannian manifold $M$. Set $m=\dim (Y)$.  For any   $m\geq l \geq 0$, the stratum  $Y^l$ is an $l$-dimensional topological manifold which is locally closed and locally convex in $Y$.
	The maximal stratum $Y^m$ is open and globally convex.

%	the set of points $y\in Y$ such that
%{\blue $T_yY =\R^l \times T_y^0$ as in Corollary \ref{cor: directions}.}
% but not $\R^{l+1}$ as a direct factor.	
%Then $Y^l$ is an $l$-dimensional  manifold which is locally convex in $Y$  and $Y^m =Y_{reg}$ for $m=\dim Y$ is globally convex.

For any $y\in Y^l$, the  closure $\bar E^y$ of  the connected component $E^y$ of $y$ in $Y^l$
  is the smallest extremal subset of $Y$  which contains $y$.
\end{thm}

%The last statement in the above theorem means, by definition, that the subsets $\bar E ^y$ appearing in the theorem are exactly the \emph{primitive
%extremal subsets} of $Y$.
% Moreover, it implies,   that the closure   $\bar E^y$ is the union of
% some connected components of some $Y^{l'}$ with  $l'\le l$.

%The set $Y_{reg}$ appearing in Theorem \ref{thm: regular} is nothing but the largest stratum  $Y^m$, with $m=\dim (Y)$.
{ It turns out that, the distance on the  strata $Y^l$ is  locally induced by a Lipschitz continuous Riemannian metric, see Theorem
\ref{thm: sing} below.  %Moreover, the subsets $Y^l$ have positive reach in $Y$:
%\begin{thm}
%	Inhalt...
%\end{thm}
Moreover, the strata $Y^l$  have positive reach in $Y$ and  the topological structure of a submetry over any stratum is rather simple:

% The extremal subsets turn out to be  the  basic building blocks for the topological structure of $P$:

\begin{thm} \label{thm: strata}
	Let $M$ be a   Riemannian manifold and let  $P:M\to Y$ be a submetry.
The preimage $P^{-1} (Y^l)$ of any stratum $Y^l$ is a locally closed subset of positive reach in $M$.   	
	
	If $M$ is complete, then for any connected component $E$ of  $Y^l$, the restriction
	$P^{-1} (E)\to E$ is a fiber bundle.
\end{thm}	

We mention, that \emph{quasigeodesics} in the space of leaves $Y$ are much simpler than in general Alexandrov spaces: they are concatenations of geodesics, Corollary \ref{cor: quasigeod}. Moreover, the quasigeodesic flow exists almost everywhere and preserves the Liouville measure, Section \ref{subsec: quasigeod}. }

 { Non-manifold fibers of $P$ are related to the boundary $\partial Y$ of $Y$:}
 % is responsible for the non-manifold leaves of $P$:

\begin{thm} \label{thm: boundary}
Let $P:M\to Y$ be a submetry from a  Riemannian manifold $M$.
For any point $y\in Y\setminus \partial Y$, the fiber $P^{-1} (y)$ is a $\mathcal C^{1,1}$-submanifold of $M$.
 \end{thm}

%Due to Theorem \ref{thm: boundary}, most fibers of any  submetry $P:M\to Y$ are  $\mathcal C^{1,1}$-submanifolds of $M$.
%It very often happens that the fibers over points in $\partial Y$ can be $\mathcal C^{1,1}$-manifolds as well.
 Of particular importance are
 {   submetries
 for which \emph{all} fibers are $\mathcal C^{1,1}$-submanifolds.
 %
%the so-called \emph{transnormal  submetries}, { i.e. submetries}
%for which \emph{all} fibers are $\mathcal C^{1,1}$-submanifolds.
 They are defined under the name \emph{manifold submetry} in  \cite{Grovesubmet} and investigated further in a more specific situation in \cite{Mendes-Rad}. We  call them \emph{transnormal submetries}, borrowing the term \emph{transnormality}  from the theory of singular Riemannian foliations, \cite{Molino}.}

 % We emphasize,   that  here and below  manifolds are allowed to have components of different dimensions.
 %The last point of the following characterizations relates our definition to the notion of  \emph{transnormality}
 %in the theory of singular Riemannian foliations, \cite{Molino}.

\begin{thm} \label{thm: transnormal}
Let  $M$ be a Riemannian manifold and $P:M\to Y$ be a  submetry.
 %with $M$ a Riemannian manifold.
 Then the following are equivalent:
\begin{enumerate}
\item $P$ is a transnormal submetry.
\item All fibers of $P$ are topological manifolds.
\item Any local geodesic in $M$ which starts  normally to any fiber of $P$ remains normal to all fibers it intersects.
\end{enumerate}
\end{thm}

Transnormal submetries are stable under some limit operations, Corollary \ref{cor: transstable}. In particular, if $P\co M\to Y$ is a transnormal submetry then, for any $x\in M$, the differential $D_xP :T_xM\to T_yY$ is  a transnormal submetry as well, Proposition \ref{prop-subnormal}.
Moreover, transnormal submetries have the property known as \emph{equifocality} in the theory of singular Riemannian foliations, see Proposition \ref{prop: equifoc}.
{ Finally, for any leaf $L$ of a  transnormal submetry $P:M\to Y$,
the foot point projection $\Pi ^L:U\to L$ in a neighborhood of $L$ restricts as a fiber bundle to any leaf $L' \subset U$, see Theorem \ref{thm-tubular}.}       
 	
 %	For a transnormal submetry $P:M\to Y$, the images in $Y$ of two horizontal local geodesics in $M$ coincide if these images coincide

 For a transnormal submetry, the preimage $P^{-1} (Y^l)$ of any stratum is a
locally closed $\mathcal C^{1,1}$ submanifold of $M$ and the restriction of $P$ to this submanifold is a $\mathcal C^{1,1}$ Riemannian submersion.
Even if $M$ is smooth one cannot expect higher regularity of the stratification of  $P$.  However,  $\mathcal C^{1,1}$ submanifolds of a manifold with curvature locally bounded from both sides again has curvature locally bounded from both sides, \cite{KL}, so we stay in the  category of manifolds we have chosen to work with.

\subsection{Questions}
There are many open questions about finer structural properties of submetries. We would like to collect a few of them below.

%A positive answer to the following question will provide a more natural converse to Theorem \ref{thm: leaf}.

%\begin{quest}
%	Given any subset of positive reach $L$ of a Riemannian manifold $M$,
%	can one  change the Riemannian metric on a small neighborhood $U$ of $L$
%	so that the new metric is  smooth and complete on $U$ and $L$ has reach $\infty$ in this metric?
%\end{quest}

We do not know if base spaces of transnormal submetries are different from base spaces of general submetries. Even if $\dim (Y)=1$ the following question is absolutely non-trivial:

\begin{quest}
Given a submetry $P:\mathbb S^n\to Y$, does there exist a \emph{transnormal} submetry of the round sphere (possibly of different dimension) with the
same quotient space $Y$?
%We do not know the answer even for $1$-dimensional spaces $Y$.
\end{quest}

 One obtains a closely related question if $\mathbb S^n$ is replaced by a general manifold $M$.

 The positive answer  to the following question in the non-collapsed case is provided in Section \ref{sec: transnorm}. We expect that in the collapsed case the answer is affirmative as well.

\begin{quest}
 Let a sequence $P_i:M_i\to Y_i$   of submetries converge in the Gromov--Hausdorff sense to a submetry $P:M\to Y$. Assume  that
 $M_i$ and $M$ have the same dimension and curvature bounds and  that $P_i$ are transnormal. Is  $P$
  transnormal?  % even if $Y$ has a lower dimension than $Y_i$?
\end{quest}

The following question is natural in view of the structural results  Theorem \ref{thm: sing} and  Corollary \ref{cor: strict}:
%In  $M$ is sufficiently smooth, the answer to the next question is affirmative in the regular part, as will be shown in the continuation of this paper

\begin{quest}
Do the strata  $Y^l$ defined in Theorem \ref{thm: extremal} have curvature locally bounded from both sides?
%Does the preimage $P^{-1} (Y^l)$ contain a $\mathcal C^{1,1}$ submanifold, open and dense in $P^{-1} (Y^l)$?
\end{quest}

%The next question is related to the previous one and to  the slice theorem, which might be true for
%transnormal submetries, cf. \cite{Mendes-Rad-slice}.
%\begin{quest}%
%	Given a transnormal submetry $P:M\to Y$ do the preimages of the strata $P^{-1} (Y^l)$ define a Whitney stratification of $M$?	
%\end{quest}

 %If the slice theorem holds true then  the differentials of a submetry along a fixed fiber should  be pairwise isometric.
 { It should be possible to derive  a positive answer to the following question as  a consequence   from Theorem \ref{thm: strata}  and Corollary \ref{lem: semicont}}:
	\begin{quest}
Given a transnormal submetry $P:M\to Y$, does
the decomposition of $M$  into   connected components  $M_i^l$ of  preimages of strata  $P^{-1} (Y^l)$ satisfy Whitney's conditions (A) and  (B)?
 \end{quest}

\begin{comment}
 An affirmative answer to the  following question can be expected in view of  Corollary \ref{cor: op} and  the validity of the result in case of 
 sufficiently smooth manifold and fibers \cite{Mendes-Rad}[Lemma 45].  
% The result and a possible extension to the non-transnormal case might be achieved by applying classical results in geometric topology, summarized 

 \begin{quest}
 Let $M$ be complete and  $P:M\to Y$ be a transnormal submetry.  Let $L',L$ be fibers of $P$ such that the foot-point projection $\Pi^L$ is uniquely defined on $L'$.
 Is $\Pi^L:L'\to L$  a fiber bundle?
 \end{quest}
\end{comment}
%{\red Theorem~\ref{thm-tubular} shows that fibers of transnormal submetries admit nice tubular neighborhoods.
% Much more challenging  is the following:
% }

{ The following question is related to the previous one and Theorem
\ref{thm-tubular} but is likely much more challenging:}
 \begin{quest}
 Let $P:M\to Y$ be a transnormal submetry. Does a version of  the slice theorem hold in $M$, compare \cite{Mendes-Rad-slice}?
 \end{quest}

{  A closely related question is the Lipschitz version of a well-known problem in singular Riemannian foliations, \cite{Molino}, \cite{Wilking}: }
\begin{quest}
	Given a transnormal submetry $P:M\to Y$, do there exist Lipschitz vector fields everywhere tangent to the fibers and generating the tangent spaces to the fibers at all points?
\end{quest}

  { The continuous dependence of differentials of a submetry along a manifold fiber,  Lemma \ref{lem: openn},  leads to the following question for   $M=\mathbb S^n$}:
\begin{quest}  \label{quest: infinite}
	How large can be the set of transnormal  submetries $P:M\to Y$ modulo isometries of $M$, for a fixed compact manifold $M$ and a fixed Alexandrov space $Y$?
%	If $M$ is fixed but $Y$ is allowed to vary, can the space of submetries be not
%	totally disconnected?
\end{quest}
{  For some non-compact manifolds or for non-transnormal submetries,
the space of submetries can be  infinite-dimensional,  Examples  \ref{ex: twoside+}, \ref{ex: infinite}.}

At least for quotient spaces $Y$ of Riemannian manifolds, the answer to the next question, related to \cite{Li}, should be affirmative.

\begin{quest}
Given an Alexandrov space $Y$, does there exist a description of all discrete submetries $P:X\to Y$, with $X$ an Alexandrov space,
similar to the Riemannian orbifold case, \cite{Lange}?
\end{quest}

 There are many  basic topological questions. {For instance:}
 %{\blue In view of the conjectural connections to collapsing with lower curvature bounds,  {\red the following question is also of interest}

\begin{quest}
	Which  manifolds  admit  non-trivial { (transnormal)} submetries for some Riemannian metric?% {\red For example, we conjecture that a K3 surface does not admit any nontrivial transnormal submetries with connected fibers.}
% {\blue Which manifolds admit non-trivial transnormal submetries? }
\end{quest}

{ See \cite{GRad}, for
  related  results for singular Riemannian foliations.}
  
%{\red It is expected that convergence with lower sectional curvature bound should in some sense be topologically modeled by submetries. This is known to be true for convergence with two-sided curvature bounds by work of Cheeger, Fukaya and Gromov. If the limit is smooth it's known to be true by Yamaguchi's fibration theorem. It also holds for noncollapsing convergence by Perelman's stability Theorem. 
%To illustrate possible implications of this conjecture let us observe that the cone over the Caley plane $\Sigma=CaP^2$ can not occur as the base of a submetry $M^n\to C\Sigma$. Indeed if such submetry existed, by Proposition~\ref{prop: maininfinite} there would exist a submetry $\Ss^k\to CaP^2$ for some $k\le n$ which would have to be a fiber bundle. This is known to be impossible by a result of Browder \cite{Browder}. Therefore we conjecture that the cone over  $CaP^2$ can not occur as a  pointed  limit of manifolds $M^n_i$ with sectional curvature bounded below. This conjecture is supported by results from \cite{Kap-collapse} and \cite{Kap-Guij}.

%Topological versions of transnormal submetries should be appropriately defined stratified maps. This naturally leads to the following question
%\begin{quest}
%Let $P\co M\to Y$ be a transnormal submetry. Is $P$ a stratified map in the sense of Whitney?
%\end{quest}
%}

\subsection{Structure of the paper}
 In Section \ref{sec: prel} we fix notation and collect  basic
facts about submetries. In particular, we introduce the notion of local submetries and discuss horizontal lifts of curves. In Section \ref{sec: alex}
we discuss some basic facts about submetries between general Alexandrov regions.
In particular, we recall that submetries preserve lower curvature bounds and have differentials at all points.  In Section \ref{sec: lift} we observe that a submetry between Alexandrov spaces lifts many semiconcave functions to semiconcave functions and commutes with the corresponding gradient flows and discuss the first structural consequences of this fact. In Section \ref{sec: infinite} we discuss the structure of differentials of submetries between Alexandrov spaces.  All these sections are of general and auxiliary character, most statements contained in them  have appeared in \cite{Lysubm} and might be known to specialists.   The findings of Section \ref{sec: infinite}  include
the proofs of Theorem \ref{thm: link} and Corollary \ref{cor: directions}.

 Only  in  Section \ref{sec: posreach} we turn to the main subject of this paper, submetries of Riemannian manifolds and prove
Theorem \ref{thm: leaf}.

 In Section \ref{sec: onefiber}  we begin to investigate the structure of the base space $Y$ and prove  a part of Theorem \ref{cor: quotient}.
 In the most technical Section \ref{sec: tech} we collect some  observation about  (semi-) continuity of differentials of a submetry.
These are used in Section \ref{sec: convex} to prove that small balls in the base space are convex and to finish the proof of Theorem \ref{cor: quotient}.
In Section \ref{sec: reg}, the structure of the regular part is investigated and Theorem \ref{thm: regular} is verified.    In Section \ref{sec: sing}
we study the properties of the natural stratification of the base space and prove Theorem \ref{thm: extremal} and Theorem \ref{thm: strata}.
 In Section \ref{sec: transnorm} we discuss manifold fibers and transnormal submetries and prove Theorem \ref{thm: boundary} and Theorem \ref{thm: transnormal}.

%Finally, in the Appendix we provide a new proof of Theorem \ref{thm: factor}.
%and also  discuss the structure of quasigeodesics and quasigeodesic flows in
%base spaces.}

\subsection{Acknowledgements}
 The study was initiated many years ago in the PhD thesis of A.L.,  but the results were not brought into a final form. In the meantime some results were found and used by other authors and the interest in the subject seem to have increased, justifying { a systematic investigation}.

{We express our} gratitude to many people who over the years have { motivated us to publish the findings}. The non-complete list includes Marcos Alexandrino, Werner  Ballmann,  Claudio Gorodski, Karsten Grove, Ricardo Mendes, Marco Radeschi.

   We are very grateful to Marco Radeschi for helpful comments.

\section{Preliminaries and basics} \label{sec: prel}
\subsection{Notations}
By $d$ we denote the distance in metric spaces.
A metric space is \emph{proper} if its closed bounded subsets are compact.

% By $B_r(x)$ we denote the open metric ball of radius
%$r$ around $x$.
 For a subset $A$ of a metric space $X$ we denote by $d_A:X\to \R$ the distance function to the set $A$
and by $B_r(A)$ the open $r$-neighborhood around $A$ in  $X$.

The length of a curve $\gamma$ will be denote by $\ell (\gamma)$.  Curves of finite length are called \emph{rectifiable}.
For a locally Lipschitz curve $\gamma :I\to X$ in a metric space
$X$ we denote by $|\gamma' (t)| \in [0,\infty)$ the \emph{velocity}
of  $\gamma $ in $t\in I$.  The velocity
is defined for almost all $t\in I$  and $\ell (\gamma)  =\int _I |\gamma'(t)| \, dt$.

 A metric is a \emph{length metric} if the distance between any pair of points equals the infimum of the lengths of curves connecting the points.

A \emph{geodesic} will denote an isometric  (i.e. distance preserving) embedding of an interval.
In particular, all geodesics are parametrized by arc-length.  A metric space $X$ is \emph{geodesic} if any pair of its points is connected by a geodesic.

\subsection{Main definitions}
Two subsets $L_1,L_2$ of a metric space $X$ are called \emph{equidistant} if $d_{L_i}$ is constant on $L_j$, for $i,j=1,2$.

Recall from the introduction that a map $P:X\to Y$ is a \emph{submetry} if for any $x\in X$ and any $r>0$ the equality
$P(B_r(x))= B_r(P(x))$ holds true.  We call $X$ the \emph{total space} and $Y$ the \emph{base} of a the submetry $P$.

The following observation  is a direct consequence of the definition.

\begin{lem} \label{lem: clear}
A map $P:X\to Y$ between metric spaces is a submetry if and only if $P$ is surjective and, for any  $y\in Y$, we have $d_y \circ P= d_{P^{-1} (y)}$.
\end{lem}

%A map $f:X\to Y$ is a submetry if and only if $f$ is surjective and, for any point $y\in Y$, we have $d_y \circ f= d_{f^{-1} (y)}$.
In particular, for any pair of points $y_1,y_2 \in Y$ with fibers $L_i := P^{-1} (y_i)$,  the function $d_{L_i}$ is  constantly equal $d(y_i,y_j)$ on $L_j$. Hence $L_1$ and $L_2$ are equidistant.
%of a submetry $f$ are equidistant.
On the other hand, if a metric space $X$ is  decomposed in a family of closed, pairwise equidistant subsets $L_i, i\in I$, then the \emph{set of leaves} $I$
becomes a metric space, when equipped with the natural distance between the corresponding subsets in $X$.  The canonical projection $P:X\to I$ sending a point $x$ to the  leaf $L_x:=P^{-1} (P(x))$ through $x$
is a submetry with respect to this metric.

Hence, there is a one-to-one correspondence (up to isometries between base spaces) of submetries with total space $X$ and decompositions of $X$ into closed equidistant  subsets, \cite[Lemma 8.1]{Mondino}.
In particular, any isometric group action on a space $X$, with all orbits closed, determines a unique submetry, the \emph{quotient map}, whose fibers are the orbits of the  action.

%If $f:X\to Y$ is a submetry, we say that  $x_1,x_2 \in X$ are
%\emph{$f$-near points}   if $d(x_1,x_2) =d(f(x_1),f(x_2))$.
%The points $x_1$ and $x_2$ are $f$-near if and only if
%$x_2$ is the closest point to $x_1$ on the fiber $f^{-1} (f(x_1))$.

\begin{rem}
	In \cite{Berest}, submetries are defined by the slightly stronger requirement that the images of all \emph{closed} balls are closed balls of the same radius.
	 For proper spaces both notions coincide.
\end{rem}

\subsection{Basic properties and operations with submetries}
Any submetry is $1$-Lipschitz and surjective.
%  In particular, it does not increase length of curves and Hausdorff dimensions of subsets.
	
	 For a submetry $P:X\to Y$, a point $y\in Y$ is isolated in $Y$ if and only if
the fiber $P^{-1} (y)$ has  non-empty interior. Thus, if $X$ is connected  then either $Y$ is a singleton, or  any fiber $P$ is nowhere dense in $X$.

  % the Hausdorff dimension of the base is not larger than the Hausdorff dimension of the total space.
 Many properties of the total space are inherited under submetries,   \cite[Proposition 1]{Berest2}:
%\cite[Section 4.6]{BGP} or \cite{bipolar} for the last and
% \cite[Proposition 1]{Berest2}.
 % We formulate the corresponding local versions, the  proofs are straight forward as in \cite{Berest2}:
% for the remaining statements.

\begin{lem} \label{lem: inherit}
	Let $P:X\to Y$ be a  submetry.
	% let $x\in X$ and $r>0$ be as in Definition \ref{defn: loc}. If  $\bar B_r (x)$ is compact, respectively a length space, then for any $s\leq \frac r 2$
	  If $X$   is  compact or proper or complete or  length space  then $Y$ has the corresponding property.
\end{lem}

For any submetry $P:X\to Y$ and $A \subset Y$ we get from Lemma \ref{lem: clear}:
\begin{equation} \label{eq: dA}
d_A\circ P=d_{P^{-1} (A)} \;.
\end{equation}

A composition of submetries is a submetry.  Moreover, if $P:X\to Y$ is a submetry and, for some  map $Q:Y\to Z$, the composition $Q\circ P$
is a submetry then $Q$ must be a submetry too.
% as  follows directly from the definition.

Any isometry $P:X\to X$ and the projection $P:X\to \{0\}$ to a singleton are submetries.
  A direct product of submetries is a submetry. In particular, the projection of a direct product onto a factor is a submetry.

   If $P:X\to Y$ is a submetry
then the cone $C(P):C(X)\to C(Y)$ between the Euclidean cones over $X$ and $Y$ is a submetry as well.  Any Euclidean cone $C(X)$ admits a submetry $C(X)\to [0,\infty)$ given by  $v\to |v|$, i.e. by the distance to the vertex of the cone.

 For any submetry, $P:X\to Y$ and any subset $Y'\subset Y$ the restriction  $P:P^{-1} (Y') \to Y'$ is a submetry.

Submetries are stable under convergence:

%Since the definition of submetry is defined in terms of distance of finitely many points in the spaces, the notion is stable under (pointed) Gromov--Hausdorff convergence. Moreover, under this convergence fibers converge to fibers in the following sense.

\begin{lem}  \label{lem: stable}
Let $P_j:(X_j,x_j)\to (Y_j,y_j)$ be a sequence of submetries between
pointed proper spaces. If the sequence of spaces $(X_j,x_j)$ converges in the pointed Gromov--Hausdorff topology to a space $(X,x)$ then, after choosing a suitable subsequence, $(Y_j,y_j)$ converge to a space $(Y,y)$, the submetries $P_j$ converge to a submetry $P:(X,x) \to (Y,y)$. Finally, under this convergence,
the fibers $P^{-1} (y_j)$ converge to $P^{-1} (y)$.
\end{lem}

\begin{proof}
The uniform compactness of balls of fixed radius around $x_j$ in $X_j$  imply the  uniform compactness of the corresponding balls around $y_j$. Since the maps $P_j$ are $1$-Lipschitz, we can choose a subsequence and assume that $(Y_j,y_j)$ converges to a space $(Y,y)$ and that $P_j$ converges to a map $P$.

  By the definition of Gromov--Hausdorff convergence,  $P$ is a submetry. Clearly, the limit of any sequence of points in $P^{-1} (y_j)$  is contained in $P^{-1} (y)$. On the other hand,
any  $z \in f^{-1} (y)$ is a limit of a sequence of points
$z_j \in X_j$ such that  $P_j(z_j)$ converges to $y$.  Consider
 $\hat z_j \in P^{-1}(y_j)$, with $d(\hat z_j, z) = d(y_j,y)$
  and observe that $\hat z_j$ converge to $z$.
\end{proof}

Readers familiar with the ultralimits  will easily verify the more general statement that any ultralimit of submetries is a submetry.

%\subsection{Examples}

Basic examples of submetries were mentioned in the introduction:
%We collect here a set ob basic examples of submetries.

%\begin{ex}
%Any metric space $X$ admits two trivial submetries: the identity $f:X\to X$  and the constant map to a one-point space $Y=\{0\}$.
%\end{ex}

%Any metric space $X$ admits a trivial submetry $f:X\to \{ 0\}$.

%If a pair of subsets $A_1,A_2$  of $X$ are equidistant then so are the closures of the subsets.

\begin{ex}
	For any isometric  action of a group $G$   on a metric space $X$
	the orbits are pairwise equidistant. Thus, the closures of the orbits  of $G$ define an equidistant decomposition of $X$ in closed subsets  and, therefore, a submetry onto the quotient space.
\end{ex}

The properties  of  isometric actions of a closed Lie group on a Riemannian manifold is a classical object of investigations \cite{Bredon}, see also  \cite{Guijarro-2}, \cite{Harvey}
for similar results on Alexandrov spaces. The present paper aims at
the generalization of the starting points of the  theory to the non-homogeneous setting.

\begin{ex}
The leaves of  a singular Riemannian foliation $\mathcal F$ of any complete Riemannian manifold $M$  are equidistant, \cite{Molino}. If all leaves are closed then $\mathcal F$  defines a submetry
 with total space $M$.
 %We refer to \cite{Rade}, for many such submetries,  which do not come from an isometric group action.
\end{ex}

\subsection{Localization, horizontal lifts of  curves and globalization} \label{subsec: loc}
Since we would like  to restrict submetries to open subsets, we localize the definition of a submetry.

\begin{defn} \label{defn: loc}
	Let $P:X\to Y$ be a map between  metric spaces.  We say that $P$ is a local submetry if for any point $x\in X $ there exists some $r>0$ with the following property.
	For any point $x' \in B_r (x)$ and any $s<r-d(x,x')$ we have $P(B_s(x'))=B_s(P(x'))$.
\end{defn}

While a restriction of a submetry $P:X\to Y$ to an open subset $U$ of $X$ { is rarely} a submetry, it is always a local submetry.
%\footnote{\blue A: The former Lemma 2.8 i now just text}

	Any local submetry $P:X\to Y$ is an open map which is locally $1$-Lipschitz. In particular, $P$ does not increase length of curves and
	Hausdorff dimension of subsets.

	For a local submetry $P:X\to Y$ we call a rectifiable curve $\gamma :I\to X$ \emph{horizontal} (with respect to $P$) if $\ell (\gamma)=\ell (P\circ \gamma)$.   In this case, we call $\gamma$  a \emph{horizontal lift} of $P\circ \gamma$.

 Let $I\subset \R$ be an interval.  A locally Lipschitz curve $\gamma :I\to X$ is horizontal if and only if for almost all $t\in I$ the velocities $| \gamma '(t)|$  and
$|(P\circ \gamma )'(t)|$ coincide.

For $1$-Lipschitz curves
the following Lemma is the special case of \cite[Lemma 4.4]{Lyt-open}.
For general rectifiable curves, the result  follows after a reparametrization of the curve by arclength.

\begin{lem} \label{lem: lift}
	Let $P:X\to Y$ be a local submetry. Assume that for some $x\in X$ and $r>0$, the closed ball $\bar B_r (x)$ is compact.  Then, for any
	curve $\eta :I\to Y$ of length  at most $ r$ which starts at $P(x)$, there exists a horizontal lift of $\eta$ starting at $x$.
\end{lem}

In particular, for a local submetry $P:X\to Y$ between proper spaces, any rectifiable curve in $Y$ admits a horziontal lift with the prescribed lift of a starting point.   From  Lemma  \ref{lem: lift} we deduce the following local-to-global property:

\begin{cor} \label{cor: locglob}
	Let $P:X\to Y$ be a local submetry between length spaces. {  Assume that for some $x\in X$ and $r>0$, the closed ball $\bar B_r (x)$ is compact.
	% Let $x\in X$ be arbitrary. If the closed ball $\bar B_r(x)$ in $X$ is compact,
	Then $P(B_s(x)) =B_s(P(x))$, for any $s\leq r$.}
	
	Moreover,  equality \eqref{eq: dA} holds on $B_{\frac {r} {3}} (x)$ for any subset $A\subset Y$ with non-empty { intersection} $A\cap B_{\frac {r} 3} (P(x))$.
	
	A local submetry between proper length spaces is a submetry.
\end{cor}

In fact, \cite[Proposition 4.3]{Lyt-open} shows that the property of being a submetry can be recognized not only locally, but infinitesimally.

Another  consequence of Lemma  \ref{lem: lift} is the following
statement that  allows us to replace the induced metric on a subset by the intrinsic one:

\begin{cor}  \label{cor: length}
	Let $X$ be locally compact and $P:X\to Y$ be a local submetry.
	 If  any pair of point in $X$ is connected by a rectifiable curve, then equipping $X$ and $Z=P(X)$ with their induced length metrics $d^X$ and $d^Z$, we obtain a local submetry
	 $P:(X, d^X) \to (Z,d^Z)$.
	 	 %any pair of point in $Z:=P(X)$ is connected by a rectifiable %curve and the }the induced length metric  $d ^X $ on $X$ is finite then the induced length metric $d^Z $ on  the image $Z=P(X)$ is finite.     If, in addition, $X$ is locally compact, then
	%the map $P:(X, d^X) \to (Z,d^Z)$ is a local submetry.
	\end{cor}

\subsection{Gradient curves and  submetries} \label{subsec: Ambrosio}
We refer to  \cite[Chapter 2]{Ambrosio}, \cite{AGS} for
gradient curves of general functions in general metric spaces and
to \cite{Petrunin-semi}, for the case of semiconcave functions in Alexandrov spaces, important for us.

Recall that for a locally Lipschitz function $g:Z\to \R$ on a metric space $Z$
%Let $U$ be an open subset of $X$ or $Y$. Let $f:U\to \R$ be  locally Lipschitz continuous.
the \emph{ascending slope}   of $g$ at a point $x$ is defined as
$$|\nabla ^+ g| (x)=\limsup _{y\to x} \frac {\max \{g(y)-g(x), 0 \} }  {d(x,y)}   \in [0,\infty )\, .$$
A  locally Lipschitz continuous curve $\gamma :[0, t )\to Z$ is called  a gradient curve of $g$ starting
at $x=\gamma (0)$  if for almost all  $s\in [0,t)$
$$|\gamma '(s)| =|\nabla ^+ g | (\gamma (s))  \; \;  \text{and}   \; \;(g\circ \gamma )' (s)= |\nabla ^+ g |^2 (\gamma (s)) \;.$$
Let now $P:X\to Y$ be a local submetry, let $g:Y\to \R$ be a locally Lipschitz  function. Then $P\circ g$ is locally Lipschitz
and, for all $x\in X$,  $$|\nabla ^+ (P\circ g ) | (x) =|\nabla ^+ g| (P(x)) \;.$$

It follows from the definition of  gradient curves that
a locally Lipschitz curve $\gamma :[0,t) \to X$ is a gradient curve of $P\circ g$ if and only if  $\gamma$ is a horizontal curve and
$P\circ \gamma$ is a gradient curve of $ g$.

\section{Submetries and lower curvature bounds} \label{sec: alex}
\subsection{Alexandrov spaces, Alexandrov regions} 
{ Following the notation in \cite{AKP}, for  $\kappa \in \R$  and
 points $x,y,z$ in a metric space $X$, we denote  by $\tilde \angle _{\kappa} (x_y ^z)$ \emph{the $\kappa$-comparison angle at $x$}, whenever it is defined.} %Thus, $\tilde \angle _{\kappa} (x_y ^z)$ is the angle at $\bar x$ in the  simply connected surface $M^2_{\kappa}$ constant curvature $\kappa$ of a triangle
%$\bar x,\bar y, \bar z$ with the same side-lengths as $x,y,z$.

 The metric space $X$ is $CBB(\kappa)$ if for any
 $p,x,y,z \in X$ the following inequality holds true, whenever all $\kappa$-comparison angles are defined:
 $$  \tilde \angle _{\kappa} (p_x ^y) +  \tilde \angle _{\kappa} (p_y ^z)  + \tilde \angle _{\kappa} (p_z ^x)   \leq 2\pi    \;.   $$
 An \emph{Alexandrov space of curvature $\geq \kappa$} is  a complete  length space of finite Hausdorff dimension which is $CBB(\kappa)$. Any such space is   proper,  \cite{BGP}, in particular it is geodesic.
 For an extensive literature on such spaces see \cite{BGP},
  \cite{Petrunin-semi}, \cite{AKP} and the bibliography therein.
  We will assume some familiarity with the theory of Alexandrov spaces.

   A metric space $X$ is an \emph{Alexandrov region}  if it is a length space of finite Hausdorff dimension in which every point $x$ has a $CBB(\kappa)$ neighborhood, with $\kappa$  possibly depending on $x$.

   For instance, any smooth Riemannian manifold is an Alexandrov region. Due to \cite{Leb-Nep}, a length space $X$ is an Alexandrov region if and only if every  point
  $x\in X$ admits  a compact neighborhood $U$ of $x$ in $X$ which is an Alexandrov space.
  %(respectively an Alexandrov space of curvature $\geq \kappa$.)
 % For instance, any smooth Riemannian manifold is an Alexandrov region.
  %Any Alexandrov space is an Alexandrov region, \cite{Petrunin-semi}; in fact,
 % compact convex  neighborhoods  can be found in any small ball.

   %\begin{rem}
  %	Due to \cite{Nepechiy}, for any Alexandrov space $X$ small compact convex   neighborhoods $U$ as in the definition can be in fact chosen very close
  %\end{rem}

  %An Alexandrov region $X$ of curvature $\geq \kappa$  is an Alexandrov space if and only if $X$  is complete. Due to \cite{Leb-Nep}, a locally compact length space $X$ is an Alexandrov region of curvature $\geq \kappa$ if and only if
  %$X$ is locally $CBB(\kappa)$ and  $X$ has finite Hausdorff dimension.

  \subsection{Basic geometric objects in Alexandrov spaces and regions}
  For any point $x$ in an Alexandrov space $X$ (and therefore in an Alexandrov region $X$) we have a well-defined \emph{tangent space} $T_xX$.
  This tangent space is a Euclidean cone  with vertex $0_x$ over the \emph{space of directions}
  $\Sigma _x X$.

We refer to \cite{Petrunin-semi} for a detailed discussion of semiconcave functions on Alexandrov spaces and their gradient flows.
We will only use the following facts. For any (in the sequel always locally Lipschitz continuous) semiconcave function $g:U\to \R$ on an   Alexandrov region $U$, there is a unique maximal
gradient curve starting at any point of $U$.  {  The local gradient flow $\Phi _t $ is locally Lipschitz continuous on $U$. }
% More precisely,
%The local Lipschitz constants of $\Phi _t$  depend only on the time $t$ and the semiconcavity  constant  of $g$: {\red
% if $g$ is $\lambda$-concave then $\Phi_t$ is locally $e^{\lambda t}$-Lipschitz.}

%For any sequence of subsets $A_1,...,A_k$ of an Alexandrov space $Z$   and any function
%$\Theta :\R^k  \to \R$ which is semiconcave and non-decreasing in each argument, the function
%$ q_{\Theta, A_1,...,A_k} := \Theta (d_{A_i}^2,....,d_{A_k}^2) :Z\to \R$ is semi-concave.  We will call any such function \emph{special}.

A  subset $E$ of an Alexandrov space $Z$ is called
an \emph{extremal subset} if   it is invariant under the gradient flow of any
 semiconcave function, \cite{Petrunin-semi}, \cite{Pet-Per}.  Equivalently, $E$ is extremal if it is invariant under the gradient flows of all functions  $d_q^2$, $q\in Z$.
     The same definition  provides a notion of an extremal subset in an Alexandrov region.
% {\blue One can show
%that  a closed subset $E$ of an Alexandrov space (region) $X$ is extremal if and only if for any $x\in E$ the \emph{tangent space} $T_xE$ is an extremal subset of $T_xX$.}

 The \emph{boundary} $\partial X$ of an Alexandrov region   is defined inductively on dimension as the set of all points, for which $\Sigma_xX$ has non-empty boundary. {  The boundary  $\partial X$ is an extremal subset of $X$ and any extremal subset is closed in $X$, \cite{Pet-Per}.}
% and is equal to
%the union of all extremal subsets of $X$ of codimension one \cite{P2,Pet-Per}.}%; equivalently, it is the set of all points, for which $T_xX$ has non-empty boundary.

A \emph{quasigeodesic} in an Alexandrov region $X$  is a  curve $\gamma :I\to X$ parametrized by arclength  such that,  for any $t\in I$,  we have
the  following  inequality for $q\in X$ converging to $\gamma (t)$:
\begin{equation} \label{eq: quasi}
(\frac 1 2 (d_q  \circ \gamma)^2) ''  (t)\leq 1 + o (d(q,\gamma (t)).
\end{equation}
 This is equivalent to the more common definition, \cite[1.7]{PP-quasi}, that
the restriction of distance functions to all points is  as concave as
the restriction of a distance function to a geodesic in the comparison space.
Any \emph{local geodesic} in an Alexandrov region is a quasigeodesic,  but in general
quasigeodesics   can be  much more complicated.
We refer to \cite{Petrunin-semi} and \cite{PP-quasi} for the theory of quasigeodesics
in Alexandrov spaces.

\subsection{Local submetries preserve lower curvature bounds}
The following result is essentially contained in \cite{BGP}.
% \cite{AKP}.
\begin{prop} \label{prop: cbb}
	Let $P:X\to Y$ be a surjective local submetry. If $X$
 is an Alexandrov region (of curvature $\geq \kappa$) then $Y$ is
 an Alexandrov region (of curvature $\geq \kappa$), if we equip $Y$ with the induced length metric.

  If $P:X\to Y$ is a submetry and $X$ is an Alexandrov space of curvature $\geq
  \kappa$ then so is $Y$.
\end{prop}

\begin{proof}
 Due to Corollary \ref{cor: length}, we may assume that $Y$ is a length space.   For any  $y\in Y$, choose an arbitrary $x\in P^{-1} (y)$.  Find   $r>0$  such that $\bar B_{3r}(x)$ is compact
 and $CBB(\kappa)$.
 % moreover we may assume that $6r$ is smaller than  the diameter of $M^2 _{\kappa}$.
  We claim that $B_ r  (y)$ is $CBB(\kappa)$.  Indeed, for any $p,y_1,y_2,y_3 \in  B_r (y)$
 we consider any $\bar p \in P^{-1}(p)$  and then $\bar y_j \in P^{-1} (y_j)$ such that $d(x,\bar p)=d(y,p)$ and $d(\bar p, \bar y_j) = d(p,y_j)$.

 Then $d(\bar y_i,\bar y_j ) \geq d(y_i,y_j)$, Thus,
 the $\kappa$-comparison angles at $\bar p$ are not smaller than the corresponding $\kappa$-comparison angles at $p$.  Therefore,
  $$ \sum _{i=1} ^3 \tilde \angle _{\kappa} (p_{y_i} ^{y_{i+1}})   \leq
  \sum _{I=1} ^3 \tilde \angle _{\kappa} (p_{\bar y_i} ^{\bar y_{i+1}})	   \leq 2\pi    \;.   $$
 Since local submetries do not increase the Hausdorff dimension, $Y$ is a finite-dimensional, locally compact, length space which is locally $CBB(\kappa)$.  By \cite{Leb-Nep},  $Y$ is an Alexandrov region.

 The global statement follows by Toponogov's globalization, in fact, it is already  contained in \cite{BGP}.
\end{proof}

 \subsection{Lifts and images  of geodesics} \label{subsec: liftgeod}
 Let $P:X\to Y$ be a local submetry between Alexandrov regions. { We call
 a curve $\gamma :[a,b]\to X$ a \emph{$P$-minimal geodesic} if $\gamma$ is parametrized by arclength and $$b-a= d(P(\gamma (a)), P(\gamma (b))) \;.$$
 If $\gamma$ is a $P$-minimal geodesic, then $\gamma$ is a geodesic in $X$, $P\circ \gamma$ is a geodesic in $Y$ and $\gamma$ is a horizontal lift of $P\circ \gamma$.
 On the other hand, any horizontal lift $\gamma$ of a geodesic $\hat \gamma :[a,b]\to Y$  is a $P$-minimal geodesic.

The image of a horizontal geodesic $\gamma $ in $X$ under a submetry does not need to be a  geodesic.
%(see e.g.  \cite[Example 3.1]{Guijarro}).
However,   cf. [Proposition 4]\cite{Guijarro}:}
\begin{prop} \label{prop: image}
Let $P:X\to Y$ be a local submetry between Alexandrov regions.
Let $\gamma:I\to X$ be a geodesic. Then $\gamma $ is horizontal if and only if
the composition $P\circ \gamma :I\to Y$ is a quasigeodesic.
\end{prop}

\begin{proof}
If $P\circ \gamma$ is a quasigeodesic, then it is parametrized by arclength. Hence $\gamma$ must be a horizontal in this case.

If $\gamma$ is horizontal then $P\circ \gamma$ is parametrized by arclength and the property \eqref{eq: quasi}
follows from    the equality   \eqref{eq: dA}  and  Corollary \ref{cor: locglob}.
\end{proof}

 \subsection{Differentiability}
 The following result in the case of submetries can be found in \cite[Proposition 11.3]{diff}. For local submetries, the differentiability  is a direct
 consequence of  Corollary \ref{cor: locglob} and
 \cite{diff}.
 %The statement that the  differential is a submetry is a consequence of Lemma \ref{lem: stable}.

 \begin{prop} \label{prop: differ}
 	Any local   submetry $P:X\to Y$ between Alexandrov regions
 	is differentiable at each point $x\in X$.   In other words, for $y=P(x)$, there exists a map $D_xP:T_xX\to T_yY$,
 	such that
 	for every sequence $t_i\to 0$, the
 	submetry $P$ seen as a map between rescaled spaces $P: (\frac 1 {t_i}X,x) \to (\frac 1 {t_i} Y,y)$
 	converge to the map $D_xP$.
 	
 	The map $D_xP:T_xX \to T_yY$ is a submetry and commutes with natural dilations of the Euclidean cones $T_xX$ and $T_yY$.
 \end{prop}

 %The map $D_xP:T_xX\to T_yY$ is automatically a submetry, which moreover commutes with canonical dilatations of the Euclidean cones $T_xX$ and $T_yY$.
 By definition, for any curve  $\gamma :[0,\epsilon) \to X$ starting in $x$ in the direction $v\in T_xX$ the curve $P\circ \gamma$ starts  in the direction
 $D_xP (v)$.

 For submetries the following result is a direct consequence of
 Lemma \ref{lem: stable}.
 %we deduce the following result for submetries.
  For local submetries the statement follows, by  adapting  the proof of Lemma \ref{lem: stable}:

 \begin{cor} \label{cor: differential}
 	Let $P:X\to Y$ be a local submetry between Alexandrov regions. Consider
 	$x\in X$ and $y=P(x)$ and the fiber $L=P^{-1} (y)$.  Then, under the Gromov--Hausdorff convergence $(\frac 1 t X, x) \to T_x X$ the
 	sets $(\frac 1 t L,x) \subset  (\frac 1 t X, x)$ converge for $t\to 0$ to the fiber $D_xP ^{-1} (0_y) $ of the differential
 	$D_xP:T_xX\to T_yY$.	
 	%
 	% In other words, $T_xL = (D_xP) ^{-1} (0_y)$.
 	  	\end{cor}

The last statement can be interpreted as the fact that the tangent cone $T_xL \subset T_xX$ is well-defined and coincides with  $(D_xP) ^{-1} (0_y)$.
% {\red We will adopt this point of view from now on.}
%of the fiber

 %Moreover, the local semi-concavity constant
 %of this function at a point $z\in Z$ depends only  on $\Theta$ in a neighborhood the point  is bounded in terms of

\subsection{Measure and coarea formula}
We will denote here and below by $\mathcal H^m$ the $m$-dimensional  Hausdorff measure.
For any Alexandrov region $Y$ the Hausdorff dimension is a natural number $m$.
The set $Y_{reg}$ of points $y\in Y$ with $T_yY =\R^m$ has full $\mathcal H^m$ measure and  is contained in an $m$-dimensional  Lipschitz  manifold $Y^{\delta}$, \cite{BGP}.

In particular, $m$-dimensional Alexandrov regions are countably $m$-rectifiably metric spaces and the metric differentiability theorem, the area and coarea formula applies to Lipschitz maps between Alexandrov regions \cite{Amb-Kir}, \cite{Kar}.

Let now $P:X \to Y$ be a  local submetry.
From Proposition \ref{prop: differ} and  \cite{Amb-Kir} directly follows:
\begin{lem}
	Let $P:X\to Y$ be a local submetry,  where $X$ and $Y$ are
	$n$- and $m$-dimensional Alexandrov regions.   Then, for $\mathcal H^n$-almost every point $x\in X$ the point $y=P(x)$ is a regular point
	of $Y$  and the submetry  $D_xP:T_xX =\R^n \to T_yY \cong \R^m$ is a linear map. %between
%	$T_xX=\R^n$ and $T_yY=\R^m$.
\end{lem}

In this situation the coarea formula   \cite{Amb-Kir}, \cite{Kar}   reads as:
\begin{cor}
	Let $P:X\to Y$  be a   local submetry,  where $X$ and $Y$ are
	$n$- and $m$-dimensional Alexandrov regions.  Then, for $\mathcal H^m$-almost every point $y\in Y$, the preimage $P^{-1} (y)$ is countably $(n-m)$-rectifiable. For every Borel subset $A\subset X$ we have the equality
	\begin{equation} \label{eq: coarea}
\mathcal H^n (A) =\int _Y  \mathcal H^{n-m} (A\cap P^{-1} (y)) \; d\mathcal H^m (y) \;.
	\end{equation}

\end{cor}

\section{Lifts of semiconcave functions} \label{sec: lift}
\subsection{Special semiconcave functions} \label{subsec: special}
Let $Y$ be an Alexandrov region.  For  $A\subset Y$  and  $y\in Y$ with
 $t=d_A(y) >0$, let $0<2r<t$ be such that $\bar B_{2r} (y)$ is compact.  Then on  $B_r(y)$ we have the equality
$$d_A = d_{S_{t-r}} + r \;,$$
where $S_{t-r}$ is the set of points  $p\in \bar B_{2r} (y)$ with
$d_A(p)=t-r$.

The semicontinuity of the squared distance functions on \emph{Alexandrov spaces} and the previous observation show that  for any \emph{Alexandrov region} $Y$ and any subset
$A\subset Y$ the squared distance function $d_A^2$ is semiconcave on $Y$.

Let $Y$ be an Alexandrov region and  let  $A_1,...,A_k\subset Y$ be closed.
Let  $\Theta :\R^k  \to \R$ be semiconcave and non-decreasing in each argument.
The function
$ q_{\Theta, A_1,...,A_k} := \Theta (d_{A_i}^2,....,d_{A_k}^2) :Y\to \R$
will be  called
%
% By definition,   $f=q_{\Theta, A_1,...,A_k}$ is
 \emph{special semiconcave}  on $Y$.

%Note that for any $A \subset Y$, any $y\in Y$ with $t:=d_A(y)$ and any $0<r<t$ such that $\bar B_{2r} (y)$ is compact we have the equality $_

Since $d_{A_i}^2$ is semiconcave, it follows that
 any special semiconcave function is semiconcave,  cf. \cite[Section 6]{Petrunin-semi}.

\subsection{Lifts of special semiconcave functions} \label{subsec: lift}
For us, the importance of special semiconcave functions is  due to the following

\begin{lem} \label{lem: compose}
Let $P:X\to Y$ be a local submetry between Alexandrov regions.  Let $f:Y\to \R$ be a special semiconcave function.  Then $P\circ f :X\to \R$ is semiconcave.
\end{lem}

\begin{proof}
If $f= \Theta (d_{A_i}^2,....,d_{A_k}^2) :Y\to \R$ then
$$f\circ P =	\Theta (( d_{A_1})^2 \circ P, ...., (d_{A_k})^2\circ P)$$
and it suffices to prove that $(d_A \circ P)^2$ is semiconcave for every subset $A\subset Y$.
But this follows from equality \ref{eq: dA}, Corollary \ref{cor: locglob} and the first observation in Subsection \ref{subsec: special}.
	\end{proof}

%Let $P:X\to Y$ be a local submetry between Alexandrov regions.
%Let $x\in X$ be arbitrary, let $r$ be such that $\bar B_{3r} (x)$ is compact
%and contained in an Alexandrov space $X'\subset X$. Set $y=P(x)$ and assume further, by making $r$ smaller if needed, that $B_{3r} (y)$ is contained in an Alexandrov space $Y'\subset Y$.

%Let $A_1,...,A_k$ be closed subsets of $Y$, all of them intersecting $B_r(y)$.
%Let  $\Theta :\R^k  \to \R$ be semiconcave and non-decreasing in each argument.
%By definition,
%$ q_{\Theta, A_1,...,A_k} := \Theta (d_{A_i}^2,....,d_{A_k}^2) :Y\to \R$
%will be  called
%
% By definition,   $f=q_{\Theta, A_1,...,A_k}$ is
%a \emph{special semiconcave} function on $Y'$.

%A special semiconcave function is indeed semiconcave, \cite{}.

%For a special semiconcave function $f=q_{\Theta, A_1,...,A_k}$ as above, we apply  Corollary \ref{cor: locglob}, and see that the composition  $P\circ f$ on $B_r(x)$ is a special semiconcave function
%\begin{equation} \label{eq: dA2}
%P\circ  q_{\Theta, A_1,...,A_k}= q_{\Theta, f^{-1} (A_1),..., f^{-1} (A_k) } \;.
%\end{equation}
%Thus, the lift  $P\circ f$ of the special semi-concave function $f$ is  special semi-concave
%on the ball $B_r(x)$.
% is  special semi-concave function $f$ is semi-concave.

% Under the local assumption that $X$ and $Y$ are Alexandrov regions and $P$ is a local submetry, we have the validity \eqref{eq: dA2} in the ball
% $B_r(x)$, whenever, the closed ball $\bar B_{3r} (x)$ is compact and whenever
% $A_i$ intersect   the ball $B_r (f(x))$ by Corollary \ref{cor: locglob}.

As explained  in Subsection \ref{subsec: Ambrosio}, the gradient curves of $f\circ P$ are exactly the horizontal lifts of the gradient curves of $f$.    Thus, $P$ sends the gradient flow of $f\circ P$ to the gradient flow of $f$.
More precisely,
%This and the controlled Lipschitz continuity of the gradient flows of semi-concave functions immediately implies:
\begin{lem} \label{lem: liftflow}
	In the above notation, let $\Phi_t$ be the gradient flow  of a special semiconcave function $f$ on $Y$ and let  $\hat {\Phi} _t $ be the gradient flow of the  function $f\circ P$ on  $X$.
	Then, for all $(z,t)$ in the domain of definition of $\hat \Phi $, we have
	$$P(\hat {\Phi} _t (z)) = \Phi _t (P(z)) \,.$$
\end{lem}

\subsection{Perelman's function and its lift}
For any Alexandrov region $Y$ and any point $y\in Y$ there exists a strictly concave function in a neighborhood of $y$ which has its maximum at  $y$.
 More precisely, \cite[Theorem 7.1.1]{Petrunin-semi},
there exists   $ \epsilon =\epsilon(y) >0$ and a special semiconcave function
$f=f_y$  on $Y$ with the following properties:

The restriction of $f$ to the ball $B_{\epsilon} (y)$ is a  strictly concave  function  and has a unique maximum at the point $y$.
Moreover,
%$B_{\epsilon} (y)$ is invariant under the gradient flow of $f$ and
the ascending slope at any point  $z \in B_{\epsilon} (y) \setminus \{ y\}$ satisfies
 $|\nabla ^+ f| (z) \geq \frac 1 2$.

 Using this function we can now easily derive:
 \begin{prop} \label{prop: Perelman}
 	Let $P:X\to Y$  be a local submetry between Alexandrov regions.
 	Let $y\in P(X)$ be arbitrary. Then there exist  a neighborhood
 %	 $O$ of $y$ in $Y$ and
 	$U$ of the fiber $L:=P^{-1} (y)$ in $X$
 	and  a special semiconcave function  $f:Y\to \R$  with the following properties.
 	 	
 	The function $g =f\circ P$ is semiconcave on $X$.
 	% and $P(U)=O$.
 %	 and Lipschitz continuous
 %	on $U$.
 	The set of maximum points of $g$ in $U$ is exactly $L$ and   $|\nabla ^+ g| (q) \geq \frac 1 2$, for any
 	$q\in U\setminus L$.

 	%The gradient flow $\Phi$ of $f$ is defined in $O$ for all times,
 	The gradient flow $\hat \Phi$ of $g$ is defined in $U$ for all times %and %the gradient lines of $g$ are exactly horizontal lifts of gradient lines in $f$.  	
 and 	for some $\delta >0$ we have $\hat \Phi _{\delta}  (U) = L$.
 %For any compact subset $K\subset L$ there exists a neighborhood $O_K$ of $y$ in $Y$ such that the image of the Lipschitz map $\hat {\Phi} _{\delta}:P^{-1} (z) \to L$ contains $K$, for any $z\in O_K$.
%
% If $X$ is complete, one can take $U=P^{-1} (O )$ for some  neighborhood $O$ of $y$. In this case,  $\hat \Phi _{\delta} :P^{-1} (z) \to L$ is surjective for all $z\in O$.
 %
 %The last statement applies to   $K=L$ and $O_L=O$.
 \end{prop}

 \begin{proof}
 Consider a small relatively compact ball $B_{\epsilon} (y)$ around $y$ and  Perelman's function $f=f_y :Y\to \R$ as described above.

 %Since $f$ is special semiconcave,
 The function  $g=P\circ f$ is semiconcave on $X$, by Lemma \ref{lem: compose}. % Since $f$ is Lipschitz continuous on $\tilde O$,
 %the function $g$ is Lipschitz continuous on the open neighborhood
 %$\tilde U:=  P^{-1} (\tilde O)$ of $L$.    Moreover,
 Clearly, $L$ is exactly the set of maximum points  of $g$ in $\tilde U :=P^{-1} (B_{\epsilon} (y)) $.

  For any $z\in B_{\epsilon} (y) \setminus \{y \}$ we have $|\nabla ^+ f|  (z) \geq \frac 1 2$.  By the definition of gradient curves, the point $y$ is the unique fixed point  of the partial gradient flow $\Phi$ of $f$ on $B_{\epsilon}$.
  Moreover, for $$\delta := 2 \cdot \inf _{z\in B_{\epsilon} (y)}  ( f(y)-f(z)) \;,$$
  any flow line of $\Phi$ defined at least  for time $\delta$ ends in $y$.

  Let $U =U_{\delta}  \subset \tilde U$ be the set of points $p\in \tilde U$ at which the flow line of $\hat \Phi$  on $\tilde U$ is defined at least for the time $\delta$.  Since $L$
  is the set of fixed points of $\hat \Phi $ in $\tilde U$, the set $U$ is an open neighborhood of $L$.   By construction and Lemma \ref{lem: liftflow},
  $\hat \Phi _{\delta} (p) \in L$, for any $p\in U$. In particular, $\hat \Phi$ is defined in $U$ for all times and $\hat \Phi _{\delta} (U) =L$.

  By Subsection \ref{subsec: Ambrosio},  $|\nabla ^+ g| (q) \geq \frac 1 2$, for any
  $q\in U\setminus L$.
  % Moreover,
  %
  %By construction, $U$ is invariant under $\hat \Phi$ and by Subsection \ref{subsec: lift},
  % the gradient lines of $g$ are exactly   horizontal lifts of  gradient lines of $f$.
%
   %For $K\subset L$ compact, find some $r>0$ such that $B_r(K) \subset  U$ is compact.
   %We find a neigborhood $O_r \subset P(U)$ of $y$ such that the gradient lines
   %$\eta ^{z}$ of $f $ starting at points $z$ of $O_r$ (and ending in $y$) have length less than $r$.
%
   %For any $x\in K$ and $z\in O_r$, we find a horizontal lift of $\eta ^{z}$ ending  at $x$, by Lemma \ref{lem: lift}.  By   Subsection \ref{subsec: Ambrosio}, this horizontal lift  is  a gradient line of $g$
   %starting on $P^{-1} (z)$. Thus, $\hat \Phi _{\delta} (P^{-1} (z) )$ contains $K$.
%
  %If $X$ is complete then,   by construction and Lemma \ref{lem: liftflow},
  %$O = P(U)$ is the set of points $z\in B_{\epsilon} (y)$, for which the flow lines of $\Phi $ are contained in $B_{\epsilon} (y)$. Again by Lemma \ref{lem: liftflow}, $U=P^{-1} (O)$.
%
  %The  final argument  in the non-complete case and the completeness of $X$ imply that $\hat \Phi _{\delta} :P^{-1} (z)\to L$ is surjective, for all $z\in O$.
%  and the last statement follows by construction.
  % By construction, $\Phi _{\delta } (z)=y$, for all $z \in O$.
  % From  Lemma \ref{lem: liftflow}, the set $O$ is invariant under the flow $\Phi$ and  $\hat \Phi _{\delta} $ sends $U$ to $L$.
  %In particular, the gradient flows $\hat \Phi$ on $U$ and the $\Phi $ on $O$ are defined for all times.
   \end{proof}
%{\red \begin{rem}If $X$ is complete then in the above Lemma we can choose $O$ to be a convex neighborhood of $y$ equal to the the superlevel set $\{f\ge f(y)-\eps\}$ for small $\eps$.
%\end{rem}
%}

% \begin{cor}
% 	Let $P:X\to Y$ be a local submetry between Alexandrov space. For any
% 	$y\in P(X)\subset Y$ and any compact subset $K $ in the fiber $L=P^{-1} (y)$, there exists a neighborhood $O$ of $y$ in $Y$ with the following property.
% \end{cor}
\subsection{Extremal subsets and dimension of fibers}
While subsequent results have  local versions, we prefer to state them only for "global" submetries, for the sake of simplicity. The first result follows directly from the definition of extremal subsets, via gradient flows of
semiconcave functions, Lemma \ref{lem: compose} and Lemma \ref{lem: liftflow}.
 An alternative more direct proof of the following statement can be found in
\cite{Guijarro}:

\begin{prop}
	Let $P:X\to Y$ be a submetry between Alexandrov spaces.
	Let $E\subset X$ be an extremal subset. Then  the image $P(E)$
	is an extremal subset of $Y$.
\end{prop}

However, extremal subsets  in the quotient are  often much more numerous than in the total space.    For instance:

% \subsection{Coarea formula revisited}
% As a first application of lifts of Perelman's function we show the following
% statements about semicontinuity of Hausdorff dimension of fibers:

  \begin{lem} \label{lem: strataext}
  	Let $P:X\to Y$ be a  submetry between Alexandrov spaces.
  	Then, for any $k\geq 0$, the set $Y(k)$ of points $y\in Y$ such that
  	$\mathcal H^k (P^{-1} (y)) =0$ is an extremal subset of $Y$.
  \end{lem}

  \begin{proof}
  	Fix $y\in Y(k)$ and   a special semiconcave function
  	$f$ on $Y$ with gradient flow $\Phi$.  Then the gradient flow $\hat \Phi$
  	of $g=f\circ P$   is defined for all times and sends fibers of $P$ surjectively onto fibers of $P$, { due to Lemma \ref{lem: compose} and  Lemma \ref{lem: liftflow}.   Since this gradient flow is locally Lipschitz,
  	$\mathcal H^k (P^{-1} (\Phi _t (y))) =0$, for all $t$.}
  	%we see that for any $t$ the $P$-fiber of $\Phi _t(y)$ has $0$ $k$-dimensional Hausdorff measure.
  	Thus, $Y(k)$ is invariant under $\Phi$.
   	\end{proof}

  Similarly, one shows that the set of points $y\in Y$ whose $P$-fibers have at most $k$ connected components are extremal subsets of $Y$.

  We can  draw the following consequence of the coarea formula:
% \begin{lem}
% 	Let
% \end{lem}

%\subsection{Dimension of fibers revisited}
%The surjectivity of the Lipschitz projections $\hat \Phi _{\delta}$ from neighboring  fibers derived in Proposition \ref{prop: Perelman} directly implies the following semicontinuity  statement about the Hausdorff dimension of the fibers:

%\begin{cor}
%	Let $P:X\to Y$ be a  local submetry between Alexandrov regions.
%	For  $y\in Y$ and $k>0$ assume
%  $\mathcal H^k (P^{-1} (y)) >0$.
   %has positive, respectively infinite $k$-dimensional Hausdorff measure,  for some $k\geq 0$,
 %   Then there is a neighborhood
%	$O$ of $y$ in $Y$ such that, $ \mathcal H^k (P^{-1} (z))>0$,
%	for any $z\in O$.
	% the fiber $P^{-1} (z)$ has Hausdorff dimension at least $k$.
%\end{cor}

%We can apply the lifts of Perelman function to slightly improve the implications of the coarea formula

\begin{cor} \label{cor: zeroset}
Let $X$ and $Y$ be $n$- and $m$-dimensional Alexandrov spaces and let  $P:X\to Y$ be a  submetry.
% from an  $n$-dimensional Alexandrov region to an $m$-dimensional Alexandrov region.
%  Then, for any
%$y\in Y$ and any relatively compact subset $U$ of $X$, the fiber
%$P^{-1} (y) \cap U$ has finite $\mathcal H^{n-m}$-measure.
%
A Borel subset $B\subset Y$ satisfies $\mathcal H^m (B)=0$ if and only if $\mathcal H^n (P^{-1} (B))=0$.	
\end{cor}

\begin{proof}
	If $\mathcal H^m (B)=0$ then $\mathcal H^n (P^{-1} (B))=0$ as one directly sees from the coarea formula \eqref{eq: coarea}.
	
	To prove the other implication we only need to show,
	due to \eqref{eq: coarea}, that the set $Y(n-m)$ defined in  Lemma \ref{lem: strataext} has $\mathcal H^m$-measure  $0$.
	But by   Lemma \ref{lem: strataext}, the set $Y(n-m)$ is an extremal subset of $Y$. Thus, if it has positive measure, it must contain an open subset.
	Then, by the coarea formula, its $P$-preimage has $\mathcal H^n$-measure
zero in $X$ and contains an open subset, which is impossible.	
	\end{proof}

%\subsection{Extremal subsets}

%From the definition of extremal subsets, Lemma \ref{lem: compose} and Lemma \ref{lem: liftflow} we directly deduce, cf. \cite{Guijarro}:

%\begin{prop}
%	Let $P:X\to Y$ be a surjective  local submetry between Alexandrov regions.
%	Let $E\subset X$ be an extremal subset. Then  the image $P(E)$
%	is an extremal subset of $Y$.
%\end{prop}

%However, extremal subsets  in the quotient are  often much more numerous than in the total space. One can obtain  versions of subsequent results also for local submetries, but  the statements in the global case, when all gradient flows are defined for all times are much clearer.

\section{Infinitesimal submetries} \label{sec: infinite}
\subsection{Lines and rays}
From Toponogov's splitting theorem we  derive:

\begin{prop} \label{prop: line}
	Let $X,Y$ be Alexandrov  spaces with non-negative curvature, let
	$P:X\to Y$ a submetry and assume that $Y$ splits  as $Y=Y_0\times \R$.
	Then there exists an isometric splitting $X=X_0\times \R$, such that $P$
	is given as $P(x,t)=(P_0(x),t)$ for a submetry $P_0:X_0\to Y_0$.
\end{prop}

\begin{proof}
	Consider the composition $Q$ of $P$ and the projection $Y\to \R$. This is a submetry. Lifting $\R$ to a horizontal line in $X$ and using the splitting theorem, we see that $X$ splits as $X_0\times \R$ such that $Q$ is the projection onto the second factor.
	
	Therefore, $P$ has the form $P(x,t)=(P^t(x),t)$ for a map
	$P^t:X_0\to Y_0$,  a priori, depending  on $t$.
	However, since $P$ is $1$-Lipschitz, $P^t$ does not depend on $t$ and equals a   map $P_0:X_0\to Y_0$.
	Since $P$ is a submetry,  $P_0$ must be a submetry as well.
\end{proof}

The first statement of the next observation follows from \eqref{eq: dA}, the second from the concavity of Busemann functions in non-negative curvature  or by a direct comparison argument:

\begin{lem} \label{lem: nonnegative}
	Let $X$ be a non-negatively curved Alexandrov space and let  $P:X\to [0,\infty )$ be a submetry and set  $L = P^{-1} (0)$.
	Then $P=d_L$.
	% is given as  the {\red signed} \footnote{\blue What do you mean here by signed? It is not a map to $\R$} distance function to $L$.
	The function $P$  is convex, in particular, $L$ is a convex subset of $X$.
\end{lem}

\subsection{Infinitesimal submetries, horizontal and vertical vectors}
In this section we  call points of a Euclidean cone $X=C(S)$ vectors {
	and denote, for   $h\in X$,   by $|h|$ the distance $|h| := d(h,0_x)$.}

We will call a submetry $P:X =C(S)\to Y =C(\Sigma)$ between two Euclidean cones over Alexandrov spaces $S,\Sigma$ of curvature $\geq 1$
an \emph{infinitesimal submetry}  if it commutes with the natural dilations of the cones. In other words, if the equidistant decomposition
of $X$ induced by $P$ is equivariant under dilations.  Equivalently, we can require that $P$ sends the origin  of the cone $X$  to the origin of $Y$ and  coincides with its own differential at the origin $0_X$.

Given such an infinitesimal submetry $P:X =C(S)\to Y =C(\Sigma)$ we define a vector $v\in X$ to be \emph{vertical} if $P(v)=0_Y$.
We call $h\in X$ \emph{horizontal}  if  $|P(h)| =|h|$.  The set of vertical and horizontal vectors in $X$ will be usually denoted by $V$ and $H$, respectively.
Clearly $V$ and $H$ are subcones of $X$.  By definition,  $V=P^{-1} (0_y)$.

By \eqref{eq: dA}, for any   point $h \in X$, we have
$$|P(h)| = d(0_Y,P(h))= d(V,h) \;.$$
This immediately implies:
\begin{lem}\label{lem-hor-polar-vert}
If $V=\{0\} $ then $H=X$.
If $V\neq \{0\}$  then  the set of unit horizontal vectors $H\cap S$ is the \emph{polar set} of $V\cap S$, i.e., the set of all points $x\in S$ such that the distance in $S$ between
$x$ and any point in $V\cap S$ is at least $\frac \pi 2$.
\end{lem}

%From Lemma \ref{lem: nonnegative} and the fact that the \emph{polar set} of any subset in an Alexandrov  space of curvature $\geq 1$ is convex we deduce that $H$ and $V$ are convex subcones of $X$. In the case that  $X$ is a Euclidean space this is equivalent to the following proposition:
%In general Alexandrov spaces, the statement of the Proposition is stronger is much stronger.

The following proposition is essentially contained in \cite{Lysubm}. For convenience of the reader we include a proof here.

\begin{prop} \label{prop: structure2}
	Let $P:X=C(S)\to Y=C(\Sigma)$ be an infinitesimal submetry.  Then the vertical and horizontal spaces $V,H$ are convex subcones of $X$.   The distance functions $d_V$ and  $d_H$ are convex functions on $X$. The foot-point projections    $\Pi ^V$ and $\Pi ^H$ from $X$ on $V$ and $H$, respectively, are well defined and  $1$-Lipschitz.

We have $P=P\circ \Pi^{H}$. The restriction $P:H\to Y$ is a submetry,
which is a cone over the submetry $P:H\cap S\to \Sigma$.
\end{prop}

\begin{proof}
 If $V=\{0\} $  the proposition follows by Lemma \ref {lem-hor-polar-vert}.

Suppose $V\ne \{0\} $.
  We claim that  any $x \in X\setminus (V\cup H)$ is contained in a
  unique subcone $C(\Gamma)$, where $\Gamma$ is a geodesic  of length $\frac \pi 2$ in the unit sphere $S$ of $X$  connecting a point in $H \cap S$ with a point in
  $V\cap S$.

   The uniqueness of $\Gamma$ follows from the fact that $H\cap S$ and $V\cap S$
   have distance $\frac \pi 2$ in $S$ and that geodesics in $S$ do not branch.

   In order to find such a quarter-plane $C(\Gamma)$ bounded by a vertical and a horizontal radial rays, consider the radial ray $\gamma$ in $Y$ through $P(x)$
   and its unique horizontal lift $\bar \gamma$ through $x$, Subsection \ref{subsec: liftgeod}.

   Then $v=\bar \gamma (0)$ is contained in $V$ and, by horizontality of $\bar \gamma$,  we have
   $$d(x,v) =d(x,V)=|P(x)| \;.$$
   Therefore, the radial vertical ray $\eta $ through $v$ meets $\bar \gamma$
   at $v$ orthogonally.

  	 Since $P$ commutes with dilations,  the radial rays $\bar \gamma _t$ through $\bar \gamma (t)$ converge, as $t\to 0$, to a radial ray $\bar \gamma _{\infty}$  such that
  	 $$P \circ  \gamma _{\infty} (t)=P \circ \bar \gamma (t)= \gamma (t) \;.$$
  	 Moreover, this ray $\bar \gamma _{\infty}$ makes  an angle of $\frac \pi 2$ with $\eta$ and the union of the rays $\bar \gamma _t$ is the required quarter-plane spanned by $\bar \gamma _{\infty}$ and $\eta$.
	   	
  	 By construction, $P(x) = P (h) $, where $h\in H$ is the closest point
  	 on $ \bar \gamma _{\infty}$ to $x$.  Moreover, $h$ is the closest point
  	 to $x$ in $H$ (otherwise, we would find a horizontal vector making an angle less than $\frac \pi 2$ with $\eta$).

  	Now, we easily see
  	$$d_H(x)= \sup _{w\in V\cap S } \{-b_w  (x) \} \;,$$
   where $b_w$ is the Busemann function of the radial ray through $w$.
   % and where the supremum is taken over all $w\in V$.
  	
  	 If we replace $-b_w$ by $\max \{-b_w ,0 \}$ then the equality remains to hold on $V$ and $H$ as well.  Thus, we have represented the distance function $d_H$ as a supremum of convex functions, since any Busemann function is concave in non-negative curvature.  This proves the convexity of $d_H$.
  	
  	 The convexity of $d_V$ is proved similarly (or directly by Lemma \ref{lem: nonnegative}).
  	
  	 Now, the gradient flows of $-d_H$, respectively of $-d_V$, converge
  	 to the closest point projection $\Pi^{H}$ and $\Pi^{V}$. The contractivity of such gradient flows proves that $\Pi^H$ and $\Pi ^V$ are well-defined and $1$-Lipschitz.

%\end{proof}

%\subsection{Infinitesimal submetries, structure}
%In the course of the proof of  Proposition \ref{prop: VH} we have in fact seen the proof of the first  statement in the following Proposition.
%We conclude:

%\begin{prop} \label{prop:  structure2}
%Let $P:X =C(S)\to Y =C(\Sigma)$ be an infinitesimal submetry with vertical and horizontal cones $H,V\subset X$.
%Then  $P$ commutes with projection to $H$, $P \circ \Pi ^{H } =P $.	The restriction $P:H\to Y$ is a submetry,
%which is a cone over the submetry $P:H\cap S\to \Sigma$.
%\end{prop}

The statement that $P=P\circ \Pi ^H$ has been obtained in the proof above
on $X\setminus (V\cup H)$. The equality is clear on $V\cup  H$ as well.

  Since $P$ is a submetry and $\Pi ^H$ is $1$-Lipschitz,  this equality implies that the restriction $P:H\to Y$ is a submetry.
  %  the second claim follows from the first one.
\end{proof}

\subsection{Infinitesimal submetries of Euclidean spaces}
We now specialize the above structural results to the case $X=\R^n$.

\begin{prop}  \label{prop: maininfinite}
Let $P:\R^n\to Y$ be an infinitesimal  submetry to   a  Euclidean cone $Y=C(\Sigma )$.  Let $H \subset \R^n$ be the subcone of horizontal vectors as above. Let $ H^E\ni 0$ be the maximal Euclidean subspace of $H$.   If $H^E\ne\{0\}$ then the restriction  $P:H^E \to Y$ is a submetry and the restriction
  of $P$ to the unit sphere in $H^E$ is a submetry onto  $\Sigma$.
  % is a submetry.

  If $H$ is not a Euclidean space then there exists a round hemisphere
  $\mathbb S^k _+ \subset H$ such that the restriction
  $P:  \mathbb S^k_+ \to \Sigma$ is a submetry.
\end{prop}

\begin{proof}
 Let $V=P^{-1} (0 _Y) \subset \R^n$ be the vertical cone.  If $V=\{0\}$ then $H=\R^n$ and the claim of the Proposition is trivial.
 	
 	 Suppose $V\ne \{0\}$.
Then $H$ is the polar cone of $V$ by Lemma ~\ref {lem-hor-polar-vert}.

 Let $v\in V$ be  arbitrary.
 Then $T_vV$ is a convex subcone of the linear span $V^+$ of $V$. The cone $T_vV$ contains $V$, so that its polar cone $H_v$ is a subset of $H$.

The differential $Q=D_vP:T_v\R^n =\R^n \to C_{0_y} Y =Y$  is an infinitesimal submetry. Its vertical subspace is  $T_vV$ by Proposition  \ref{prop: differ};
hence $H_v$ is the horizontal subspace of the infinitesimal submetry $Q$.

By Proposition \ref{prop: structure2}, the map $Q:H_v \to Y$ is a submetry.
Identifying $H_v$ with the starting directions of $P$-horizontal rays in $\R^n$ starting at $v$, we see that under the canonical identification of $H_v$ with a subset of $H$, the map $Q$ is just the restriction of $P$ to $H_v$.

%Let $V^+$ be the linear span of $V$.
Choosing $v$ to be an inner point of $V$ in its linear span $V^+$, we get
$T_vV =V^+$ and   $H_v =H^E$  proving the first statement.
%is a Euclidean space, the orthogonal complement of $V^+$.

On the other hand,  $V\neq V^+$ if and only if $H \neq H^E$.
% is not a Euclidean space.
In this case, we can always find a point $v$ in the boundary of $V$ in $V^+$ at which $T_vV$ is a Euclidean halfspace.  Then $H_v$ is a Euclidean halfspace as well.

 Restricting to the unit spheres in the so obtained cones $H_v$,  we find the desired submetries with base $\Sigma$.
\end{proof}

 Recall, that any Alexandrov space $Z$ has a unique decomposition
$Z=\R^k \times Z_0$, where $Z_0$ does not admit any $\R$-factor, see  \cite{FL}.  If $Z$ is non-negatively curved, then $Z_0$ does not contain lines, by Toponogov's splitting theorem. Now
we state:

\begin{prop} \label{prop: directions}
	Let $P:\R^n \to C(\Sigma )$ be an infinitesimal submetry, let $H \subset \R^n$
	be the cone of horizontal vectors and let $C(\Sigma ) =\R^l \times C(\Sigma _0)$ be the canonical decomposition, so that $C(\Sigma _0)$ does not contain lines.

	There is a natural splitting
	$\R^n= H^0\times \R^{n-l}$ with $H^0=\R^l$ and
	$$P=(Id, P_0):\R^l \times \R^{n-l}\to \R^l \times C(\Sigma _0)\;.$$
	The set $H^0$ consists  of all points $h\in H$ such that  $P^{-1} (P(h) \cap H)=\{h\}$.
	
	The space $\Sigma _0$ has diameter at most $\frac \pi 2$.
%	
%	Further, 	$H^0$ is the set of all points $h\in H$ such that  $P^{-1} (P(h) \cap H)
%	\{h\}$.
%	
%	Then the set of points $h\in H$ at which  the fibers of the restriction  $P:H\to Y$ are singletons is a Euclidean subspace $\R^l$. The restriction $P_0$ of $P$ to the orthogonal complement is a submetry $P_0:\R^{n-l} \to C(\Sigma _0)$ and $P$ has a product structure
%	$$P=(Id, P_0):\R^l \times \R^{n-l}\to \R^l \times C(\Sigma _0)\;.$$
\end{prop}

\begin{proof}
The first statement  follows from	Proposition \ref{prop: line}. By definition of horizontal points, $H^0\subset  H$.
	
	 Using the product structure of $P$, we  reduce the statement to the case $l=0$, hence $\Sigma =\Sigma _0$.  Thus, we may assume $\diam (\Sigma ) <\pi$.  We then need to prove that the diameter of $\Sigma$ is at most $\frac \pi 2$ and that for all
	 $h\neq 0$ in $H$, the fiber through $h$ of the submetry $P:H\to C(\Sigma)$ has more than one point.

	 %Since $\Sigma =\Sigma _0$, the diameter of $\Sigma $ is less than $\pi$.
	 Due to Proposition \ref{prop: maininfinite}, there exists a submetry $\hat P:\mathbb S^k \to \Sigma $.
	
	  Assume that the diameter of $\Sigma $ is larger than $\frac \pi 2$. Consider $p,q\in \Sigma$ such that
	  $d(p,q) >\frac \pi 2$ equals the diameter of $\Sigma$.
	
	  Let $L_p=\hat P^{-1} (p)$ and $L_q =\hat P^{-1} (q)$. By \eqref{eq: dA}, $L_q$ is contained in the set of points in $\mathbb S^{k-1}$ with maximal distance to $L_p$, and since this maximal distance is larger than $\frac \pi 2$, we see that $L_q$ is a singleton.
	 Then the antipodal point of $L_q$ is also a fiber of $\hat P$, by the same argument. Therefore, $\diam (\Sigma )=\pi$  in contradiction to the assumption $\Sigma =\Sigma _0$.

Consider the restriction $P:H \cap \mathbb S^{n-1} \to \Sigma$ and assume that some fiber of this restriction is a singleton. By Proposition \ref{prop: maininfinite}, we find a
sphere $\mathbb S^k \subset H \cap \mathbb S^{n-1}$ such that the restriction
$P:\mathbb S^k \to \Sigma$ is  a submetry.  Clearly, this restriction has also a singleton fiber.  But this implies,  $\diam (\Sigma ) =\pi$.

 This contradiction  finishes the proof.	
\end{proof}

The above proof shows that if $P\co \Ss^n\to \Sigma$ is a submetry then either $\diam \Sigma=\pi$ or $\diam \Sigma\le\pi/2$, see also \cite{Mendes-Rad}, \cite{Grovesubmet}.

\section{Fibers have  positive reach} \label{sec: posreach}
%In order to describe many submetries of smooth manifolds with non-manifold fibers, recall the notion  of a subset of positive reach.
\subsection{Distance functions in manifolds}
 From now on let $M$ be  a Riemannian manifold
 with local two sided bounds on curvature. We will always equip $M$ with the
   $\mathcal C^{1,1}$ atlas of distance coordinates, \cite{Ber-Nik}.

For any subset $A\subset X$ the distance function $f=d_A$ is semiconcave on $M\setminus A$ and $f^2$ is semiconcave on $M$.   For an open subset
$O\subset M\backslash A$ the function $f$ is $\mathcal C^{1,1}$ in $O$ if and only if
for any $x\in O$ there exists  at most one geodesic $\gamma _x :[0,\epsilon ) \to O$ starting at
$x$  such that $f(\gamma (\delta)) =f(x)-\delta$ for all $0<\delta< \epsilon$,  \cite{KL}.
In this case, the gradient curves of $f$ are geodesics.
% {\red A similar statement holds for $f^2$ and an open subset $O\subset M$ except gradient curves in this case are reparameterized geodesics.} \footnote{\blue: What is this for? \red I don't remember now why I added it. I thought it this was used somewhere. If not this sentence can certainly be deleted.}

We denote by $\mathcal {U}(A)$ the largest open
subset of $M$ on which $d_A^2$ is $\mathcal C^{1,1}$.   Due to the previous considerations $\mathcal U(A)\setminus A$ is foliated by geodesics,
the gradient curves of $d_A$.

%{\red We will often use the following trivial Lemma
%\begin{lem}\label{lem-triv}
%Let $f\co U\to \R$ be a 1-Lipschitz semiconcave  function on an open set $U$ in an Alexandrov space $X$.
%Suppose $|\nabla f|=1$ on $U$. Then gradient curves of $f$ are geodesics along which $f$ increases with unit speed.
%\end{lem}
%}
% Moreover, $\mathcal U(A)$ contains $A$ if and only if $A$ is a set of positive reach.

\subsection{Positive reach} \label{subsec: posreach}
Recall that a  subset $L$ of a Riemannian  manifold
$M$ has \emph{positive reach}  if the foot-point projection on $L$ is uniquely defined in some open neighborhood $U$   of $L$, \cite{Federer}.
This is a local property, which is moreover independent of the Riemannian metric and any $\mathcal C^{1,1}$ submanifold has this property,  \cite{Bangert,KL}.

%\footnote{\red I added the KL reference because Bangert only deals with smooth metrics.}
A locally closed subset $L$ of $M$ has positive reach if and only if the open set
$\mathcal U(L)$  of points at which $d_L^2$ is $\mathcal C^{1,1}$ contains $L$.

% Moreover,
%for any set $L$ of positive reach in $M$ and any point $x\in L$ there exists some $\delta >0$, such that $\bar B_{\delta} (x)\cap L$ has positive reach in $M$,  \cite[Lemma 3.4]{Rataj}, \cite{Bangert}.

Structure of sets of positive reach is rather well understood. For any  set $L$
of positive reach,  the topological and the Hausdorff dimensions of $L$ coincide
and $L$
is locally contractible \cite[Remark 4.15]{Federer}.
%The topological and the Hausdorff dimensions of $L$ coincide
The intrinsic and the induced metrics on $L$ are locally equivalent and in the intrinsic metric the space $L$ locally has curvature bounded from above in the sense of Alexandrov,  \cite[Theorem 1.1]{Ly-reach}.

For all  $x\in L $, there exists  a well-defined \emph{tangent cone}  $T_xL$ which is a convex subset of $T_xM$, \cite[Theorem 4.8]{Federer}. The \emph{normal cone}
$T ^{\perp}_x L$ is the convex cone of all vectors in $T_xM$ enclosing angles at least $\frac \pi 2$ with all vectors in $T_xL$,  i.e. it is the polar cone of $T_xL$ in $T_xM$.
For any $x\in L$ and unit  $h\in T^{\perp}_xL$ the geodesic $\gamma _h$ starting  in the direction of $h$
satisfies $d_L (\gamma _h (s)) = s$, for all $s$ with $\gamma _h (s)\in \mathcal U(L)$, \cite{KL}.

Moreover, we have, \cite[Proposition 1.4]{Ly-conv}:

% and the topological and the Hausdorff
%We summarize  some geometric and topological properties of
%sets of positive reach, \cite[Remarks 4.15, 4.20]{Federer}, \cite[Theorem 1.1]{Ly-reach},  \cite[Proposition 1.4]{Ly-conv}, \cite{Rataj}:

\begin{prop} \label{cor: fiber}
	Let $L$ be a set of positive reach in  $M$.   Then
	%	\begin{enumerate}
	%		\item  $L$ is locally contractible.
	%		\item  The topological and the Hausdorff dimension of $L$ coincide.
	$L$  contains a  $\mathcal C^{1,1}$ submanifold $K$  of $M$ which is dense and open in $L$.
	%At all points $x\in L $, there exists  a tangent cone  $T_xL$ which is a convex subset of $T_xM$.
	
	%The set  $L$ itself	  is a $\mathcal C^{1,1}$ submanifold if and only if $L$ is a topological manifold.  This happens if and only if
	%all tangent spaces $T_xL$ are Euclidean spaces.
	
	 Moreover, the following are equivalent:
	\begin{enumerate}
	\item The set  $L$  is a $\mathcal C^{1,1}$ submanifold;
	\item The set  $L$  is a  topological manifold;
	\item All tangent spaces $T_xL$ for $x\in L$ are Euclidean spaces.
	\end{enumerate}
	
	%	\item The intrinsic and induced metrics on $L$ are locally biLipschitz.
	%	
	%	\item In the intrinsic metric, the space $L$  locally has curvature bounded from above in the sense of Alexandrov.
	%	\end{enumerate}
\end{prop}

However,  a connected subset $L$ of positive reach does not need to have locally constant dimension nor  does it have to admit a triangulation.
%The set of positive reach described in the next example can easily be seen to be the fiber $P^{-1} (0)$ of

\begin{ex} Let $L\subset \R^2$ be the set of points $(x,y)$ with $0\leq y\leq f(x)$, where $f:\R\to [0,\infty)$  is smooth with
	$f^{-1} (0)$
	being the union of the rays $(-\infty,0]$ and $[1,\infty) $ and a Cantor set $C\subset [0,1]$. Then $L$ has positive reach in $\R^2$.
%	 Such a set $L$ can be seen to be the  fiber $P^{-1} (0)$ of a submetry $P:H^3\to [0,\infty )$, where $H^3$ is the hyperbolic $3$-space, see Lemma \ref{lem: locsubm}.  Indeed, one can embed $L$ in $H^3$ so that it is contained in a totally geodesic $H^2$ and such that in this $H^2$ the set $L$  has arbitrary small "second fundamental form" hence, due to the negative curvature,
%	infinite reach.
\end{ex}

\subsection{Connection to submetries}
It is now not difficult to see that sets of positive reach are intimately related to submetries.
\begin{prop} \label{lem: locsubm}
	Let $L \subset M$ be a closed subset of positive reach,
	 nowhere dense in $M$.
	Then $P=d_L:\mathcal U(L) \to [0,\infty)$ is a local submetry with $L=P^{-1} (0)$.
\end{prop}

\begin{proof}
	 $P$ is $1$-Lipschitz, thus $P(B_r(x)) \subset B_r(P(x))$ for all $x\in \mathcal U(L)$.

{
%Let $\Pi\co U(L)\to L$ be the foot-point map.
Let $x\in \mathcal U(L)$ be such that $\bar B_r(x)$ is compact in $\mathcal U(L)$.

For $x\in \mathcal U(L)\setminus L$, we restrict $P$ to the maximal gradient curve $\gamma $ of $d_L$ through $x$, and use that $|(d_L\circ \gamma )'|=1$
to deduce $P(B_r(x))=B_r(P(x))$.

If $x\in L$, consider a unit vector $h\in T_x^{\perp}$, and the geodesic $\gamma =\gamma _h$ starting in the direction of $h$. Since $d_L\circ \gamma _h (t) =t$ for all $t\leq r$, we see
%
 % and dec By above  on $\mathcal U(L)\backslash L$ it holds that $P$ is $C^{1,1}$, $|\nabla P|=1$ and gradient curves of $P$ are unit speed geodesics along which $P$ increases with unit speed.
%Hence the restriction of $P$ to the (closure of) the maximal gradient line  of $P$ through $x$  shows that $P(B_r(x))=B_r(P(x))$. Moreover this also shows that $\Pi(\alpha(t))=\Pi(x)$  for all $t$.
%
%Now suppose $x\in L$, Let $v\in T_xL^\perp$ be unit. Let $\gamma(t)=\exp_x(tv)$. Then by above $P(\gamma(t))=t$ for all $0\le t\le r$ which also implies that
$P (B_r(x))=[0,r) =B_r(P(x))$, as desired.}
%
%
%
%
 %Then by the previous case there is $\delta>0$ such that the gradient flow $\Phi_t$ of $P$ starting at $x_i$ exists for all $i$ and any $0\le t\le \delta$. Then $y_i=\Phi_\delta(x_i)$ subconverges to a point $y$ such that $d(y,L)=\delta$. By above $\Pi(y_i)=\Pi(x_i)\to x$ and hence $\Pi(y)=x$, i.e.
%$x$ is the closest point to $y$ in $L$. Then $P$ increases with unit speed along the geodesic connecting $x$ to $y$ which implies that $P (B_r(x))=[0,r) =B_r(P(x))$ as desired.
%
%
%
 % If $x\in L$, there is a unit normal vector $h\in T_x^\perp L$, since $L$ is nowhere dense in $M$. Let $\gamma(t)=exp(tv)$ be the geodesic starting in the direction $v$. We claim that $P(\gamma(t))=t$ for all sufficiently small $t$. Indeed, for all small $t$ we have that $\gamma(t)\notin L$ and hence by the previous case the gradient flow of $P$ starting at $\gamma(t)$ exists for some uniform time $\delta>0$, is given by a geodesic starting t $\gamma(t)$ and $P(\Phi_s(\gamma(t))=P(\gamma(t))+s$ for any small $t$ and all $0\le s\le \delta$. Then
%
%
  %Restricting $d_L$ to the geodesic $\gamma _h$, we see $P (B_r(x))=[0,r) =B_r(P(x))$. {\red V: why?}
%
%
\end{proof}

\begin{rem} \label{rem: global}
Changing the metric in a neighborhood  of $L$ by some conformal factor depending  on  the distance function $d_L^2$, it is easy to see the following. Any \emph{compact} subset $L$ of positive reach in $M$ admits a complete metric in a neighborhood $U$ of $L$ in $M$, such that this metric  still has  local two sided curvature bounds and such that the distance function to $L$ in this metric is a (global) submetry $P:U\to [0,\infty)$.  
 This should be possible for noncompact $L$ as well.
\end{rem}

  The following converse  contains Theorem \ref{thm: leaf} as a special  case:

 \begin{thm}\label{thm-fibers-pos-reach}
 	Let $M$ be a  Riemannian manifold and let $P:M\to Y$ be a surjective local submetry.
 	Then any fiber $L=P^{-1} (y)$ is a subset of positive reach in $M$. 	
 \end{thm}

 \begin{proof}
 Applying Proposition \ref{prop: Perelman}, we find a neighborhood $U$ of $L$ and a semiconcave function $g:U\to \R$, which has $L$ as its set of maximum points. Moreover, $|\nabla ^+ g|  (p) >\frac 1 2$ for $p\in U\setminus L$.

 Thus, $L$ is a regular sublevel set of the semiconvex function $-g$, in the sense of \cite{Bangert}, and by the main result of \cite{Bangert}  (see also \cite{KL}) $L$ has positive reach in $M$.
 %{\red (\cite{Bangert} applies to sufficiently smooth metrics; see ~\cite{KL} for the case of a general manifold with two sided curvature bounds)}.
 \end{proof}

 \subsection{Examples}
A  subset $L$ of positive reach in $M$ has \emph{reach $\geq r$}
 if the set $\mathcal U(L)$ contains
the $r$-tube $B_r(L)$ around $L$.

%If, in addition, $L$ is nowhere dense and $\bar B_r(L)$ is complete in $M$ then
%Proposition \ref{lem: locsubm}
% (and its proof)
%shows that $d_L:B_r(L) \to [0,r)$
%is a submetry.

%Assume in addition that $\bar B_r (L)$ is complete. Then $L$ has \emph{positive reach $r>0$}
%if and only if   any point $x\in B_r(L) \setminus L$ lies on a geodesic $\gamma :[0,r ) \to M$ starting on $L$ and such that
%$d_L (\gamma (t))=t$ for all $t\in [0,r)$.

%The above equivalence shows that the property of having positive reach $r>0$ can be interpreted as follows:
%\begin{lem}
%	A closed subset $L$ in a  complete Riemannian manifold $M$ has positive reach $r >0$ if the distance function $d_L :B_r(L)\setminus L \to (0,r)$ is a submetry.
%\end{lem}

%By a simple continuity argument, we deduce from this that
%Theorem \ref{thm: leaf} cannot be improved:
%\begin{lem} \label{lem: locsubm}
%	A closed subset $L$   in a smooth complete Riemannian manifold $M$ is nowhere dense in $M$ and  has positive reach $r >0$ if and only if%
%	the distance function $d_L :B_r(L) \to [0,r)$ is a submetry.
%\end{lem}

Closed convex sets in $\R^n$ are exactly sets of reach $\infty$ in $\R^n$.  By Proposition \ref{lem: locsubm} this implies:
\begin{ex}
	Let $C\subset  \R^n$ be a closed subset.  The distance function $d_C:\R^n\to [0,\infty )$ is a submetry  if and only if  $C=f^{-1} (0)$
	is convex and nowhere dense in $\R^n$.
	%  For any point $x\in C$, the vertical space $V_x$ is the tangent space $T_xC$. Thus, unless $C$ is an affine subset, there are points at which the vertical space $V_x$ (and also the horizontal space $H_x$) are non Euclidean spaces.	
\end{ex}

If $L$ is a co-oriented $\mathcal C^{1,1}$ hypersurface in $M$ then we can consider in Proposition \ref{lem: locsubm} the oriented distance function instead of the non-oriented one and see:

%Any set of positive reach provides a submetry as follows:

%\subsection{Smoothing}% We have seen in Lemma \ref{lem: locsubm} that  any nowhere dense set $L$ of positive reach $r>0$ in a smooth manifold $M$ defines a submetry $P=d_L: B_r(L) \to [0,r)$.
% We are going to change the metric and prove the second part of Theorem \ref{thm: leaf}.

%\begin{lem}
% Let $L$ be a nowhere dense set $L$ of positive reach $r>0$ in a smooth complete Riemannian manifold $(M,g)$. Let  $U$ be the open subset $U=B_r(L)$.
% Then there exists a smooth metric $g_0$ on $U$, such that $(U,g_0)$ is complete and the distance function $f=d_L:U\to [0,\infty)$ is a submetry with respect to this metric.

% subset $U$ be the neighborhood
%\end{lem}

\begin{ex} \label{ex: twoside}
	If $L$ is a   $\mathcal C^{1,1}$ hypersurface in $M$ with oriented normal bundle  then  the signed distance $P$ to $L$
	defines a local submetry $P:\mathcal U(L) \to \R$.
\end{ex}

The following example is a prominent theorem in the theory of non-negative curvature:

\begin{ex}
	Let $M$ be a complete open manifold of nonnegative sectional curvature. Then any soul $S$ of $M$ is a subset of reach $\infty$ and the distance function $d_S:M\to [0,\infty)$  is a submetry, \cite{ Per-soul}, \cite{Wilking}.
\end{ex}

The next examples provide large infinite-dimensional families of submetries
$P:H^2\to \R$ and $P:\mathbb S^4 \to [0,\frac \pi 2]$ with non-smooth leaves.

\begin{ex} \label{ex: twoside+}
	If $H^2$ is the hyperbolic plane and $L\subset H^2$ is a complete $\mathcal C^{1,1}$-curve with geodesic curvature bounded by $1$ at every point.
	Then the signed distance function to $L$ is  a submetry $P:H^2\to \R$.
\end{ex}

\begin{ex}  \label{ex: infinite}
	Let $C\subset  \mathbb S^3$ be a convex subset without interior points.
	Let $L=\mathbb S^0\ast C \subset \mathbb S^4$ be the suspension of $C$.
	Then the distance function $d_L:\mathbb S^4\to [0,\frac \pi 2]$ is a submetry.
\end{ex}

\section{Local structure of the base} \label{sec: onefiber}

%\subsection{Fibers have positive reach}
%We can now prove a generalization of Theorem \ref{thm: leaf}.
%\begin{thm}\label{thm-fibers-pos-reach}
%	Let $M$ be a  Riemannian manifold and let $P:M\to Y$ be a surjective local submetry.
%	Then any fiber $L=P^{-1} (y)$ is a subset of positive reach in $M$. 	
%\end{thm}

\subsection{Injectivity radius in the base}
We get a local  version of Theorem \ref{cor: quotient}, (i), (iii):
\begin{thm} \label{cor: posreach}
	Let $M$ be a Riemannian manifold and $P:M\to Y$  a surjective local submetry. For any  $y\in Y$, there exists   $r >0$, such that  any  $v\in \Sigma _yY$ is the starting direction of a geodesic   of length $r$ in $Y$.
	
	For any $s<r$ the distance sphere $\partial B_s (y)$ is an Alexandrov space.
\end{thm}

\begin{proof}
{We may assume that	 $y$ is not isolated. Then $L=P^{-1}(y)$ is a nowhere dense set of positive reach in $M$ and the distance function  $d_L$ has ascending slope $1$ at all points in the neighborhood $\mathcal U(L)$ of $L$.
	
	For   $x\in L$ consider $r>0$ such that  $\bar B_{3r} (x)$ in $\mathcal U(L)$ is compact. Then, by Subsection \ref{subsec: Ambrosio} and Corollary
	\ref{cor: locglob}, the distance function $d_y$ has ascending slope $1$
	at all points $z\in B_{r} (y)$. Thus, $d_y$ is semiconcave, 1-Lipschitz and  $|\nabla d_y|=1$ on $B_{r} (y)\setminus \{y\}$. Therefore its maximal gradient curves in $B_{r} (y)\backslash \{y\}$ are unit speed geodesics of length $r$ starting at $y$.
	Thus, any $v\in \Sigma _y$ is the starting direction of a  geodesic $\gamma _v$ of length $r$.}
	
%	Choose some $x\in P^{-1} (y)$
%	and consider a standard local picture $P:B_{3r} (x)\to B_{3r} (y)$ as above.    Since $L= P^{-1} (y)$ has positive reach by Theorem~\ref{thm-fibers-pos-reach}, there exists some $r>\epsilon >0$ such that the  distance function
%	$d_L$ has ascending slope $1$ at any point  in $B_{3\epsilon} (x)$.
	
	%Thus, also $d_y$ has ascending slope  $1$ at all points of $B_{\epsilon} (y)$.
	%This implies the statement about geodesics.
	
	For $s<r$ we set
		$$N^s=P^{-1} (\partial  B_s (y)) \cap B_{3r} (x) \;.$$
Then $P:N^s \to \partial B_s(y)$ is a surjective local submetry.

On the other hand, $N^s$ is the level set $d_L^{-1} (s)$  of the $\mathcal C^{1,1}$-submersion $d_L :B_{3r} (x)\setminus L \to \R$.  Thus, $N^s$ is a
$C^{1,1}$-submanifold of $M$. Hence in its intrinsic metric, $N^s$ has curvature locally bounded from both sides, {see  \cite{KL}}.

 By Proposition \ref{prop: cbb} and Corollary \ref{cor: length},
$\partial B_s$ is an Alexandrov region. Due to compactness, $\partial B_s(y)$ is an Alexandrov space.
\end{proof}	

\subsection{Infinitesimal structure}
Let $P:M\to Y$  be a local submetry,  let $x\in M$ be arbitrary and set $y=P(x)$. By Proposition \ref{prop: differ}, the
differential $D_xP: T_xM\to T_y Y$  is an infinitesimal submetry.
%If $P$ is fixed
 We will denote vectors in $T_xM$ which are vertical, respectively horizontal, with respect to $D_xP$, { \emph{vertical}, respectively \emph{horizontal},  with respect to $P$. }

 By  Proposition \ref{prop: differ},
the vertical space $D_xP ^{-1} (0_y)$ is exactly the tangent space $T_xL$, where $L:=P^{-1} (y)$ is the fiber of $P$ through $x$.

By Lemma~\ref{lem-hor-polar-vert}, a vector $h\in T_xM$ is a horizontal vector   if and only if it is contained in the \emph{normal space}
$T_x ^{\perp} L$ of the set of positive reach $L$ at the point $x$.
 %that is if and only if the angle between $h$ and any non-zero vector in $T_xL$ is at least $\frac \pi 2$.  This condition is equivalent to the fact that
%$$d(L,\exp (sh) )= d(x,\exp sh)=|sh|\,$$
%for all $s<r$, for some $r$ depedning only on $x$.
Note that the convex cone $T_x^{\perp} L$ is a Euclidean space if and only if $T_xL$ is a Euclidean space.

%We obtain a
%directly deduce the following
%local generalization    of Theorem \ref{thm: link}:
\begin{thm} \label{thm: halfspace}
	Let  $M$ be an  $n$-dimensional Riemannian manifold and $P:M\to Y$  a surjective  local submetry.
If $y\in Y$ is non-isolated	then
	 there exists a submetry
	$P':\mathbb S^k \to \Sigma _y$, for some $k < n$.

If   $L=P^{-1} (y)$ is not a $\mathcal C^{1,1}$-submanifold of $M$ then,
for some $k< n$, there
exists a submetry of the closed hemisphere $\mathbb S^k _+$ onto $\Sigma _y$.
\end{thm}

\begin{proof}
 Consider an arbitrary $x\in L$. If $L$ is not a $\mathcal C^{1,1}$ submanifold of $M$, we may choose $x$ such that
 the vertical space $T_xL$ is not a Euclidean space, by  Proposition \ref{cor: fiber}. In this case the horizontal space $T_x^{\perp} L$ is not a Euclidean space as well.

The differential $D_xP:T_xM\to T_yY$ is an infinitesimal submetry,  Proposition \ref{prop: differ}. The claim now follows from Proposition \ref{prop: maininfinite}.
\end{proof}

The first part of the above theorem contains as a special case Theorem \ref{thm: link}. Now, Proposition
\ref{prop: directions} implies  Corollary \ref{cor: directions}.

%In addition we get:
%\begin{lem}  \label{lem: bound}
%Let $P:M\to Y$ be a   local submetry, where $M$ is an   $n$-dimensional Riemannian manifold.
%If for some $y\in Y$ the fiber $L=P^{-1} (y)$ is not a $\mathcal C^{1,1}$ submanifold of $M$ then
%for some $k\leq n$ , there
%exists a submetry of closed hemisphere $\mathbb S^k _+$ onto $\Sigma _y$.
%\end{lem}

%\begin{proof}
%By Proposition \ref{cor: fiber}, there exists a point $x\in L$ such that
%$T_xL$ is not a Euclidean space.

%By Proposition \ref{prop: differ}, the differential $D_xP:T_xM\to T_yY$ has as
%its horizontal space $T_xL^{\perp}$ which is not a Euclidean space.

%Now the statement follows from Proposition \ref{prop: maininfinite}.
%\end{proof}

\subsection{Standard local picture} \label{subsec: standard}
 %We introduce some notation.
Let $M$ be a Riemannian manifold  and $P:M\to Y$ be a local submetry.
For a point $x\in M$ we denote by $L_x$ the fiber $P^{-1} (P(x))$.
As before let $\mathcal U(L_x)$ be  the open neighborhood of $L_x$ on which $d_{L_x} ^2$ is $\mathcal C^{1,1}$.
Consider  $r>0$ { such that   the set  $\bar B_{10r} (x)$ is a compact, convex  subset of $\mathcal U(L_x)$. Moreover, we require that geodesics in
	$\bar B_{10r} (x)$ are uniquely determined by their endpoints and such that
%Moreover, we choose $r$ so that
%$\bar B_{10r} (x)$ is compact, strictly convex and uniquely geodesic.
$P(\bar B_{10r} (x))$ is contained in a compact Alexandrov space $Y'\subset Y$}.

Then, for $y=P(x)$, the restriction $P:B_{10r} (x) \to B_{10r} (y)$ will be called the \emph{standard local picture at $x$}.
Eventually, we will later adjust the choice of $r =r_x$.
Note that  $r=r_x$ in the standard local picture around $x$ satisfies the statements of Corollary \ref{cor: posreach}.

%$P:M\to Y$, where $M$ is a smooth Riemannian manifold.  For any $x\in M$
%we fix some $r=r_x>0$ (to be adjusted later)  with the following properties:

%The ball $\bar B_{3r}  (x)$ is compact, strictly convex and uniquely geodesic;
%the ball $B_{3r}(P(x))$ is contained in a compact Alexandrov space $Y'\subset Y$.

%We  call such a situation a \emph{standard local picture}, but will adjust  $r_x$ in the course of the section, so that further  properties  hold true.

\begin{prop}  \label{prop: horvector}
	Let $P:M\to Y$ be a local  submetry. For any $x\in M$, let $r=r_x$ be as in the standard local picture around $x$.
	{For a unit vector $h\in T_xM$  and} the geodesic $\gamma _h$ starting in the direction of $h$ the following  are equivalent:
	\begin{enumerate}
		\item  \label{hor-vector}The vector $h$ is horizontal.
		\item  \label{hor-curve} The geodesic $\gamma _h :[0,r]\to M$     is a horizontal curve.
		\item  \label{p-min-geod} The geodesic $\gamma _h:[0,r]\to M$ is $P$-minimal.
	\end{enumerate}
\end{prop}

\begin{proof}
Set $y=P(x)$ and set $L=P^{-1} (y)$.
%Choose $r$ as in   Corollary \ref{cor: posreach} and such that
% $\bar B_{3r} (x)$ is compact in $\mathcal U(L)$.
 %  as in the standard local picture, moreover, such that $r<\epsilon$ in Corollary \ref{cor: posreach}.

{ If $h$ is  horizontal,  then $h\in T_x^{\perp} L$ and  $d(\gamma _h (r), L)= r$.
 Hence, $P\circ \gamma _h$ is a geodesic starting in $y$.
Thus,\eqref{hor-vector} implies \eqref{p-min-geod} .

Clearly \eqref{p-min-geod}  implies \eqref{hor-curve} and \eqref{hor-vector} .  }

Given  \eqref{hor-curve} , the image $P\circ \gamma _h$  is a quasi-geodesic of length $r$ starting in $y$, by Proposition~\ref{prop: image}.
By Corollary \ref{cor: posreach}, there exists a geodesic in $Y$ of length $r$ with the same starting vector as  $P\circ \gamma _h$.
Due to \cite{PP-quasi}, the quasigeodesic $P\circ \gamma _h$  coincides with this geodesic, proving \eqref{p-min-geod}.
\end{proof}

The  last argument
in  the above proof also shows that quasigeodesics in  $Y$ are of a  much more special form than in general Alexandrov spaces:
\begin{cor} \label{cor: quasigeod}
	Let $P:M\to Y$ be a local surjective submetry.  If $\gamma :[a,b]\to Y$
	is a quasigeodesic then there exists a finite subdivision $a=t_1\leq t_2 ...\leq t_k=b$ such that the restriction of $\gamma$ to any of the subintervals $[t_i,t_{i+1}]$ is a geodesic.
\end{cor}

\section{Some technical statements}  \label{sec: tech}
\subsection{Setting for semicontinuity questions} \label{subsec: notat}
In this section we fix  a local submetry $P:M\to Y$.  As before we denote by $L_x$ the fiber $P^{-1} (P(x))$ through a point $x\in  M$.
We are going to analyze the (semi)-continuity of vertical spaces $T_xL_x$, as $x$ varies over $M$.

 The following example should be seen as a warning:
\begin{ex}
For $C=[0, \infty)\subset \R\subset \R^2$, consider the submetry  $P=d_C:\R^2 \to[0,\infty)$.
Then $C=L_{(0,0)}$ and $T_{(0,0)} L_{(0,0)} =C$.  On the other hand, for $x= (-t,0) \in \R^2$, the vertical space
 $T_x L_x =\{0 \} \times \R$ is orthogonal to $C$.
\end{ex}

%Let the submetry $P:\R^2 \to[0,\infty)$ be given by the distance function to a ray $P=d_{\Gamma} $, with $\Gamma = [0, \infty)\subset \R\subset \R^2$.
%Then  the vertical space $T_0 L$ at the origin $0\in \Gamma$ is just $L_0=\Gamma$. On the other hand for any $x= (-\epsilon,0) \in \R \subset \R^2$, the
%vertical subspace $T_x L_x$ is the vector space $\{0 \} \times \R$ orthogonal to the linear hull of $\Gamma$. Thus at the origin $0$  there is absolutely no (semi-) continuity
%of vertical and horizontal spaces.

%All subsequent statements are based on a simple observation: For any local submetry $P:M\to Y$ any limit of $P$-minimal geodesics in $M$ is a  $P$-minimal geodesic,
%in particular, it is a horizontal geodesic.

%Given a fiber $L$ of $P$, any horizontal geodesic starting on $L$ can be extended to a $P$-minimal geodesic of length $\epsilon$
%for some $\epsilon$ bounded on $L$ locally uniformly from below.

%In particular, this implies that for a sequence of points $x_j\in L$ converging to $x$ and for a sequence of unit horizontal vectors $h_j \in T_{x_j} ^{\perp} L$
%any limit point of any    subsequence of $h_j$ is contained in $T_x ^{\perp} L$.  Thus, along a leaf $L$ of the local submetry $P:M\to Y$ the horizontal spaces and, therefore, the vertical spaces,
%vary in a semi-continuous manner.

%In order, to describe what happens for a variation of points not sitting on a single leaf, we will need some notations. We fix a local submetry
%$P:M\to Y$.
  We fix a sequence $x_j \in M$ converging to $x\in M$ and $r=r_x>0$ in the standard local picture around $x$,  Section \ref{subsec: standard}.
  We may assume  $x_j \in B_r(x)$, for all $j$.
  Set $y_j=P(x_j)$ and $y=P(x)$.
  % We  fix  $r=r_x>0$ in the standard local picture around $x$,  Section \ref{sec: onefiber}.
  %We assume that $x_j \in B_r(x)$, for all $j$.
% the standard local picture around $x$.
%In particular

 We set $\Sigma _j = \Sigma _{y_j} Y$ and $\Sigma =\Sigma _yY$.   We denote by $H_j$  and $H$ the set of unit horizontal vectors in $T_{x_j} M$ and $T_xM$ respectively.
By $Q_j $ and $Q$ we denote the differentials $Q_j:=D_{x_j} P$  and $Q=D_xP$ of $P$  and their restrictions $Q_j:H_j \to \Sigma _j$ and $Q:H\to \Sigma$.

After replacing $x_j$ by a subsequence, we  assume that the infinitesimal submetries $Q_j$  converge to a submetry $\hat Q:T_x M\to C(\hat \Sigma)$, where
$\hat \Sigma$ is a Gromov--Hausdorff limit of $\Sigma _j$.   By $\hat H\subset T_xM $ we denote the set of unit horizontal vectors of the infinitesimal submetry $\hat Q$.

We fix some $0<\epsilon <r$ and consider the compact subsets $\Sigma _j ^{\epsilon} \subset \Sigma _j$ %and $\Sigma ^{\epsilon} \subset \Sigma $
of all starting directions
of geodesics of length $\epsilon$.  The preimages $H_j ^{\epsilon}$
of  $\Sigma _j ^{\epsilon}$
%and   $\Sigma ^{\epsilon}$
under the submetries $Q_j$ are exactly the  subsets of  unit vectors at $x_j$,   which are the starting directions of $P$-minimal geodesics of length $\epsilon$.
Choosing a subsequence, we may assume that the subsets $H_j^{\epsilon}$ converge in the Hausdorff topology to a subset $\hat H^{\epsilon}$ of $\hat H$ in   $T_x M$.

All subsequent statements are based on the simple observation that  $P$-minimal geodesics converge to $P$-minimal geodesics. In particular, {such limits are horizontal. Thus,
 $\hat H^{\epsilon}$ is contained in $H$.}

\subsection{Semicontinuity}
The situation is  easily described if $x_j$ vary along the same fiber.  The first statement   of the next Lemma just means that tangent (and normal)
spaces of a set of positive reach vary semicontinuously.  The second statement,
 related  to Question \ref{quest: infinite},   means that along any manifold fiber
the differentials vary continuously.
%   See also Question \ref{quest: infinite}.  \cite{Mendes-Rad-slice} and Question \ref{quest: infinite}.
\begin{lem} \label{lem: openn}
In the notations above, assume in addition that $x_j$ are contained in the leaf $L=L_x$.
Then $\hat H \subset H$. % In other words, horizontal vectors at $x_j$ converge to horizontal vectors at $x$.

If, in addition, $x$ is contained in a $\mathcal C^{1,1}$-submanifold $K$, open in $L$, then $\hat H=H$ and
the submetries $Q_j=D_{x_j} P:T_{x_j} M\to T_y Y$ converge to $Q=D_xP:T_x M\to T_yY$.
\end{lem}

\begin{proof}
Since $x_j \in L_x$, we have $y_j=P(x_j)=P(x)=y$. Thus, $\Sigma _j=\Sigma _{y_j} Y=\Sigma _y Y=\Sigma$.
Hence, by the choice of $\epsilon <r$, $\Sigma _j ^{\epsilon} =\Sigma _j$ and $H_j^{\epsilon}=H_j$  for all $j$.
Therefore, $\hat H^{\epsilon}=\hat H\subset H$, proving the first statement.

If $x_j, x$ are contained in  $K$, a $\mathcal C^{1,1}$-submanifold open in $L$, then  $T_{x_j} L =T_{x_j}K$ are vector spaces
converging to $T_xK$. Therefore,  $T_{x_j} ^\perp L$ converge to $T_{x} ^\perp L$. Due to Proposition \ref{prop: structure2},
 it suffices to prove that $Q_j :H_j\to \Sigma $ converge to $Q:H\to \Sigma$.

For $h_j \in H_j$ converging to $h\in H$, the geodesics $\gamma _{h_j}, \gamma _h :[0,r]\to M$ are $P$-minimal.
Hence, $P\circ \gamma _{h_j}$ are geodesics  starting in $y$ and converging to $P\circ \gamma _h$. Thus, $Q_j (h_j)$ converges to $Q(h)$.
\end{proof}

{ If $x_j$ vary in  different fibers we still have:}

%In the notations above, let $L=P^{-1} (y)$ be the fiber through $x$.

%If  $x_j \in L$ then $H_j^{\epsilon} =H_j$. Thus, we deduce that $\bar H\subset H$ in this case. This  is just the well-known semi-continuity of normal (and tangent ) spaces of a subset of positive reach.

%As the examples presented in the previous subsection shows, this semi-continuity does not always hold. But we have the following statement:

\begin{lem} \label{lem: gensemi}
	In the notations above, the linear span $W_j$ of $H_j^{\epsilon}$ contains
	$H_j$ and the linear span of $\hat H^{\epsilon}$ contains  $\hat H$.
\end{lem}

\begin{proof}
   Let  $m=\dim Y$. By the Bishop--Gromov inequality  in $B_r(y)\subset Y$, we have a uniform lower bound $\mathcal H^{m-1} (\Sigma _j ^{\epsilon})   \geq \delta >0$ for some $\delta >0$ and all $j$ { large enough.
	Due} to the continuity of the Hausdorff measure, \cite{BGP}, we have the same lower bound
	 $\mathcal H^{m-1} (\hat \Sigma ^{\epsilon} ) \geq \delta $.

% =\hat Q (\hat H^{\epsilon}) )\subset  \hat \Sigma$,   since this set is the limit set of the sets $\Sigma _j ^{\epsilon}$.
	
Assume that the linear space $W_j$  of $H_j ^{\epsilon}$ does not contain the convex set $H_j$.
Then  $\mathcal H^{n_j} (H_j ^{\epsilon}) =0$, where
$n_j=\dim (H_j)$. Due to Corollary \ref{cor: zeroset},
it implies  $\mathcal H^{m-1} (\Sigma _j^{\epsilon} ) =0$,
 %\subset \Sigma _j$ is $0$,
  in contradiction to the previous observations.

The same { reasoning shows} that the linear span of $\hat H^{\epsilon}$ contains  $\hat H$.
\end{proof}

As a consequence we deduce a first general analogue of Lemma \ref{lem: openn}:

\begin{cor} \label{lem: semicont}
	In the notations above, assume in addition that $T_xL$ is a vector space.
% hence  that $H$
%	is a round sphere.
 Then $\hat H \subset H$.
 % Thus, the limit of any sequence of horizontal vectors at $x_j$ is a horizontal vector at $x$.
\end{cor}

\begin{proof}
We have $\hat H^{\epsilon} \subset H$. By Lemma \ref{lem: gensemi}, the linear hull of $\hat H$ is contained in the linear hull of $H$.
Since $T_xL$ is a vector space, $H$ is a round sphere, thus the unit sphere in its linear hull. Hence, $\hat H\subset H$.
\end{proof}

If the local  submetry $P$ is \emph{transnormal}, Corollary \ref{lem: semicont}  applies to all
points $x$; see also \cite[Lemma 45]{Mendes-Rad}.

{ Using  Lemma \ref{lem: openn} we prove that foot-point projections from fibers to manifold fibers are open maps:}
\begin{cor} \label{cor: op}
Let  $P:M\to Y$ be a surjective local submetry.
Let $y',y \in Y$ be connected in $Y$ by a unique geodesic $\gamma$ of length $s$.
Assume that the starting direction $v\in \Sigma _{y'} Y$ of the geodesic $\gamma$ has in $\Sigma _{y'}Y$ an antipodal direction.
Set $L=P^{-1} (y)$ and $L':=P^{-1} (y')$.
Let $x\in L$ be such that  the closed ball $\bar B_{3s}(x)$ in $M$ is compact.

Then the foot-point projection
$$\Pi^L :L'\cap B_{2s} (x) \to L$$
is uniquely defined and continuous.  If, in addition,
 $L\cap B_{3s} (x)$ is a $\mathcal C^{1,1}$-submanifold,
then this foot-point projection
%$$\Pi^L :L'\cap B_{2s} (x) \to L$$
is an open map.
\end{cor}

\begin{proof}
 Due to compactness  of  $\bar B_{3s}(x)$ a horizontal lift of $\gamma$ starts
at any  $x'\in L'\cap B_{2s} (x)$. Moreover, it is unique by the assumption on $v$ and Proposition
\ref{prop: directions}.
The map $\Pi^L$ on $L'\cap B_{2s} (x)$ just assigns to $x'$ the endpoint of this $P$-minimal geodesic.  Thus,   the foot-point projection
$$\Pi^L :L'\cap B_{2s} (x) \to L$$
is uniquely defined and  continuous.

The image of this map is contained in $K=L\cap B_{3s} (x)$. Assume now that $K$ is a $\mathcal C^{1,1}$ submanifold.

In order  to prove the openness of $\Pi^L$, set $z=\Pi ^L (x')$ and let $z_j\to z$ be a sequence in $K$.  Consider the starting direction
$h$ of the geodesic $zx'$. Now we apply Lemma \ref{lem: openn} and find unit horizontal directions $h_j \in T_{z_j} M$ converging to $h$
such that $D_{z_j} P(h_j)= D_zP(h)$ is the starting direction of $\gamma$ in $\Sigma _yY$.

 The points $x_j =\exp (s\cdot h_j) $ lie on  $L'$, converge to $x'$ and satisfy $\Pi^L (x_j)=z_j$.
 This finishes the proof.
\end{proof}

\subsection{Continuity  of vertical spaces} We start   with { \blue the  simple} 
%We find an  easy continuity statement as a consequence of  Corollary \ref{lem: semicont}:

 \begin{cor} \label{cor: stable1}
 	Assume that the vertical spaces $T_{x_j} L_{x_j}$ and $T_xL _x$ are vector spaces of the same dimension.
 	%Let $P:M\to Y$ be a local submetry. Let $x_j$ be a sequence of points in $M$ converging to  $x$. Assume that the vertical spaces $V_{x_j}$ and the vertical space $V_x$ are vector spaces of the same dimension.
 	%
 	Then $T_{x_j} L_{x_j}$ converges to $T_xL_x$.
 	%Moreover, there exists a submetry $Q_0: \hat \Sigma \to \Sigma $  such that the submetry
 	%$Q:T_xM\to T_yY =C(\Sigma )$ factorizes as
 	%$$Q=C(Q_0)\circ \hat Q \;, $$
 	%where $\hat Q$ is as before the limit submetry of the differentials $D_{x_j} P :T_{x_j} M\to C(\Sigma _j)$ and
 	%$C(Q_0)$ is the cone over the submetry $Q_0$.
 	%
 	%For $y_j=P(x_j)$ and $y=P(x)$, the spaces  $\Sigma _{y_j} Y$ converge to $\Sigma _y Y$ in the Gromov--Hausdorff sense. The submetries $D_{x_j}P$ converge to $D_xP$.
 \end{cor}

 \begin{proof}
 	The statement is equivalent to the equality $\hat H=H$.  Due to Corollary \ref{lem: semicont},  $\hat H \subset H$. By assumption,
 	$H_j$ and $H$ are round spheres of the same dimension, hence so is $\hat H$. The required equality follows.
 	% o a subsequence of $H_j$ is contained in $H$
 	% Let, as above denote by $H_j$ and $H$ the set of unit vertical vectors in $T_{x_j} M$ and $T_xM$, respectively.   Due to Corollary \ref{lem: semicont},
 	% the  limit of any  subsequence  of $H_j$ is contained in $H$.  Since both spaces, $H$ and the limit, are spheres of the same dimension, they coincide.
 	% Thus, $H_j$ converges to $H$. Hence, $V_{x_j}$ converges to $V_x$.
 \end{proof}

 The next continuity and stability statement is more involved and  more surprising. It is the  key to Theorem \ref{thm: regular}.
 \begin{thm} \label{thm: surprise}
 Let $P:M\to Y$ be a surjective local submetry. Let $y_j$ be a sequence of points in $Y$ converging to $y$.  Assume that the 	spaces of directions $\Sigma _{y_j}Y$ converge in the Gromov--Hausdorff topology to $\Sigma _y Y$.

Then there exists some $\delta >0$ and some  $j_0$ such that, for all  $j>j_0$, the following holds true.  The spaces $\Sigma _{y_j} Y$ and $\Sigma _y Y$ are isometric and any direction in $\Sigma _{y_j} Y$ is a starting direction of a geodesic of length $\delta$.
 \end{thm}

 \begin{proof}
 	Choose $x\in P^{-1} (y)$ and $x_j \in P^{-1} (y_j)$ converging to $x$.
 In the following we use the notation introduced in  Subsection \ref{subsec: notat}.

 	%Consider the map
 	 %  Let $r>0$ be as in the standard local picture around $x$. We may assume that $y_j\in B_{\frac r 3} (y)$ for all $j$.
 	%Fix $\epsilon <\frac r 2$. For any $j$ consider the subset $\Sigma_j ^{\epsilon}$  consisting of all starting directions $v\in \Sigma _{y_j}$ of geodesics $\gamma _v$  of lengths at least $\epsilon$.
 	
 	{  Define $F_j : \Sigma_j ^{\epsilon} \to \Sigma$  by sending $v \in \Sigma _{j}^{\epsilon} \subset \Sigma _{y_j} Y$ to the starting direction $w \in \Sigma _yY =\Sigma$
 	of  the  unique geodesic connecting   $y$ and  $\gamma _v(\epsilon )$.
 	
 	Since $\epsilon <r$  and satisfies   Corollary \ref{cor: posreach},   $F_j (\Sigma _j ^{\epsilon}) \subset  \Sigma _y Y $
 	converge in the Hausdorff topology to $\Sigma _y Y$.}
 	
 	Due to the semicontinuity of angles in Alexandrov spaces, \cite[Section 7.7.4]{AKP}, the maps  $F_j$ converge (after choosing a subsequence), to a $1$-Lipschitz map $F:\hat \Sigma ^{\epsilon} \to \Sigma $, where $\hat \Sigma ^{\epsilon}\subset \hat \Sigma$ is the limit
 	of the subsets $\Sigma _j ^{\epsilon}$ in the Gromov--Hausdorff limit
 	$\hat \Sigma $ of the sequence  $\Sigma_j$.

 	 The map $F$ is a surjective and  $1$-Lipschitz and $\hat \Sigma$ is isometric to $\Sigma$.
   Hence,  $\hat \Sigma^{\epsilon} =\hat \Sigma$ and $F$ is an isometry,     \cite[Section 1.2]{Petruninpar}.
 	
 	In particular, for any $\rho >0$ and all sufficiently large $j$, the set
 	$\Sigma _{j} ^{\epsilon}$ is $\rho$-dense in $\Sigma _j$. We fix some $\rho < \frac \pi 2$.

Consider $L_j =P^{-1} (y_j)$.   For any $z\in L_j\cap B_r (x)$
the preimage  $H_z^{\epsilon}$ of $\Sigma _j ^{\epsilon}$ in the set of unit horizontal vectors $H_z \subset T_z M$ under the submetry $D_zP:H_z \to \Sigma _{y_j}$  is $\rho$-dense in $H_z$.  Moreover, for any vector $h\in H_z^{\epsilon}$ the geodesic $\gamma _h :[0,\epsilon] \to M$ is $P$-minimal, thus, it satisfies
$d(L_j, \gamma _h (\epsilon) ) =\epsilon$.

Applying now \cite[Proposition 1.8, Theorem 1.6]{Ly-conv}, we find some
$\delta >0$ depending only on $\epsilon$ and the curvature bounds of $B_r(x)$ such that $L_j$   has  reach $\geq \delta$ in $B_{\frac r 3} (x)$.

This implies the last statement of the Theorem.

For all $j$, such that $d(y_j,y)<\delta $, the unique geodesic between $y$ and $y_j$ can be extended beyond both endpoints as a  geodesic. Due to \cite{Petruninpar}, this implies that $\Sigma _{y_j}$ and $\Sigma _y$ are isometric.
 	\end{proof}

%\begin{lem}
%Let $U$ be a metric space and  $O$ a manifold.
%Let $f_j :U\to O$ be a sequence of open, continuous maps.  Assume that $f:U\to O$ is continuous and $f_j(x_j)\to f(x)$, for all sequences $x_j$ in $U$ converging to $x$.
%Assume that $f^{-1} (x)$ is compact. Then, for
%\end{lem}

\subsection{The most technical statement}
Now we turn to the final continuity statement.   Its proof  uses a part of Theorem \ref{thm: regular}, which will be  proved later in Section \ref{sec: reg} not relying on the present Subsection.  Another simple ingredient in the proof will be the following variant of the classical theorem of Hurwitz in complex analysis:

\begin{lem} \label{lem: openmap}
 Let $U$  be a locally compact space, let $f_j,f:U\to B$ be continuous and open maps to a Euclidean ball $B$. Assume that for any $x_j \to x$ in $U$ the
 points $f_j (x_j)$ converge to $f(x)$.

  Let  $p\in B$ be such that $C:=f^{-1} (p)$ is non-empty and compact.
  Then, for all $j$ large enough, the preimage  $f_j ^{-1} (p)$ is not empty.
\end{lem}

\begin{proof}
Assume the contrary. Going to a subsequence we assume that $p\notin f_j (U)$, for all $j$.  Consider a compact neighborhood $V$ of $C$ in $U$.
	
Find  $x_j$ in $V$ such that $p_j:=f_j (x_j)$ is the closest point to $p$ in $f_j(V)$,
which exists, since $V$ is compact.  Since $f_j$ is an open map on $U$, the point $x_j$ must be contained in the boundary $\partial V$. Since $f_j$ converges to $f$ on $C$, the points $p_j$ converge to $p$.

 Going to a subsequence we may assume that $x_j$ converges to a point $x\in \partial V$.  Then $f(x)=p$ but $x\notin C$, which is impossible.
 % This  contradiction finishes the proof.
\end{proof}

The statement we are going to prove  is reminiscent of \cite{Petruninpar} and,   in view of \cite{Petruninpar}, could have been expected  for all
 Alexandrov spaces. However, this expectation is wrong in general, as has been shown by a clever $3$-dimensional counterexample by Nina Lebedeva.

 \begin{prop} \label{prop: difficult}
 	Let $P:M\to Y$ be a surjective local submetry. Let $y, z\in Y$ be  connected by a geodesic $\gamma:[a,b]\to Y$. Let $v\in \Sigma _yY$ and $w\in \Sigma _zY$ be  the starting and  ending directions of $\gamma$.
 Assume that there exist $v'\in \Sigma _yY$ and $w'\in \Sigma _zY$ such that
 	$d(v,v')=d(w,w')=\pi$.
 	
 	Then $\Sigma _yY$ and $\Sigma _zY$ are isometric.
 \end{prop}

 \begin{proof} 	
 	Due to \cite{Petruninpar}, all spaces of directions $\Sigma _{\gamma (t)} Y$
 	for $t\in (a,b)$ are pairwise isometric.  Thus, by symmetry, it suffices to prove that $\Sigma _{\gamma (t)} Y$ is isometric to $\Sigma _y Y$ for some (and hence any) $t\in (a,b)$.
 	
 	Choose a Lipschitz submanifold $K\subset L:= P^{-1} (y)$, open in $L$.
 	Let $x \in K$  be arbitrary. Choose $r>0$ as in the standard local picture around $x$. In addition, we may assume that $B_{10r} (x)\cap L\subset K$.
 	
 	Consider a $P$-minimal lift of $\bar \gamma :[a,a+r] \to M$ of $\gamma$ starting in $x$.
 	Choose $x_j=\bar \gamma (t_j)$ for $t_j>a$ converging to $a$.
 	Set $y_j=P(x_j)=\gamma (t_j)$.
 	
 	Now we are in the situation described in Subsection \ref{subsec: notat}.
 Moreover, $\Sigma _j =\Sigma _{y_j} Y$ are pairwise isometric and, therefore, isometric to their limit space $\hat \Sigma$.  We need to prove that
 $\Sigma $ is isometric to $\hat \Sigma$.

 Set $K_j= P^{-1} (y_j) \cap B_r(x)$, which is an open subset in the fiber $P^{-1} (y_j)$.  Due to Corollary \ref{cor: op}, the foot-point projection $\Pi ^L :K_j \to L$ is a continuous, open map $\Pi ^L:K_j \to K$ which sends $x_j$ to $x$.
 Moreover, by our assumption on the starting direction of $\gamma$  and Proposition~\ref{prop: directions} the vector $\gamma'(a)$ has a \emph{unique} horizontal lift at any $z\in K $. Therefore
 the map $\Pi ^L:K_j\to K$ is injective.

   Thus $\Pi^L :K_j \to \Pi^L (K_j)$ is a homeomorphism onto an open subset of $K$.  Hence, $K_j$ is a $\mathcal C^{1,1}$ submanifold of the same dimension as $K$.

 In particular, the assumptions of Corollary \ref{cor: stable1} are satisfied. We deduce, that $T_{x_j} K_j$ converge to $T_xK$ and  that $H_j$ converge to $H$. Thus $\hat H=H$.  Therefore, it suffices to prove that the submetries
 $Q:H\to \Sigma $ and $\hat Q:H\to \hat \Sigma $ have the same fibers.

{ \begin{rem} A word of caution: The convergence of  $H_j$  to $H$ is not (!) sufficient to conclude that $\Sigma _j$ converge to $\Sigma$. The counterexample one should have in mind is given by a submetry with discrete fibers, for instance $\R\to \R/\mathbb Z_2 = [0,\infty )$, or a product of such a submetry with a Riemannian submersion. The remaining part of the proof   excludes a behaviour such as in these examples. At a final instance, the proof relies on Lemma \ref{lem: openmap} and the injectivity  of the projection
	of $K_j$ to $K$.
\end{rem}}

 For any $u\in H$ we denote by $F_u$ respectively $\hat F_u$ the $Q$-fiber, respectively the $\hat Q$-fiber through $u$.
 We already have seen, that there exists a subset $\hat H^{\epsilon} \subset H$ of positive measure, such that $\hat F_u \subset F_u$ for all $u\in  H^{\epsilon}$.
 % which is a union of $\hat Q$-fibers, such that for any $u\in H^{\epsilon}$ the $\hat Q$-fiber through $u$ is contained in the $Q$-fiber through $u$.
%
%For any $u\in H$ we denote by $F_u$ respectively $\hat F_u$ the $Q$-fiber, respectively the $\hat Q$-fiber through $u$.

 	The rest of the proof proceeds in three steps.
 	\begin{enumerate}
 	\item \label{step1} There exists a subset $\tilde H \subset \hat H^{\epsilon}$  of
 	positive measure, such that, for any $u\in \tilde H$,  the fiber $\hat F_u$
 	 is a union of  connected components of $F_u$.
 	
 	\item  \label{step2} For any $u\in \tilde H$ the fiber $\hat F_u$ intersects all
 	 components of  $F_u$.
 	
 	\item  \label{step3} For any $u\in H$ the fibers $F_u$ and $\hat F_u$ coincide.
 	\end{enumerate}
 	
 	Starting with Step 1, we apply Theorem \ref{thm: regular} to $Q$ and $\hat Q$ and find a subset $\tilde H$ of { full} measure in $H^{\epsilon}$, such that for any $u\in \tilde H$ the fibers $F_u$ and $\hat F_u$ are closed submanifolds of the unit sphere $H$, both of the same dimension
 	$\dim (H)-\dim (\Sigma)=\dim(H) -\dim (\hat \Sigma)$.
 	By construction, $\hat F_u$ is contained in $F_u$ for all $u\in \tilde H$.  This implies \eqref{step1}.

  We proceed with step \eqref{step3}, assuming that \eqref{step2} has already been verified.  Combining with \eqref{step1}, we see that  $F_u =\hat F_u$ for all $u\in \tilde H$. 	
 	
 For any $h,h'\in H$ we have $F_h=F_{h'}$ if and only if $d(F_u,h) =d(F_u,h')$ for all $u\in \tilde H$.  This, is due to \eqref{eq: dA}, to the positive measure of
 $Q(\tilde H)$, Corollary \ref{cor: zeroset},
 and to the fact that  distance functions to points in a set of positive measure separate points in any Alexandrov space.

  Similarly, for any $h,h'\in H$ we have $\hat F_h=\hat F_{h'}$ if and only if $d(\hat F_u,h) =d(\hat F_u,h')$ for all $u\in \tilde H$.  These two statements together imply that for any $h\in H$ the fibers $F_h$ and $\hat F_h$ coincide.
  % proving that $\Sigma $ and $\hat \Sigma $ are isometric.

 	It remains to prove \eqref{step2}, { for a possibly smaller subset $\tilde H ^0$ of positive measure of   $\tilde H$ chosen in Step (1).  In order to find this subset, we call a point $p\in B_{3\epsilon} (y) \setminus B_{2\epsilon} (y)$ a \emph{good point}, if $T_p Y = \R^m$ and there }
 	  exists exactly one geodesic connecting $y_j$ and $p$, for all $j$.  Almost all points
 	in $B_{2\epsilon} (y) \setminus B_{\epsilon} (y)$ are good points (in any Alexandrov space).
 Since $T_p Y=\R^m$,  for any $j$ the geodesic $y_jp$ satisfies the assumptions of Corollary \ref{cor: op}.

 For any good point $p$ { and any $j$,  consider  $P$-minimal lifts
 	 $\eta _j ^p$ starting at $x_j$ } of the geodesics  $y_jp$.
   The set of starting directions   of all such lifts $\eta _j ^p$  is a fiber of the submetry
 $Q_j:H_j \to \Sigma _j$.   The limit of these fibers exists and it is a fiber of $\hat Q$.   The union $\tilde H ^1$ of all such fibers of $\hat Q$ is, by construction,  contained in $\hat H^{\epsilon}$. The  argument based on the Bishop--Gromov volume comparison,  that we used to verify that $\hat H^{\epsilon}$ has positive measure in $H$, shows that
 $\tilde H^1$ has positive measure in $H$.
 % (In fact, the complement $\hat H^{\epsilon} \setminus \tilde H$ has measure $0$).

{
We now set $\tilde H^0:=\tilde H \cap \tilde H^1$ and are going to verify \eqref{step2}
for any $u\in \tilde H^0$.}

 By  construction, we have a good point $p\in Y$ such that the following holds true. The direction $u$
 is the starting direction of a $P$-minimal geodesic from $x$ to $L':=P^{-1} (p) { \cap B_{5\epsilon} (x)}$.    The $Q$-fiber $F_u$ through $u$ is  the set of all starting directions of $P$-minimal geodesics from $x$ to $L'$.

 Moreover, { letting $E_j \subset L'$ be the set  of endpoints of $P$-minimal geodesics from $x_j$ to $L'$, the $\hat Q$-fiber $\hat F_u$ is described as follows.  The sets $E_j$ converge in
 the Hausdorff topology to  $E\subset L'$ and $\hat F_u$ is
  the set of starting directions  of $x$ to points in $E$.}
 % of $P$-minimal geodesics from $x$ to $L'$ whose endpoints $E\subset L'$ is the Hausdorff limit of the sequence $E_j$ of set  of endpoints of $P$-minimal geodesics from $x_j$ to $L'$.

  The set $E_j$ is the preimage of $x_j$
 under the open, continuous foot-point projection $\Pi ^{K_j} :L'\to K_j$, Corollary  \ref{cor: op}.

 Set $f= \Pi :L'\to K$ and $f_j:= \Pi \circ \Pi ^{K_j} :L'\to K$.  Note that $f$ and $f_j$ are open maps, by Corollary \ref{cor: op}.  Clearly for any $q_j \to q$ in $L'$ the images $f_j (q_j)$ converge to $f(q)$.

The map $\Pi :K_j \to K$   is an open embedding, hence fibers of $f_j$ and of $\Pi ^{K_j}$ coincide.   Therefore,  the claim that $\hat F_u$ intersects all components of $F_u$ just means that  any connected component of the fiber $f^{-1} (x)$ contains a limit point of a sequence $q_j \in f^{-1} _j(x)$.

The fiber $F_u$, homeomorphic to $f^{-1} (x)$,  has finitely many components. Thus, assuming the contrary, we find a component $C$ of $f^{-1} (x)$, and a small  neighborhood $W$ of $C$ in $L'$ such that $f_j^{-1} (x)$
does not intersect $W$, for all $j$ large enough.  But this contradicts Lemma
  \ref{lem: openmap}.

 This contradiction finishes the proof of Step 2.

  Thus, we have proved that $Q$ and $\hat Q$ have the same fibers. Therefore,  $\Sigma $ is isometric to $\hat \Sigma$. This finishes the proof of the Proposition.
 \end{proof}

\section{Strict convexity}  \label{sec: convex}
In this section we prove that small balls in the base space $Y$ are convex. This somewhat  technical section is not used for other results.

Let $N\subset M$ be a closed $\mathcal C^{1,1}$ submanifold of a manifold $M$ with two sided bounded curvature.   Then the distance function $d_N$ is semiconvex in a  small neighborhood $O$ of $N$,  and the function  $f=d_N^2$ is $\mathcal C^{1,1}$ in $O$,    \cite{Ly-conv}, \cite{KL}.
 We are going to see that $f$ is  strictly convex in directions almost orthogonal to $N$.
%For the sake of simplicity, we deal
%with the squared distance function $f=d_N^2$.

\begin{lem} \label{lem: strict}
Let $N\subset M$ be a $\mathcal C^{1,1}$-submanifold, let $x\in N$ be arbitrary. Set   $f =d_N^2 :M\to \R$ and fix some distance coordinates around $x$.
%Let $f =d_N^2$ denote the squared distance function to $N$.   Let $p$ be a point in $N$ and fix some $\mathcal C^{1,1}$ coordinates around the point $p$.
 Then  there exists $\epsilon >0$ { with the following property.  For any geodesic}
 $\gamma :[0,\epsilon] \to B_{5\epsilon} (x)$, with
 $d(\gamma'(0), T_x^{\perp}N )<\epsilon$,  
 we have 
 %
 %at almost all $t\in [0,\epsilon]$ one has }
%
%$<\gamma '(0), v> \leq \epsilon$  for all unit vectors $v \in T_pN$
%we have
  $$(f\circ \gamma) ''(t)  \geq 1  $$
  { for almost all  $t\in [0,\epsilon]$.}

\end{lem}

\begin{proof}
Fix a sufficiently small $\epsilon >0$ and $U=B_{5\epsilon} (x)$.
   Then, for all
 $z\in B_{\epsilon}  (x) \cap L$,  the function
 $d_z ^2 $ is $\mathcal C^{1,1}$ in $U$ and the restriction to any geodesic $\gamma$ in $U$ satisfies { $||(d_z^2 \circ  \gamma )''(t) - 2|| <\delta$  for almost all  $t$,
 where $\delta $ goes to $0$ as $\epsilon\to 0$.

  Since  $d_N$ is semiconvex in $U$ and $d_N=0$ on $N$, the $\mathcal C^{1,1}$ function $f=d_N^2$ satisfies
 $(f\circ \gamma ) '' \geq -\delta$  almost everywhere on any geodesic $\gamma$
 in $U$, where again $\delta $ goes to $0$ with $\epsilon$.  Together with the lower curvature bound this implies
  $|(f\circ \gamma )''(t)|\le 3$,  for almost all  $t$.}

 %Since the distance function $d_N$ is semiconvex in $U$ and constant $0$ on $N$, the $\mathcal C^{1,1}$ function $f=d_N^2$ satisfies
 %$(f\circ \gamma ) '' \geq -\delta$  almost everywhere on any geodesic $\gamma$
 %in $U$.

	Geodesics in $U$ are uniformly $\mathcal C^{1,1}$, \cite{Yaman}, \cite{Ber-Nik}.
Thus, for any geodesic $\gamma$ in $U$, the assumption	$d(\gamma'(0), T_x^{\perp}N )<\epsilon$
implies $d(\gamma '(t),  T_x^{\perp}N )<C\cdot \epsilon$ for some fixed $C\geq 1$ and all $t$ in the domain of definition of $\gamma$.

Denote by  $H_pf :T_pM\times T_p M\to \R$ the Hessian of $f$ at  a point $p\in O$, whenever it exists.
Since $f$ is  $\mathcal C^{1,1}$ in $U$ the Hessian exists at almost all points in $U$.
Standard measure theoretic arguments  ("on almost all geodesics the Hessian $H_pf$ exists at almost all points")
 show that it suffices to prove the  following  \emph{claim}, for all sufficiently small $\epsilon $:

 For  all
 $p \in U$ at which   the Hessian $H_pf$ exists,
we have   $H_p f(w,w) \geq 1$, for all unit vectors  $w\in T_p M$
with $d(w, T_x^{\perp}N )\leq \epsilon$.
%<w,v> <\epsilon$ for all $v\in T_p N$.

 Consider the projection $q=\Pi ^L  (p)$. Consider  the Lipschitz submanifold $K_p=\exp_q (O_q)$, where $O_q$ is the open ball of radius $5\epsilon$ in  the normal space
 $T_q ^{\perp} N$ around the origin.
 If $\epsilon$ is small enough, then  the (unit sphere in) tangent space  $J_p=T_pK_p\subset T_p M$ to $K_p $ at $p$ is
 $\delta$-close to $T_x^{\perp} N$, where $\delta $ goes to $0$ with $\epsilon$.

  We compare now the $\mathcal C^{1,1}$-functions $d_q ^2$ and $f$ at the point $p$. Note, that $d_q ^2$ and $f$ have the same value and derivative at $x$ and that they coincide on $K_p$.    Applying the Taylor formula at $p$,  we get  $H_pf (w,w)\geq \frac 3 2$,  for any unit direction $w\in J_p$. Since   $||H_pf||\le 3$  %$H_pf$ is bounded from below by an arbitrary small constant $-\delta$,
   we deduce  $H_p f(v,v) \geq 1$, for all unit vectors $v\in T_pM$ which are $\delta$-close to $J_p$, for a sufficiently small $\delta$.
  This proves the claim and the Lemma.
\end{proof}

As a consequence we derive:
\begin{thm} \label{cor: strict}
Let $P:M\to Y$ be a surjective local submetry. Then, for any $y\in Y$, there exists some $r>0$ such that the function $f=d_y^2$ is a strictly convex function
on $B_r (y)$.
\end{thm}

\begin{proof}
Consider a point $x\in L= P^{-1} (y)$ such that  a small neighborhood of $x$ in $L$ is a $\mathcal C^{1,1}$ submanifold.

 For sufficiently small $r>0$ any short geodesic $\gamma$ in $B_r(y)$ can be lifted to a horizontal geodesic $\hat \gamma$ in $B_{2r}(x)$.
Due to Lemma \ref{lem: semicont}, the starting direction of any such geodesic encloses an angle almost equal to $\frac \pi 2$
with any vector in $T_xL$  (considered in some fixed distance coordinates).  Then  	$f\circ \gamma =d_L^2\circ \hat \gamma$, and this function is
{$1$-convex by Lemma \ref{lem: strict}.}	
\end{proof}

This result finishes the proof of  Theorem \ref{cor: quotient}.

%\begin{rem}
%With somewhat more care it is possible to extend the statement of Lemma \ref{lem: strict} and to prove that in Corollary  \ref{cor: strict} also the (non-squared) distance function to
%$y$ is a convex function in $B_r (y)$.
%\end{rem}

\section{Stratification, regular part} \label{sec: reg}
\subsection{Stratification, regular part, boundary}
As before, let $P:M\to Y$ be a surjective local submetry and let $m=\dim Y$. For $0\leq l\leq m$, denote by
$Y^l _+$ the set of points $y\in Y$, at which the tangent cone $T_yY$ splits  $\R^l$ as a direct factor.

For all Alexandrov regions,  the complement $Y\setminus Y^l_+$ has Hausdorff dimension at most $l-1$, \cite{BGP}, see also   \cite{Naber}.
%showing as a special case, that $Y\setminus Y^{l-1}$ has locally finite $(l-1)$-dimensional Hausdorff measure.

In our case, the subsets $Y^l$ defined in Corollary \ref{cor: directions} are exactly
$$Y^l=Y^l_+ \setminus Y^{l+1} _+ \;.$$
Semicontinuity of spaces of directions, \cite{BGP}, stability of tangent spaces along  geodesics, \cite{Petruninpar}, and   Corollary \ref{cor: directions} directly imply:
\begin{lem}  \label{lem: genstrat}
	For any $l$,
	the space $Y^l_+$ is open in $Y$.   The set $Y^l$ is locally closed in  $Y$. The closure $\bar Y^l$ of $Y^l$ is contained in $Y\setminus Y_+ ^{l+1}$.
	
	 For any geodesic $\gamma :[0, a)\to Y$ with $\gamma (0)\in Y^l _+$ we have $\gamma (t)\in  Y^l _+$ for all $t>0$.
\end{lem}

{ The set of regular points is} $$Y_{reg}: = Y^m = Y^m_+ \;.$$
Thus, $y\in Y_{reg}$ if and only if $T_yY=\R^m$.
From Lemma \ref{lem: genstrat} and the density of regular points in any Alexandrov space, \cite{BGP}, we deduce:

\begin{lem} \label{lem: yreg}
Let $P:M\to Y$ be a surjective local submetry.  Then the set $Y_{reg}$ is open and dense in $Y$.
Any geodesic $\gamma:[0,a)\to Y$  which starts at a point in $Y_{reg}$ is completely contained in $Y_{reg}$.
\end{lem}

A point $y$ is contained in $Y^{m-1}$ if and only if  its tangent space $T_yY$ splits off $\R^{m-1}$ but is not isometric to $\R^m$. This happens if and only if $T_yY$ is isometric to the Euclidean halfspace $\R^m_+= \R^{m-1} \times [0,\infty)$.

In particular, $T_yY$ has non-empty boundary at any point $y\in Y^{m-1}$, thus
$Y^{m-1} \subset \partial Y$. On the other hand,
$Y^{m-1}$ is dense in $\partial Y$, as it is the case  in any Alexandrov region.

%Recall that the boundary $\partial Y$ of $Y$ can be defined as the closure of

%\begin{proof}
%The last statement has been proved for general Alexandrov spaces in  \cite{Petruninpar}. The density of $Y_{reg}$ in $Y$ is
%a well-known statement in arbitrary Alexandrov spaces, \cite{BGP}.  For $y\in Y_{reg}$ and $y_j \in Y$ converging to $y$
%the spaces of directions $\Sigma _{y_j}$ converge to $\mathbb S^{m-1}$ in the Gromov--Hausdorff topology, due to the semi-continuity of
%angles, \cite{BGP}.   From    Corollary \ref{cor: directions}, this implies that $\Sigma _{y_j} =\mathbb S^{n-1}$ for all $j$ large enough.
%Thus, $Y_{reg}$ is open.
%\end{proof}

\subsection{Higher regularity of $Y_{reg}$}

From Theorem \ref{thm: surprise} we now deduce:
\begin{cor} \label{cor: uniform}
Let $P:M\to Y$ be a surjective local submetry. For any $y_0\in Y_{reg}$, there exists $r_0=r_0 (y_0) >0$,
such that for any $y\in B_{r_0} (y_0)$ any
$v\in \Sigma _yY$ is the starting direction of a geodesic of length $r_0$.
\end{cor}

%By the lower  semi-continuity of tangent spaces, at any regular point $y\in Y$ the
%stability assumptions  of Proposition \ref{prop: injstable} are fulfilled with
%$A=Y$. Therefore we obtain:

%\begin{cor}
%	Let $P:M\to Y$ be a surjective  local submetry, let $y\in Y$ be a regular point of $Y$. Then there exists some $6\epsilon >0$  such that an any point in $B_{\epsilon} (y)$ the injectivity radius is at least $6\epsilon$.
	
%	In particular, any geodesic $\gamma$ between  points $y_1,y_2$ in
%	$B_{\epsilon} (y)$ can be extended to a geodesic between $z_1$ and $z_2$ with {\red $
%		d(y_i,z_i)=2\epsilon$ for $i=1,2$}.
%\end{cor}

For any $y_0\in Y_{reg}$, we fix  $x_0\in P^{-1} (y_0)$ and $r_0>0$ as in the standard local picture around $x_0$ and satisfying   Corollary \ref{cor: uniform}.
%Moreover, we can assume that $y_0= P(x_0)$ and that $\bar B_{10r_0} (x_0)$ is compact.

 Now we consider  \emph{strainer maps} around $y_0$, \cite{BGP}.
Namely, we choose points $y_1,...,y_m$ at distance $r_0/2$ to $y_0$ such that the incoming directions of geodesics connecting $y_0$ to $y_i$ enclose
pairwise angles $\frac \pi 2$  (or sufficiently close to $\frac \pi 2$).  Then, we consider the map $F:B_s(y_0) \to \R^m$ with a sufficiently small $s =s(r_0)$,
whose coordinates are distance functions $f_i = d_{y_i}$.  The map $F$ is biLipschitz onto an open subset of $\R^m$, with biLipschitz constant arbitrary close to $1$,
\cite{BGP}, provided $r_0$ is sufficiently small. We  show:

\begin{lem} \label{lem: c11}
	Strainer maps $F:B_s(y_0) \to \R^m$ as above define a $\mathcal C^{1,1}$ atlas on $Y_{reg}$.  The distance function on $Y_{reg}$ is defined by a Riemannian metric, which is  Lipschitz continuous with respect to this atlas.
\end{lem}

\begin{proof}
{	Consider $x_0\in P^{-1} (y_0)$, $r_0, s>0$  and a strainer map $F :B_s(y_0) \to \R^m $  with coordinates $f_i =d_{y_i}$ as above. }
	
	Then $f_i \circ P = d_ {L_i}$ on $B_{r_0}(x_0)$, where $L_i=P^{-1} (y_i)$.
	By the choice of $r_0$,   the ball $ B_{r_0} (x_0)$ is contained in the  set  $\mathcal U(L_i)$ on which  $d_{L_i}^2$ is $\mathcal C^{1,1}$.
	
	 % Considering the restriction of $d_{L_i}$ to horizontal lifts of geodesics  in $B_s(y_0)$, we see that   $f_i $ is $\mathcal C^{1,1}$ in
	  %$B_s (y_0)$ in the following sense.
	
	  {  There exists  $C>0$ such that for any geodesic $\gamma$ in $B_s(y_0)$ the restriction $f_i\circ \gamma$ is differentiable everywhere and $(f_i\circ \gamma )' $ is $C$-Lipschitz.  This follows by restricting $d_{L_i}$ to a horizontal lift of $\gamma _i$ to $B_r(x_0)$.}

This implies the existence of   some $C_1 =C_1(C)>0$ with the following property.  For any  geodesic $\gamma:[a,b]\to   B_s(y_0)$ with midpoint $q$
the distance between $F(q)$ and the midpoint in $\R^m$  between the endpoints $F(\gamma (a))$ and $F(\gamma (b))$ is at most $C_1\cdot  \ell ^2(\gamma)$, compare \cite[Lemma 2.1]{Yaman}.

Therefore,
 for any Lipschitz continuous semiconcave function
 $g:B_s(y_0) \to \R$, the composition $g\circ F^{-1}$ is semiconcave.  This implies that any transition map between different strainer maps as above, has coordinates which are semiconvex and semiconcave at the same time, hence they are  of class $\mathcal C^{1,1}$.

This proves that the strainer maps define a $\mathcal C^{1,1}$ atlas on $Y_{reg}$.

  The  distance function on $Y_{reg}$ is described in any distance coordinates  $F:B_s(y_0)\to \R^m$  by  a Riemannian metric $g$ whose coordinates
  are expressed  as algebraic functions of some distance functions to some points and their derivatives, as explained in \cite{P3},  see also \cite{Ambrosio-Bertrand}. Thus,
  the Riemannian metric $g$ is locally Lipschitz continuous.
\end{proof}

  We can now easily prove a local generalization of  Theorem \ref{thm: regular}:
 \begin{thm}
 	Let $P:M\to Y$ be a surjective local submetry. Then  $Y_{reg}$ is open and dense   in $Y$.  Any geodesic $\gamma :[0,a)\to Y$ starting on $Y_{reg}$
 	is contained in $Y_{reg}$.
 	
 	$Y_{reg}$ carries a natural structure of a Riemannian manifold with a Lipschitz continuous Riemannian metric and  the restriction
 	$P:P^{-1} (Y_{reg}) \to Y_{reg}$ is a $\mathcal C^{1,1}$ Riemannian submersion.
 \end{thm}
  \begin{proof}
   $Y_{reg}$  is open, dense and has the stated convexity property, due to  Lemma \ref{lem: yreg}.
   Due to Lemma \ref{lem: c11}, $Y_{reg}$  has a  natural $\mathcal C^{1,1}$ atlas of distance coordinates which makes it into a Riemannian manifold with a Lipschitz
  continuous metric tensor.

  It remains to verify that $P:P^{-1} (Y_{reg})\to Y_{reg}$ is a $\mathcal C^{1,1}$-Riemannian submersion.
    We fix $y_0\in Y_{reg}$ and $x_0\in P^{-1} (y_0)$
     and consider the standard local picture around $x_0$.
  For $y_0=P(x_0)$ consider distance coordinates $F:B_s(y_0) \to \R^m$ around $y_0$. We have already observed, that  the composition
  of $P$ with any coordinate function $f_i$  of $F$ is given in $B_s(x)$  by the distance  to a fiber  of $P$, which is $\mathcal C^{1,1}$, thus $F\circ P$ is $\mathcal C^{1,1}$ in a small ball around $x_0$.  Hence $P$ is $\mathcal C^{1,1}$.

  The differential $D_{x_0} P$  of $P$ at $x_0$ is a linear map and a  submetry between Euclidean spaces. Hence, it is a surjective linear map which preserves the length of any vector orthogonal to the kernel. Thus, $P$ is a Riemannian submersion.
  \end{proof}

 \section{Singular strata}  \label{sec: sing}
\subsection{Main statement}
We now turn to the singular strata and prove:

\begin{thm} \label{thm: sing}
	Let $P:M\to Y$ be a surjective local submetry. Then the set $Y^l$
	is an $l$-dimensional manifold.
	
	For every point $y\in Y^l$
	there exists some $r>0$ with the following properties. The closed ball $\bar B_r(y)$
	is compact and convex in $Y$ and so is the intersection $\bar B_r(y) \cap Y^l$.
	Moreover, for any $y'\in B_r(y)\cap Y^l$  any geodesic starting in $y'$ can be extended to a geodesic of length $r$.

	The set $Y^l$ has a natural $\mathcal C^{1,1}$ atlas, such that the distance  on $Y^l$ is locally defined by a locally Lipschitz
	continuous Riemannian metric.
\end{thm}

\begin{proof}
Fix $y\in Y^l$. Thus, $T_y Y= \R^l \times C(\Sigma _0)$ with
$\diam (\Sigma _0) \leq \frac \pi 2$.

  Let $r>0$ be such that  the  distance function $d_y^2$ is convex on the compact ball  $\bar B_{2r} (y)$, Corollary \ref{cor: strict}. Due to Theorem \ref{thm: surprise},  we may choose $r$ so
   that, for any  $z\in B_r(y)$ such that $\Sigma _{z}Y$ is isometric to $\Sigma _y Y$,   any geodesic starting at $z$ can be extended to a geodesic of length $3r$.
	
	We claim that $z\in B_r(y)$ is contained in $Y^l$ if and only if the starting direction $v\in \Sigma _y Y$ of the geodesic  $\gamma$ connecting $y$ and $z$  lies in the direct factor $\R^l$ of the tangent space $T_yY$.
	
	Indeed, if $v$ is contained in $\R^l$ then it has an antipode in $\Sigma _yY$. Since, $z$ is an inner point of the geodesic $\gamma$, (by the choice of $r$), also the incoming direction  $w \in \Sigma _z Y$ of $\gamma$ has in $\Sigma _z Y$ an antipode.  Due to Proposition \ref{prop: difficult},
	$\Sigma _z Y$ and $\Sigma _yY$ are isometric. In particular, $z\in Y^l$.
	
 { On the other hand, assume
 %The other direction of the claim is valid in all Alexandrov spaces, as the following folklore argument reveals. If
 $v$ is not contained in the $\R^l$-factor of $T_yY$. Then the tangent space $T_v (T_yY)$ has $\R^{l+1}$ as a direct factor.}
	Under the convergence of the rescaled spaces $(\frac 1 j Y,y) \to (T_y Y,0)$
	the points $y_j=\gamma (\frac 1 j)$ converge to the point $v \in T_yY$.
	For all $j$, the spaces of directions $\Sigma _{y_j}Y$ are all isometric to $\Sigma _zY$, due to \cite{Petruninpar}.  By the semicontinuity of spaces of directions, there exists a distance non-decreasing map from  $\Sigma _v (T_yY)$ to $\Sigma _z Y$ (the Gromov--Hausdorff limit of $\Sigma _{\gamma (\frac 1 j)} Y$).   This implies, that $\Sigma _z Y$ has at least $l$ pairs of antipodal points
	at pairwise distance $\geq \frac \pi 2$. Thus, $T_zY$ splits off    $\R^{l+1}$, hence $z\notin Y^l$.

The exponential map $\exp_ y$ defines a homeomorphism
from the $r$-ball around the origin in $T_yY$ to the $r$-ball in $Y$. By above, this homeomorphism restricts to a homeomorphism from the $r$-ball around the origin in the Euclidean factor $\R^l \subset T_yY$  with $B_r(y)\cap Y^l$.  Thus, $Y^l$ is an $l$-dimensional topological manifold.

 We have seen that, for any $z\in B_r(y)\cap Y^l$, the space of directions
 $\Sigma _z Y$ is isometric to $\Sigma _yY$.  By the choice of $r$,  for any $z\in Y^l \cap B_r(y)$, any geodesic starting in $z$ can be extended to a geodesic of length $3r$.

 We now claim that $B_r(y) \cap Y^l$ is a convex subset of $Y$. Consider a geodesic  $\gamma$ connecting two points $z,z'$ in $Y^l \cap B_{r} (y)$. The geodesic $\gamma$ can be extended beyond $z'$, by the choice or $r$. Thus the space of direction $\Sigma _{z'} Y$ is isometric to spaces of directions at all other points on $\gamma$, by \cite{Petruninpar}.  Thus, all points of $\gamma$ are in $Y^l$, proving the claim.

Now we apply the same arguments as in the proof of Lemma \ref{lem: c11}, to see that the distance coordinates in the Alexandrov region $B_r(y)\cap Y^l$ define
a $\mathcal C^{1,1}$-atlas such that the distance is given in these coordinates by a Lipschitz continuous Riemannian metric.
	\end{proof}

The connected components of the sets $Y^l$ are the so-called \emph{primitive extremal subsets}:
\begin{thm} \label{thm: extrloc}
	Let $P:M\to Y$ be a surjective local submetry, with $M$ a Riemannian manifold.  Let $y$ be a point in $Y^l$ and let $E$ be the connected component of $y$ in $Y^l$. Then the closure $\bar E$ is the smallest extremal subset of $Y$ which contains the point $y$.
\end{thm}

\begin{proof}
Clearly, the Riemannian manifold $E$ does not contain proper extremal subsets.
Thus, any extremal subset of $Y$ which contains $y$ must contain $E$. Since any
extremal subset is closed, it must contain the closure $\bar E$.

For any point $z\in  Y^k$, for any $k$, the tangent space $T_zY$ splits as $T_zY=\R^k \times T'$ and this direct factor $\R^k =T_z Y^k$ is an extremal subset of $T_zY$  (since $T'$ is the cone over a space with diameter $\leq \frac \pi 2$).  From this and  \cite[Proposition 1.4]{Pet-Per},
the intersection $B_r(z) \cap Y^k$ is an extremal subset of the Alexandrov region $B_r(z)$ for all sufficiently small $r$.

Due to \cite{Pet-Per}, it remains to prove, that for any natural $k$ and any $z\in \bar E \cap Y^k$, the set $\bar E$ contains
a  small ball $B_r(z)\cap Y^k$.  However, we can choose $r$ (for any fixed point $z\in\bar E \cap Y^k$) as in Theorem \ref{thm: sing}.  Thus, for any $z' \in B_r(z)\cap Y^k$, any geodesic from $z$ can be extended as a geodesic to length $r$.  Thus, for some $y '\in E$ sufficiently close to $z$ the geodesic
$\eta$ from $z'$ to $y'$ can be extended beyond $y'$. By \cite{Petruninpar}, all
points on $\eta$ but $z'$ have the same spaces of directions as $y'$, hence they belong to $E$.  Therefore, $z'\in \bar E$. This finishes the proof.	
\end{proof}

 \subsection{Topological structure} Over any fixed stratum $Y^l$, the local submetry $P$ has a local product structure.
%	 In case of completeness, we obtain a global structure of a fiber bundle over any connected component:
	
	\begin{prop} \label{prop: loctop}
	Let $P:M\to Y$ be a surjective local submetry, with $M$ a Riemannian manifold. Let $y\in Y^l$ and $x\in P^{-1} (y)$ be arbitrary.
	Then there exists a neighborhood $O$  of $x$ in $P^{-1} (Y^l)$ and a homeomorphism
	$$J: P(O) \times (O\cap P^{-1} (y)) \to O$$
	such that $P\circ J$ is the  projection onto the first factor.

	% Let $E$ be a connected component of a stratum $Y^l$ and   Then any point $x\in P^{-1} (E)$ has  a neighborhood $O$ in $P^{-1} (E)$ homeomorphic to $P(O) \times C$, where $C$ is a neighborhood of $x$ in its fiber $L_x=P^{-1} (P(x))$, such that
	%$$P:O =P(O)\times C \to P(O)$$ becomes a projection to the first factor under this homeomorphism.
	
	If $M$ is complete then, for any connected component $E$ of $Y^l$, the restriction $P:P^{-1} (E)\to E$ is a fiber bundle.
	\end{prop}

\begin{proof}
Let $r>0$ as in the standard local picture around $x$ and satisfying Theorem \ref{thm: sing}. Set $C :=B_r(x) \cap P^{-1} (y)$ and $U:=B_r(y)\cap Y^l$. For
any $z \in U$, we have a unique geodesic $\gamma ^z$ from $y$ to $z$.
This geodesic  $\gamma ^z$ is contained in $Y^l$, Theorem \ref{thm: sing}. Moreover, for any $x' \in C$ there exists
a unique horizontal lift $\gamma ^{z,x'}$ of $\gamma ^z$ starting in $x'$.

These lifts define a map $J:U\times C\to B_{2r} (x) \cap P^{-1} (U)$ given as
$$J(z,x'):= \gamma ^{z,x'} (d(z,y)) \;.$$
 The map   $J$ is continuous, injective and the composition $P\circ J$ is just the projection onto the first factor $U$.

In order to see that $J$  is an open map, consider   $x_j \in B_{2r} (x) \cap P^{-1} (U)$ converging to a point
$x_0$  in the image of $J$. Consider  geodesics $\eta _j$ in $Y^l$ from $P(x_j)$ to $y$ and their unique
 unique horizontal lifts  $\gamma _j$
  %lifts of the geodesics $\eta _j$ in $Y^l$ from $P(x_i)$ to $y$
  starting in $x_j$. These horizontal geodesics converge to the unique shortest curve from $x_0$ to $P^{-1} (y)$. Therefore, the  endpoints of $\gamma _j$  converge to a point in $C$.
Since $C$ and $U$ are open,  for large $j$,  the endpoints of $\gamma _j$ are in $C$ and the point $x_j$ in the image of $J$.

If $M$ is complete, then the above construction, works for $C=L_x$, showing that
$P:P^{-1} (U)\to U$ is a trivial fiber bundle. Since being a fibre bundle is a local condition for connected base spaces, the restriction $P:P^{-1} (E)\to E$ is a fiber bundle.
%Thus
%$P^{-1} (U)$ is homeomorphic to $U\times L_x$ and 	
\end{proof}

\subsection{Strata have positive reach}
We are going to prove that the preimage $P^{-1} (Y^l)$ of any stratum is
a subset of positive reach and start with some preliminaries. The first two statements about differentiable manifolds  are probably well-known, but we could not find a reference.

\begin{lem} \label{lem: c11+}
 Let $N$ be a $k$-dimensional topological submanifold of a Riemannian manifold $M$. Assume that, for all $x\in N$, the blow up
	$$T_xN:= \lim _{t\to 0} (\frac 1 t N, x) \subset T_xM$$
	is a well-defined $k$-dimensional linear subspace. Finally, let
	the map $x\to T_xN$ be locally Lipschitz.   Then $N$ is a $\mathcal C^{1,1}$ submanifold of $M$.
\end{lem}

\begin{proof}
The assumptions and the claim are local. Choosing a chart around a point $x$, we may assume that $M=\R^n$ and $T_x N=\R^k \subset \R^n$.

Consider the orthogonal projection $G:\R^n \to \R^k$.   We find a neighborhood $O$ of $x$ in $N$, such that $T_zN$ is close to $\R^k$, for all $z\in O$. In particular,
$D_zG:T_z N\to \R^l$  has biLipschitz constant close to $1$.
Hence, making $O$ smaller if needed, we see that $G:O\to \R^l$
 is a biLipschitz map onto an open subset $O'$ of $\R^l$, \cite[Proposition 1.3]{Lyt-open}.
	
	The inverse $F:=G^{-1}:O'\to  O$ is a  biLipschitz map onto $O$. The map
	$F:O'\to \R^n$ is differentiable at each point $y\in O'$, and the differential is the inverse map of $D_z G$, where $z=F (y)$.
	
	By assumption, we find Lipschitz continuous  functions $b_1,...,b_k :O\to  \R^n$ such that for any  $z\in O$ the vectors $b_j (z)$ define
	a basis of $T_z N$.  Then $a_j (y) :=D_z G (b_j (z))$, with $z=F (y)$,
	is a basis of  Lipschitz continuous vector fields in $O'$. Moreover,
	$D_y F (a_j (y)) =b_j(F(y))$ for all $y\in O'$.
	
	 Now, we can express the canonical vector fields  $e_j$ on $O'$ as  linear combinations of the vector fields $a_i$ with Lipschitz continuous coefficients. Therefore,
	$y\to D_y F (e_j)$ are linear combinations of  vector fields $b_i$ with Lipschitz continuous coefficients.
		Hence, the partial derivatives of $F$ are Lipschitz continuous and the map $F$ is $\mathcal C^{1,1}$.
		% This proves the claim.	
\end{proof}

\begin{lem} \label{lem: lipcont}
	Let
	$b$ be a Lipschitz continuous  vector field on an open subset  $O\subset \R^k$. Let $\phi _t$ denote the local flow of $b$.
	Then, for any $p\in O$, there exists a neighborhood $O'$  and $A>0$  such that  the inequality
$$||(\phi _t(z)-\phi _t(y)) -(z-y)|| \leq A\cdot t\cdot ||z-y||$$
holds true	 for all $y,z\in O'$ and all $t$, with $|t| \leq \frac 1 A$.
\end{lem}

\begin{proof}
% The flow of a Lipschitz vector field $b$ is locally Lipschitz.
	Differentiating we see
	$$\frac d {dt}(\phi _t(z)-\phi _t(y)) -(z-y))= b(\phi _t(z)) -b(\phi_t (y))\; .$$
	The flow $\phi$ of the Lipschitz vector field $b$ is locally Lipschitz.
%	By he local Lipschitz continuity of $b$ and $\phi _t$
	Thus, the right hand side has norm bounded by $A\cdot ||z-y||$ for some $A>0$, all sufficiently small $t$
	and all $z,y$ sufficiently close to $x$.
 %{\red V: This uses that the flow $\phi_t$ is Lipschitz. Is that obvious? It can be proved of course in the same way one proves that flow of a semiconcave function is Lipschitz but you seem to suggest it's even easier than that?}   
 The claim now follows by integration.
	\end{proof}

Now we can prove:
{
\begin{thm} \label{thm: posreachstr}
	Let $M$ be a Riemannian manifold and let $P:M\to Y$ be a surjective local submetry. Then, for any stratum $Y^l \subset Y$, the preimage $P^{-1} (Y^l)$
	has positive reach in $M$.
\end{thm}

 \begin{proof}
The claim is local. Thus, we may fix some $x\in P^{-1} (Y^l)$ and  may replace $M$ by an arbitrary small neighborhood  $O$ of  $x$.  Set $y=P(x)$ and $L=P^{-1} (y)$.

 For any $x' \in L$, we   consider standard local pictures around $x$ and $x'$. We observe that $P^{-1} (Y^l)$ has positive reach in a neighborhood of $x$ if and only if it has positive reach in a neighborhood of $x'$ (and this happens if and only if  $Y^l$ has positive reach around $y$ in $Y$).
Thus, we may replace $x$ by any other point $x'$  in $L$. Therefore, we can assume that a neighborhood of $x' $ in $L$ is a $\mathcal C^{1,1}$ submanifold.

Restricting to a  neighborhood around such  $x=x'$, we can apply
 Proposition \ref{prop: loctop} and assume    that $N:=P^{-1} (Y^l)$ is a topological manifold  of dimension $k=l+e$, where
$e$ is the dimension of $L$.

 Since fibers converge to fibers under
convergence of submetries, Lemma \ref{lem: stable}, we see as in Corollary \ref{cor: differential}, that $T_p N \subset T_p M$ exists for all
$p\in N$ and equals $(D_pP)^{-1} (T_p Y^l)$.	 Due to Proposition \ref{prop: directions}, $T_p N$ is a direct product $T_pN= U_p \times V_p$ of a vector space $U_p=\R^l$, the set of  horizontal vectors in $T_pN$ and $V_p =\R^e$, the tangent space to the fiber $L_p$.

%{\red By \cite{Bangert, KL}, to establish that $N$ has positive reach it's enough to verify that $N$ is  a $C^{1,1}$ submanifold of $M$.}

%Due to Proposition \ref{cor: fiber}
% {\red V: We need a statement that a $C^{1,1}$ submanifold has positive reach. Proposition \ref{cor: fiber} assumes positive reach. }and 

%Since $\mathcal C^{1,1}$ submanifolds have positive reach,
Due to Lemma \ref{lem: c11+}, 
we  need to prove that $U_p$
and $V_p$ depend in a locally Lipschitz way on $p \in N$,
 since $\mathcal C^{1,1}$ submanifolds have positive reach, see Subsection \ref{subsec: posreach}.

We fix a small $r>0$   as in Theorem \ref{thm: sing}  and  fix points $y_1,...,y_l$ in $B_{\frac r 3} (y)   \cap Y^l$, such that the distance functions $d_{y_i}$ define  a distance coordinate  chart around $y$ in $Y^l$.
Set $L_i:=P^{-1} (y_i)$.  Then, the distance functions  $f_i:=d_{L_i} =d_{y_i}\circ P$ are $\mathcal C^{1,1}$ in a small ball $O_0$ around $x$.
Moreover, the gradients $\nabla _p {f_i}$ build  at every point  $p\in O_0 \cap N$
a basis of the vector space $U_p$, the horizontal part of $T_pN$.
This shows that the assignment $p\to U_p$ is locally Lipschitz continuous in $O_0 \cap N$.

In order to see that also the distribution of tangent spaces to fibers $p\to V_p$ is locally Lipschitz continuous on $O_0\cap N$, we proceed as follows. The gradient flows $\phi ^i$ of $-f_i$, thus the flows of the Lipschitz continuous vector fields  $-\nabla f_i$ preserve the subset $N\cap O_o$. Moreover, this flow preserves the leaves of $P$, Lemma \ref{lem: liftflow}.  From Lemma \ref{lem: lipcont}, we see that along the flow lines of the flows  $\phi^i$, the tangent spaces to leaves $V_{\phi^i_t (p)}$ depend in a Lipschitz manner on the time $t$.

For $p,q\in O_0 \cap N$, we find some $t_i, 1\leq i \leq l$, such that $\sum |t_i | \leq C\cdot d(p,q)$, for some (universal) constant $C$, such that the composition of $\phi _{t_i }^i$ sends the fiber of $p$ to the fiber of $q$, as we  verify by looking at the projection of the flows to $Y^l$.

Thus, we find $p'$ in the fiber $L_p$ through $p$, such that $d(q,p') \leq 2C\cdot d(q,p)$ and such that $V_q$ and $V_{p'}$ are at distance $C'\cdot d(q,p)$ for some constant $C'$.

All fibers $L_p$, for $p\in O_0\cap N$ have uniformly positive reach and therefore, they are locally  uniformly $\mathcal C^{1,1}$, \cite{Ly-conv}.

This shows that $U_p'$ and $U_p$ are at distance bounded by a constant times $d(p,p')$.  Thus, the distance between $V_p$ and $V_q$ is bounded by a constant times $d(p,q)$, for all $p,q\in  N\cap O_0$.

This finishes the proof of the theorem.	
 \end{proof}

Note that a combination of Theorem \ref{thm: posreachstr} and Proposition \ref{prop: loctop} finishes the proof of Theorem \ref{thm: strata}.
}

\section{Manifold fibers} \label{sec: transnorm}
\subsection{Existence of long geodesics} \label{subsec: quasigeod}
 The following result can be localized and   generalized, see the subsequent Remark \ref{rem: new} and compare  \cite[Theorem 1.6]{LT}.
 {   In the proof below we apply the machinery of \cite{KLP}, to conclude
 that the geodesic flow on $Y_{reg}$  preserves the Liouville measure. It is possible that this can be seen in a more direct way.}

 %For our present means the following is sufficient:
% and In the latter part of this section we describe the quasigeodesic flow on the base $Y$ of any locally surjective submetry $P:M\to Y$.  However, the only result which will be used beyond the discussion of quasigeodesics is the following Proposition, which we therefore, state and prove separately.
%We derive the following result from \cite{KLP}.  The result can in fact be derived by more elementary means verifying that the geodesic flow in the Riemannian manifold $Y_{reg}$ preserves the Liouville measure.

\begin{prop} \label{prop-liuv}
	Let  $M$ be a Riemannian manifold and $P:M\to Y$ a submetry. Assume that $Y$ is compact and has no boundary. Then the union of bi-infinite local geodesics $\gamma :\R \to Y_{reg}$ is dense in $Y$.
%	{\red V: I would state it more strongly as: at almost every point  $y\in Y_{reg}$ almost every direction $v\in \Sigma_yY$ is the starting direction of a biinfinite local geodesic $\gamma :\R \to Y_{reg}$  }
\end{prop}

 \begin{proof}
 	The space $Y$ is an Alexandrov space without boundary. The set $Y_{reg}$ of all regular points in $Y$ is a $\mathcal C^{1,1}$ manifold with a Lipschitz continuous Riemannian metric, Theorem \ref{thm: regular}.
 	
 	 Due to \cite[Theorem 1.6]{KLP},  almost every unit direction $v$
 	 at almost every point in $Y_{reg}$ is the starting direction of a bi-infinite local geodesic $\gamma _v:\R\to Y$ completely contained in $Y_{reg}$ provided the so-called \emph{metric-measure boundary} of $Y$ vanishes.

 	Now  \cite[Theorem 1.7, Lemma 6.4]{KLP} imply that $Y$ has vanishing metric measure boundary.  More precisely, following the notations of \cite{KLP},
 	one needs to verify that any limiting measure in the space of signed Borel measures $\mathcal M(Y)$ on $Y$  of the sequence $$\nu:=\lim_{r_j \to 0}  \frac {\mathcal  V_{r_j} } {r_j}$$
 	is the $0$ measure.  { The signed measures $\mathcal V_{r}$ appearing in the formula is the average  deviation measure from the  Euclidean volume growth:
 		$$ \mathcal V_{r}  (A) = \int _A (1 - \frac {\mathcal H^m (B_r(y))} {\omega _m \cdot r^m}) \; d\mathcal H^m (y) \;,$$
 		where $A$ is a Borel subset of $Y$ and $\omega _m$ the volume of unit ball in $\R^m$. }

 	Due to \cite[Theorem 1.7 (3)]{KLP}, any such measure $\nu$ is a Radon measure
 	concentrated on the set of regular points $Y_{reg}$, since $Y$ has no boundary.  Moreover, due to \cite[Theorem 1.7 (1)]{KLP}, the measure is absolutely singular with respect to the Hausdorff  measure  on $Y_{reg}$.
 	
 	Since the distance on $Y_{reg}$ is  obtained  from a Lipschitz continuous Riemannian metric,   the \emph{minimal metric derivative measure} $\mathcal N$, defined in \cite[Section 6.3]{KLP},   (essentially, just a bound on the derivatives of the coordinates of the Riemannian metric), is absolutely continuous with respect to the Hausdorff measure. Finally, due to
 	\cite[Lemma 6.4]{KLP}, the measure $\nu$ is absolutely continuous with respect to  $\mathcal N$.  These facts together imply the vanishing of $\nu$ and therefore the result. 	
 	\end{proof}
{
\begin{rem} \label{rem: new}
 Localizing the above argument and   using \cite[Section 3.6]{KLP} the following can be  shown.  For any surjective local submetry $P:M\to Y$  and almost every unit tangent vector $v\in TY_{reg}$ there exists is one maximal quasigeodesic $\gamma _v :(-a_v,b_v)\to Y$ starting in the direction of $v$.
 This quasigeodesic $\gamma _v:(-a_v,b_v) \to Y$ is completely contained in
 $Y^m \cup Y^{m-1}$, intersects $Y^{m-1}$ in a discrete set of times. Outside these intersection points $\gamma _v$ is a local geodesic.	 Finally, the local quasigeodesic flow  preserves the Liouville measure.
\end{rem}
}

\subsection{Non-manifold fibers and  the boundary} The following Lemma is formulated for general Alexandrov spaces.  The existence of many infinite local geodesics assumed on $Y$ is conjecturally satisfied for all
boundaryless Alexandrov spaces, \cite{PP-quasi}, \cite{KLP}. For some cases it has been verified in \cite{KLP}; for bases of submetries of manifolds it has been shown in Proposition \ref{prop-liuv}.

\begin{lem}  \label{cor: boundary}
	Let $P:X\to Y$ be a submetry between Alexandrov spaces of curvature $\geq 1$.   Assume
	that for any  $y\in Y$ there exist  local geodesics $\gamma _n:[0,\infty )\to Y$ such that $\gamma _n(0)$ converge to $y$.
		 Then $\partial X=\emptyset$.
\end{lem}

\begin{proof}
	Assume the contrary and choose some $x \in X\setminus \partial X$.  By assumption, we find  a  local geodesic $\gamma :[0,\infty)\to Y$, such that
	$$d(\gamma (0), P(x)) < d(x, \partial X)\;.$$
	Consider a horizontal lift $\bar \gamma :[0,\infty ) \to X$ of $\gamma$  such that $d(\bar \gamma (0) ,x) =d(\gamma (0) , P(x))$.  Thus $\bar \gamma (0) \notin \partial X$.

	  The lift $\bar \gamma$ is a  local geodesic,
	 since $\gamma$ is a local geodesic.  Since $\partial X$ is an extremal subset of $X$,  any geodesic  meeting $\partial X$ in an inner point of the geodesic must be  contained in $\partial X$.  Thus, $\bar \gamma$ cannot intersect $\partial X$.

	The distance function to $\partial X$ is strictly concave on $X$, \cite{P2}. More precisely, \cite[Theorem 1.1]{AB-convex} shows that
	$$f:=\sin \circ d _{\partial X} \circ \bar \gamma  :[0,\infty)\to \R$$
	satisfies in the weak sense
	$$f''(t) +f(t) \leq 0 \;.$$
But such a positive function $f$ can be defined at most on an interval of length $\pi$.
	This  contradiction finishes the proof.
\end{proof}

We are going to prove the following local version of Theorem \ref{thm: boundary}.

\begin{thm}
	Let $P:M\to Y$ be a  local submetry.  If a fiber $L=P^{-1} (y)$ is not a
	$\mathcal C^{1,1}$-submanifold of $M$  then $y \in \partial Y$.
\end{thm}

\begin{proof}
Since $L$ cannot be open in $M$, the point $y$ cannot be isolated.
	
 Due to Theorem \ref{thm: halfspace} we  find  a submetry $Q:\mathbb S^k_+ \to \Sigma _y$ for a hemisphere $\mathbb S^k_+$ for some $k\ge 0$.  Thus, $\Sigma _y$ is connected. If $\dim (\Sigma _y )=0$ then $\Sigma _y$ must be a point. {Hence,  $T_yY =[0,\infty)$ and $y\in \partial Y$.

 Assume  $\dim (\Sigma _y Y) \geq 1$.  Since $\partial \mathbb S^k_{+} \neq \emptyset$ and $\mathbb S^k _+$,  we deduce that $\partial (\Sigma _yY) \neq \emptyset$
 from Lemma \ref{cor: boundary} and Proposition \ref{prop-liuv}.
 % that $\partial (\Sigma _yY) \neq \emptyset$,
 %since $\partial \mathbb S^k_{+} \neq \emptyset$ and $\mathbb S^k _+$ is positively curved.}
 % the space $\Sigma _y$ must have non-empty boundary as well, if $\dim (\Sigma _y) \geq 1$, as we deduce from Lemma \ref{cor: boundary} and Proposition \ref{prop-liuv}.
   Thus, $y\in \partial Y$.}
\end{proof}

 \subsection{Transnormal submetries}
 As in the introduction,
 { a (local) submetry $P:M\to Y$ satisfying the equivalent conditions of the next proposition will be called \emph{transnormal}.}
 %
 %
 % \emph{transnormal}  if, for    any $P$-horizontal unit vector  $h$, the local geodesic $\gamma _h$ in $M$ is a $P$-horizontal curve, for all times of its existence. {\red V: this is the proper definition but in the intro transnormal submetries are defined as having manifold fibers because that's easier to state. I don't really like this.}
 % We have the following characterization:

 \begin{prop}\label{prop-subnormal}
 Let $P:M\to Y$ be a local submetry. Then the following are equivalent:
 \begin{enumerate}
 \item\label{fiber-man} All fibers of $P$ are topological manifolds.
 \item \label{fiber-c1-1}All fibers of $P$ are $\mathcal C^{1,1}$ submanifolds.
 \item \label{p-transnormal} For any  $P$-horizontal unit vector  $h$, the local geodesic $\gamma _h$ in $M$ is a $P$-horizontal curve, for all times of its existence.
 \item \label{dp-transnormal} For any $x\in M$, the differential $D_xP:T_xM\to T_{P(x)}Y$  satisfies  condition \eqref{p-transnormal}.
% \footnote{\red V: I changed the phrasing because it was circular.}
 %is a transnormal submetry
 \end{enumerate}

 \end{prop}

\begin{proof}
The equivalence of \eqref{fiber-man}  and \eqref{fiber-c1-1} follows from
Proposition \ref{cor: fiber}, since  all fibers of $P$ have positive reach in $M$, by
Theorem \ref{thm: leaf}.
% and the properties of sets of positive reach.

Assume \eqref{fiber-c1-1} and let $\gamma _h :[0,a] \to M$ be a local geodesic starting in a horizontal direction.  By Proposition \ref{prop: horvector},
there exists some $0<r \leq a$ such that $\gamma _h :[0,r) \to M$  is horizontal, and we can choose $r$ to be maximal with this property.
Applying Proposition \ref{prop: horvector} at $\gamma _h (r)$, we deduce that the incoming direction $-\gamma _h '(r) $ is horizontal in $T_{\gamma _h (r)} M$.

Due to \eqref{fiber-c1-1},  the horizontal  space at $\gamma _h(r)$ is a vector space. Therefore,  the direction $\gamma _h '(r)$ is horizontal as well. Thus,
for a small $\epsilon >0$ also the restriction of $\gamma _h:[r,r+\epsilon) \to M$  is horizontal. If $r<a$ we obtain a contradiction to the maximality of $r$.
Hence, \eqref{fiber-c1-1}  implies \eqref{p-transnormal}.

Assume \eqref{p-transnormal} and suppose a fiber $L=P^{-1} (y)$  not be a  submanifold.  Then for some
 $x\in L$ the tangent space $T_xL$  is not a vector space, Proposition \ref{cor: fiber}.  Thus,  the horizontal space $T_x ^{\perp} L$ is not a vector space and we find a unit vector $h\in T_x^{\perp} L$ such that $-h$ is not horizontal.
Thus, for a small $\epsilon >0$, the geodesic $\gamma _h :[0,\epsilon] \to M$ is horizontal, while $\gamma _{-h}$ is not horizontal.
Therefore, setting $w=-\gamma _h '(t)$ for some $\epsilon >t>0$, we find a horizontal vector, such that the geodesic in the direction of this vector does not stay horizontal for all times. This contradiction shows that
   \eqref{p-transnormal} implies  \eqref{fiber-c1-1}.

Moreover, the argument above also implies that the submetry
$D_xP:T_xM\to T_y Y$  does not satisfy  \eqref{p-transnormal} %is not transnormal 
 in this case, hence \eqref{dp-transnormal} implies  \eqref{fiber-c1-1}.

It remains to prove that   \eqref{p-transnormal} implies \eqref{dp-transnormal}. Thus, assume that $P$  satisfies \eqref{p-transnormal}but $D_xP$   does not  satisfy \eqref{p-transnormal} for some $x\in M$.  Find $v,w \in T_xM$ and  a point $u$ on the segment  $[v,w] \in T_xM$  such that the segment $[v,u]$ is horizontal for $D_xP$ but the segment $[u,w]$ is not  horizontal.

 Choosing  $v$ closer to $u$ we may assume  that $[v,u]$ is a
 unique   $D_xP$-minimal geodesic from $v$ to the fiber of $D_xP$ through $u$.
   We find sequences $p_i$ in $M$, such that  under the convergence $(\frac 1 i M,x) \to (T_xM,0)$
 the sequences $p_i$ converges to $v$.  By lifting appropriate geodesics
 horizontally, we find a sequence of $P$-minimal geodesics $\gamma _i$ from $p_i$ to some point $q_i$, such that  under the convergence $(\frac 1 i M,x) \to (T_xM,0)$ the sequence $\gamma _i$ converge  to a $D_xP$-minimal geodesic starting in $v$ and going to the  $D_xP$-fiber  through $u$.  By uniqueness, $\gamma _i$ must converge to the segment $[v,u]$.

 We extend $P$-minimal geodesics $\gamma_i $ to  geodesics $\hat \gamma _i$  in $M$ by some fixed length $r>0$  beyond $p_i$ and $q_i$.  Under the convergence to $T_xM$, these geodesics converge to the line $\hat \gamma$ extending the segment $[v,u]$.

 By assumption $\hat \gamma _i$ is horizontal. By Proposition ~\ref{prop: image} the images $P\circ \hat \gamma _i$ are quasigeodesics in $Y$. By construction, they converge  to the curve $P\circ \hat \gamma$ in $T_yY$. But under a non-collapsed convergence, a limit of quasi-geodesics is a quasi-geodesic, \cite{Petrunin-semi}.
 Thus, $P\circ \hat \gamma$ is parametrized by arclength. Therefore, the line $\hat \gamma$ is horizontal in contrast to our assumption.

This contradiction finishes the proof of the final implication.
\end{proof}

	 The proof of the last implication  above shows the following:
	 \begin{cor} \label{cor: transstable}
	 	Let $M_i \to M$ and $Y_i\to Y$ be  convergent   sequences
	 	 in the  pointed Gromov--Hausdorff topology
	 	 of  Riemannian manifolds
	 	(with locally uniform bounds on curvature)  and Alexandrov spaces, respectively. Let  $P_i:M_i\to Y_i$ be a sequence of  transnormal submetries converging to a submetry $P: M\to Y$.
	 	If the the convergence
	 	$Y_i\to Y$ is non-collapsing then the submetry $P$ is transnormal.
\end{cor}

\subsection{Equifocality}
The following result has appeared in a slightly more special form in  \cite{Mendes-Rad}.  We only formulate and prove here  global version of the result, which is known in the theory of singular Riemannian foliations under the misleading name of \emph{equifocality}.

\begin{prop} \label{prop: equifoc}
	Let $P_{i}:M_i\to Y, i=1,2$ be transnormal submetries with the same base space. If $\gamma _{i} :[0,a) \to M_i$ are horizontal local geodesics such that $P\circ \gamma _i$ coincide on
	$[0,\epsilon)$, for some $\epsilon >0$ then $P\circ \gamma _1=P\circ \gamma _2$ on $[0,a)$.
	
	 Let $Q_{i}:\mathbb S^{n_i} \to Z$ be  transnormal submetries to the same base space $Z$. If $Q_1 (v_1)=Q_2(v_2)$ the $Q_1(-v_1)=Q_2(-v_2)$.	
\end{prop}

\begin{proof}
%We set $n=\dim (M)$ in the first and
Let $n_i=\dim M_i$ in the first statement and let
$n= \max \{n_1,n_2\}$ in both statements. We will prove both statements simultaneously by induction on $n$.
%\footnote{\red V: I changed the wording here because in the first statement we also don't need to assume that the dimensions are equal.}
The case $n=1$  is left to the reader.

Assume that both claims are know in dimension $n-1$.  Then to prove the first claim, we find a maximal $b \leq a$ such that $P_i\circ \gamma _i$ coincide on
$[0,b)$. Assume $b<a$ and    consider $x_i=\gamma _i (b)$.  Then $P_1(x_1)=P_2(x_2)$.  Denote this point by $y$. Let  $H_i$ be the set of unit $P_i$-horizontal vectors at $x_i$, which are unit spheres.

The differentials  $Q_i:=D_{x_i} {P_i} :H_i \to Z:=\Sigma_y Y $ are transnormal submetries which send the incoming directions $v_i$ of $\gamma _i$ at $x_i$ to the same vector in $\Sigma _yY$.  By the inductive assumption, also $Q_i (-v_i)$ coincide. Thus, for small $r>0$  the $P_i$-minimal geodesics $\gamma_i:[b, b+r) \to M_i$ are sent to  geodesics in $Y$ starting in $y$ in the same direction. Thus $P\circ \gamma _i$ coincide on $[b,b+r)$ in contradiction to the maximality of $b$.  This proves the first statement.

To prove the second, we take an arbitrary $Q_1$-horizontal direction $w_1$ at $v_1$ and the geodesic
$\gamma _1 :[0,\pi] \to \mathbb S^{n_1}$ starting in this direction.  Consider then a $Q_2$-horizontal lift $w_2\in T_{v_2} \mathbb S^{n_2}$ of the direction $D_{v_1} Q_1 (w_1)$. Let $\gamma _2:[0,\pi]\to \mathbb S^{n_2}$ be the geodesic starting in the direction $w_2$.

    Since $Q_i$ are  transnormal the geodesic $\gamma _i$ is $Q_i$-horizontal.
   By construction their  $Q_i$-images coincide initially.
  By the already proved first  statement, $Q_1\circ \gamma _1 =Q_2\circ \gamma _2$ on $[0,\pi]$.  At the time $\pi$ we obtain $Q_1(-v_1) =Q_2(-v_2)$.
\end{proof}

%{\red The above result implies that the base space of a transnormal
%	submetry $P:M\to Y$ of a complete Riemannian manifold $M$ admits
%	a \emph{continuous quasigeodesic flow} in the following sense.

%Let $Y$ be a complete Alexandrov region. Let $\mathcal R$ be the set of all pairs of points $(y,z)\in Y\times Y$ which are connected by a unique geodesic $\gamma _{y,z}$. Note, that the set $\mathcal R$ has full measure in $Y\times Y$, in particular, it is everywhere dense.
%Denote by $Q(Y)$ the set of quasigeodesics $\gamma :\R \to Y$ with the topology of pointwise (equivalently, of uniform) convergence.

%We say that $Y$ admits a continuous quasigeodesic flow if there exists a
%continuous map $G:\mathcal R \to Q(Y)$ such that the quasigeodesic $G((y,z))$ assigned to the pair $(y,z)$ coincides with the geodesic $\gamma _{y,z}$ on the interval $[0, d(y,z)]$.
 	
%Note, that for any $(y,z) \in \mathcal R$ there exists an extension of $\gamma _{(y,z)}$ to an bi-infinite quasigeodesic, \cite{PP-quasi}. Conjecturally, for almost all $(y,z)\in \mathcal R$ there exist exactly 
%one such extension, see \cite{KLP}.  This conjecture would imply, that if a continuous map $G$ exists it is uniquely defined. 

%Now we can easily show:

%	\begin{cor}
%		Let $M$ be a complete Riemannian manifold and let   $P:M\to Y$ be  a  transormal subemtry.  Denote by $\mathcal R$ the set 
%		of pair $(y,z) \in Y\times Y$ which are connected by a unique geodesic.
%	\end{cor}

%}
{
\subsection{Topological structure of   projections between close fibers}
We are going to describe the topology of foot-point projections onto manifold fibers. For the sake of simplicity we only  state a global result in the case of transnormal submetries and for  connected fiber.
For disconnected fibers, the claim remains true for all connected components.

 %leaving the extension to general  local submetries to interested readers.

The proof of the next theorem heavily relies on deep results in geometric topology, characterizing  fiber bundles.  We refer to    \cite[Theorem 4.5, Theorem 4.8]{lyt-nagano-top-reg} for a more detailed discussion of these results.

%between two close fibers of a submetry.

%Let $M$ be a Riemannian manifold and $P:M\to Y$ be a local  submetry, let $L$ be a fiber $P^{-1} (y)$.  Then we find an open neighborhood $U$ of $L$, such that the foot point projection 
%$\Pi ^L :U\to L$ is well-defined and locally Lipschitz continuous.  We are going to describe the topology of the restriction of $\Pi$ to another leaf $L'$.

%The subsequent result is  local in nature and can be formulated for local submetries, however we restrict to the global case, for the sake of simplicity.

\begin{thm} \label{thm-tubular}
Let $M$ be a complete Riemannian manifold, let $P\co M\to Y$ be a transnormal submetry.  Let $L\subset M$ be a  connected leaf of $P$.
%which is a connected topological manifold.

Then there exists $r>0$ such that the foot point projection $\Pi^L\co U=B_r(L)\to L$ is a fiber bundle. Moreover, for any fiber $L'\subset U$ the restriction $\Pi^L|_{L'}\co L'\to L$ is
% a Hurewicz fibration.
%This restriction is
 a fiber bundle as well.
%  if the submetry $P$ is transnormal.
%Let $M$ be a Riemannian manifold and $P:M\to Y$ be a local  submetry,		
\end{thm}

 \begin{proof}
		Let $r$ be smaller than the constant provided by Theorem~\ref{cor: quotient}.
		Then the normal exponential map  $Exp^N$ gives a homeomorphism between the $r$-neighborhood of the zero section of the normal bundle to $L$ and $U$ which commutes with the foot-point projections.
		This proves the first part of the theorem.

	%	The proof of the second part of the theorem  relies on several deep and highly nontrivial topological results.

		Let $L'$ be another fiber contained in $U$. Due to Corollary
	 ~\ref{cor: op} the map $f:=\Pi^L|_{L'}\co L'\to L$ is open.
	 Set $y:=P(L), z:=P(L')$ and $v_0\in \Sigma _y Y$  be the starting direction of the unique geodesic from $y$ to $z$.  Consider 
	 the point $v\in T_yY$ lying in direction $v_0$ at distance $d(y,z) $ from the origin $0_y \in T_yY$ (thus, $\exp _y(v)=z$).
	 
	  Under the homeomorphism given by the
	 normal exponential map $Exp ^N$ the fibers $f^{-1 }(x)$  of $f:L'\to L$ are sent to fibers $F_x:=D_xP ^{-1} (v)$, where $D_xP :H_x \to T_yY$ is the restriction of the differential of the submetry to the horizontal part.  
	 
	 By Proposition \ref{prop-subnormal}, the fibers of $D_xP$ and, therefore, the fibers of $f$ are topological manifolds.   Moreover,
	 the point $v$ has positive injectivity radius in $T_yY$, by Theorem \ref{cor: quotient}. Thus, the fibers $F_x=D_xP ^{-1} (v)$ all have positive reach $s>0$, independent of $x$.  Therefore, there exists 
	 some $\epsilon >0$ independent of $x$, such that any  ball of radius $\delta <\epsilon$ in $F_x$ is contractible, \cite{Federer}, \cite{Ly-conv}.
	 
	 In other words, the fibers of the map $f$ are compact topological manifolds which are \emph{locally uniformly contractible}. Now a combination of results  \cite[Theorem 1]{ungar-hur}, \cite{Dyer-Hamstrom}, \cite[Theorems 1.1-1.4]{Ferry-duality}, \cite[Theorem 2]{Raymond-hur} implies that $f$ is a fiber bundle, see \cite[Section 4]{lyt-nagano-top-reg} for a detailed discussion.
	\end{proof}

%Let $M$ be a Riemannian manifold and $P:M\to Y$ be a local  submetry, let $L$ be a fiber $P^{-1} (y)$.  Then we find an open neighborhood $U$ of $L$, such that the foot point projection 
%$\Pi :U\to L$ is well-defined and locally Lipschitz continuous.
%The subsequent result 

%Assume now that $L$ is a topological manifold. Hence, $L$ is a $\mathcal C^{1,1}$ manifold.  Then, the normal exponential map $\exp ^N: \nu(N )\to M$ from the normal bundle $\nu (N)$  is a Lipschitz map and, $U$ is the image $U=\exp ^N (O)$, where $O$ is an open neighborhood of the $0$-section in $\nu (N)$ and $\exp^N:O\to U$ is biLipschitz. 

%We can restrict $O$ and thus $U$ by a smaller neighborhood and assume that, for any $x\in L$, the intersection  $O_x:==\cap T_x^{\perp} L$ 
%is a ball $B_{r_x} (0_x) \subset T_x ^{\perp}L$, where $r_x>0$ depends continuously on $x$.    Note that $r_x$ can be  
}

{ \subsection{A comment on the factorization theorem}
The following result has been proved in \cite{Lysubm}.
\begin{thm}
	Let $P:X\to Y$ be a submetry between Alexandrov spaces. Then
	the connected components of fibers of $P$ define an equidistant decomposition of $X$. Thus, $P$ admits a canonical factorization
	$P=P_1\circ P_0$, where the submetry $P_0:X\to Y_0$ has connected fibers and the submetry $P_1:Y_0\to Y$ has discrete fibers.
\end{thm}
The proof of this result is  technical and remains technical if the Alexandrov space $X$ is replaced by a complete Riemannian manifold $M$.  However, if the submetry is transnormal, the proof is much easier.

Indeed, in this case the connected components $L_p^0$  of fibers $L_p$ of $P$ define a \emph{transnormal decomposition} of $M$ in the sense of \cite{Molino}:  thus $L_p^0$ are $\mathcal C^{1,1}$ submanifolds and any
local geodesic in $M$ which starts orthogonal to any  leaf of the decomposition remains orthogonal to all leaves it intersects.  But such a transnormal decomposition is equidistant, as one readily verifies by the first variation formula.}

%{\blue
%\subsection{Whitneys conditions}
%Let again $P:M\to Y$ be a  surjective local submetry. For any point
%$x\in M$ consider the point $y=P(x)$,  the unique stratum $Y^l$ of $Y$ containing $y$ and the connected component $E$ of $Y^l$ through $y$.

%We set $G^x :=P^{-1} (E)$ and
% denote by $G^x_0$ the connected component of $P^{-1} (E)$ through $x$.
%The set $G^x$
% and $G^x_0$
% have positive reach in $M$, by Theorem \ref{thm: posreachstr}.  If $P$ is transnormal then each of the subsets $G^x _0$ are connected  $\mathcal C ^{1,1}$ submanifolds of $M$.

%The family of all extremal subsets is  locally finite in any Alexandrov space, \cite{Pet-Per}, thus   also in the Alexandrov region $Y$. Therefore,  the families of subsets $G^x$ and  $G^x_0$ are locally finite in $M$.

%As has been shown in the proof of Theorem \ref{thm: extrloc}, the closure
%$\bar E$ of $E$ is a union of $E$ and some components $E_i$ of some other strata $Y^{l'}$ with $l'<l$. Since $P$ is a local submetry, we see that
%the closure of $G^x=P^{-1} (E)$ is a union of the preimages $P^{-1} (E_i)$. Thus,
%$$\bar G^x = \cup _{z\in \bar G^x} G^z \;.$$
%}

%\bibliographystyle{my-amsalpha}
%\bibliography{submet}

\providecommand{\bysame}{\leavevmode\hbox to3em{\hrulefill}\thinspace}
\providecommand{\MR}{\relax\ifhmode\unskip\space\fi MR }
\providecommand{\MRhref}[2]{%
  \href{http://www.ams.org/mathscinet-getitem?mr=#1}{#2}
}
\providecommand{\href}[2]{#2}

\end{document}